\documentclass[twoside,11pt]{article}
\usepackage{blindtext}
\usepackage[preprint]{jmlr2e}
\usepackage{bbding}
\usepackage{microtype}
\usepackage{graphicx}
\usepackage{subfigure}
\usepackage{booktabs} % for professional tables
\usepackage{microtype}
\usepackage{graphicx}
\usepackage{subfigure}
\usepackage{booktabs} 
\usepackage{amsmath}
\usepackage{amssymb}
\usepackage{mathtools}
\usepackage{paralist}
\usepackage{multirow}
\RequirePackage{algorithm}
\RequirePackage{algorithmic}

\newtheorem{ass}{Assumption}
\newcommand{\Norm}[1]{\left\|#1\right\|}

\def \cite {\citep}
\def \E {\mathbb{E}}
\def \R {\mathbb{R}}

\def \h {\mathbf{h}}

\def \r {\mathbf{r}}

\def \u {\mathbf{u}}
\def \v {\mathbf{v}}
\def \s {\mathbf{s}}
\def \w {\mathbf{w}}

\def \F {\mathcal{F}}
\def \G {\mathcal{G}}
\def \z {\mathbf{z}}
\def \x {\mathbf{x}}
\def \y {\mathbf{y}}

\def \LL {L} %{\mathcal{L}}

\def \X {\mathcal{X}}
\usepackage{lastpage}
\jmlrheading{1}{2025}{1-\pageref{LastPage}}{XX/XX}{XX/XX}{21-0000}{Jiang, Yang, Yang, Wang, Wan, Li, and Zhang}

\ShortHeadings{Projection-free Multi-level Algorithms for Stochastic Constrained Optimization}{Jiang, Yang, Yang, Wang, Wan, Li, and Zhang}
\firstpageno{1}

\begin{document}

\title{Projection-Free Multi-level Algorithms for Stochastic Constrained Compositional Optimization}

\author{\name Wei Jiang\textsuperscript{\rm 1,2} \email {jiang\_wei@njust.edu.cn} 
       \AND
       \name Sifan Yang\textsuperscript{\rm 2,3} \email yangsf@lamda.nju.edu.cn 
       \AND
       \name Wenhao Yang\textsuperscript{\rm 2,3} \email yangwh@lamda.nju.edu.cn \AND
       \name Yibo Wang\textsuperscript{\rm 2,3} \email wangyb@lamda.nju.edu.cn
        \AND
       \name Yuanyu Wan\textsuperscript{\rm 4} \email wanyy@zju.edu.cn 
       \AND
       %\\ State Key Laboratory of Blockchain and Data Security, Zhejiang University, Hangzhou, China
       \name Zechao Li\textsuperscript{\rm 1} \email zechao.li@njust.edu.cn
    \AND
    \name Lijun Zhang\textsuperscript{\rm 2,3} \email zhanglj@lamda.nju.edu.cn \AND
       \addr
       \textsuperscript{\rm 1}School of Computer Science and Engineering, Nanjing University of Science and Technology, China\\
        \textsuperscript{\rm 2}State Key Laboratory of Novel Software Technology, Nanjing University, China\\
        \textsuperscript{\rm 3}School of Artificial Intelligence, Nanjing University, China\\
        \textsuperscript{\rm 4}School of Software Technology, Zhejiang University, China
       }

\editor{My editor}

\maketitle

\begin{abstract}%   <- trailing '%' for backward compatibility of .sty file
This paper studies projection-free algorithms for stochastic constrained multi-level compositional optimization. 
In this context, the objective function is a nested composition of several smooth functions, and the decision set is closed and convex. 
Since projection onto the constraint set can be computationally expensive, we develop projection-free methods that rely on linear minimization oracles.
For non-convex objectives, we propose variance-reduced projection-free algorithms and establish complexity guarantees under both the Frank-Wolfe gap and the gradient mapping criteria.
We also develop momentum-based methods that achieve convergence guarantees under weaker smoothness assumptions.
Additionally, by using a stage-wise design, we derive a parameter-free variant that preserves the same complexities for the Frank-Wolfe gap.
Such a design can be further used to develop algorithms for convex and strongly convex functions whose rates match those of single-level projection-free counterparts. 
Finally, we consider finite-sum problems and derive complexities for non-convex, convex, and strongly convex objectives.
Numerical experiments across multiple tasks demonstrate the effectiveness of the proposed methods.
\end{abstract}


\section{Introduction}
In this paper, we investigate projection-free algorithms for stochastic constrained multi-level compositional optimization of the form
\begin{equation} \label{prob:1}
\min_{\x\in \X} F(\x) = f_K \circ f_{K-1} \circ \cdots \circ f_1(\x),
\end{equation}
where $\X$ is a closed convex set. We assume that each function $f_i$ and its gradient are accessible only through unbiased stochastic estimations, denoted by $f_i(\cdot;\xi)$ and $\nabla f_i(\cdot;\xi)$, such that $\E_\xi \left[ f_i(\cdot;\xi) \right] = f_i(\cdot)$ and $\E_\xi \left[\nabla f_i(\cdot;\xi) \right] = \nabla f_i(\cdot)$,
where $\xi$ denotes a sample drawn from the oracle. Problem~(\ref{prob:1}) arises in a range of machine learning applications, including reinforcement learning~\citep{Dann2014PolicyEW}, government planning~\citep{Bruno2016RiskNA}, risk management~\citep{ Dentcheva2015StatisticalEO}, model-agnostic meta-learning~\citep{Ji2020MultiStepMM}, robust learning~\citep{li2021tilted}, risk-averse portfolio optimization~\citep{Shapiro2009LecturesOS}, and graph neural network training~\citep{balasubramanian2020stochastic}.

Although stochastic multi-level optimization has been investigated extensively in recent years~\citep{Yang2019MultilevelSG,balasubramanian2020stochastic,Zhang2021MultiLevelCS,chen2021solving,jiang2022optimal,pmlr-v235-zhang24aw,zhang2025on,TPAMI:2025:Jiang}, most existing work focuses on unconstrained settings (i.e., $\X=\R^d$). 
In many practical problems, such as risk-averse portfolio optimization, the decision set is constrained~(e.g., $\x$ lies in a simplex), and standard approaches enforce feasibility via projection operations.
However, projections are usually computationally intensive, motivating projection-free methods that replace the projection operation~(a convex optimization problem) with multiple steps of efficient linear minimization~\citep{NEURIPS2022_7e16384b}.

Projection-free methods typically rely on two oracles: (i) a Stochastic First-order Oracle~(SFO), which takes $\x$ and returns $(f(\x;\xi), \nabla f(\x;\xi))$ for a sample $\xi$;
and (ii) a Linear Minimization Oracle~(LMO), which takes a direction $\mathbf{d}$ and outputs $\arg \min_{\x \in \X} \langle \x, \mathbf{d} \rangle$.
The performance of projection-free algorithms is typically measured by the number of SFO and LMO calls required to achieve an acceptable solution.
For non-convex objectives, such a solution $\x$ is usually defined by the Frank-Wolfe gap~\citep{lacostejulien2016convergence}:
\begin{align}\label{FW}
    \F(\x) \coloneqq \max_{\hat{\x} \in \X}\langle\hat{\x}-\x,-\nabla F(\x)\rangle \leq \epsilon,
\end{align}
where $\epsilon$ is a small value. More recently, the gradient mapping criterion~\citep{pmlr-v80-qu18a} has also been used:
\begin{align}\label{GM}
    \G(\x) \coloneqq \Norm{\beta\left(\x-\Pi_{\mathcal{X}}\left(\x-\frac{1}{\beta}\nabla F(\x)\right)\right)}^2\leq \epsilon,
\end{align}
where $\Pi_{\X}$ denotes projection onto $\X$ and $\beta$ is a positive constant.
Notably, when $\X=\R^d$, this reduces to the standard stationary point $\Norm{\nabla F(\x)}^2\leq \epsilon$ for unconstrained stochastic optimization.
For convex or strongly convex objectives, the optimal gap criterion is used instead, expressed as 
\begin{align}\label{OP}
    F(\x)-\min_{\hat{\x} \in \X} F(\hat{\x})\leq \epsilon,
\end{align}
which measures the difference between the objective and the optimal value.

The current projection-free approach for stochastic multi-level compositional optimization, LiNASA+ICG~\citep{NEURIPS2022_7e16384b},
combines the linearized NASA~\citep{Ghadimi2020AST} estimator with inexact conditional gradient~\citep{Balasubramanian2018ZerothOrderNS} and provides guarantees under the gradient mapping criterion.
It finds an acceptable point with $\mathcal{O}(\epsilon^{-2})$ SFO calls and $\mathcal{O}(\epsilon^{-3})$ LMO calls.
However, it has several limitations: (i) its SFO complexity does not match the state-of-the-art $\mathcal{O}(\epsilon^{-1.5})$ rate for stochastic unconstrained problems; (ii) its analysis focuses solely on gradient mapping and does not provide results under the more widely used Frank-Wolfe gap; and (iii) the method is restricted to non-convex objectives, and it is unclear how to obtain improved rates for convex/strongly convex objectives or for finite-sum structures.

To address these issues, we propose new algorithms that leverage the variance-reduction estimator~\citep{cutkosky2019momentum} to track both inner function values and the overall gradient more accurately.
Coupled with a tailored Frank-Wolfe procedure~\citep{pmlr-v28-jaggi13}, this approach improves the rate for gradient mapping and also enables the analysis for Frank-Wolfe gap.
Moreover, by using larger batch sizes, we can reduce the number of iterations and thus improve LMO complexity while maintaining the same SFO rate.
We further develop a stage-wise algorithm with warm starts and establish guarantees for convex and strongly convex objectives. Compared with previous methods, this paper makes the following contributions:
\begin{compactenum}
\item We provide the first theoretical guarantees for the Frank-Wolfe gap in the multi-level setting (Theorem~\ref{thm:main}). The resulting rates match those of single-level projection-free methods~\citep{Zhang2019OneSS,pmlr-v97-yurtsever19b}.
\item For gradient mapping (Theorem~\ref{thm:2}), we obtain an improved SFO complexity of $\mathcal{O}(\epsilon^{-1.5})$ and an LMO complexity of $\mathcal{O}(\epsilon^{-2.5})$. The SFO complexity matches the lower bound~\citep{Arjevani2019LowerBF}, and the LMO complexity can be further reduced to $\mathcal{O}(\epsilon^{-2})$ using larger batch sizes. 
\item We establish complexity results for convex (Theorem~\ref{thm:3}) and strongly convex objectives (Theorem~\ref{thm:4}), deriving optimal SFO rates for these settings, which have not been studied in prior projection-free multi-level optimization.
\end{compactenum}
A preliminary version of this work~\citep{ICML:2024:Jiang} was presented at the 41st International Conference on Machine Learning (ICML 2024). This journal version substantially extends the conference paper in the following aspects:
\begin{compactenum}
\item We develop a fully parameter-free algorithm with a stage-wise design and a refined analysis. This method does not require problem-dependent parameters to set hyperparameters and preserves the same complexities under the Frank-Wolfe gap measure~(\textbf{Theorem~\ref{thm:main+}}).
\item We propose a simpler momentum-based projection-free algorithm that attains the rates under a weaker smoothness assumption (instead of average smoothness), improving the LMO complexity compared to prior results~(\textbf{Theorem~\ref{thm:main-}, \ref{thm:2-}}).
\item We further investigate the finite-sum setting, where we can periodically compute exact function values and gradients and reuse them to refine our estimators. This yields improved complexities for non-convex (\textbf{Theorem~\ref{thm:fsnc},~\ref{thm:fsnc2}}), convex (\textbf{Theorem~\ref{thm:3+}}), and strongly convex functions (\textbf{Theorem~\ref{thm:4+}}). 
\item We provide additional experimental results, including the newly proposed methods and more datasets, to demonstrate empirical effectiveness more comprehensively.
\end{compactenum}
We compare our theoretical results with existing methods under the Frank-Wolfe gap in Table~\ref{t1}, under the gradient mapping criterion in Table~\ref{t1-}, and under the optimal gap in Table~\ref{t2}, which also present differences from the conference version.
\begin{table*}[t]
\caption{Summary of results for projection-free methods under the Frank-Wolfe gap. We compare our methods with 1-SFW~\citep{Zhang2019OneSS} and SPIDER-FW~\citep{pmlr-v97-yurtsever19b}. *PF denotes parameter-free methods.}
\vspace{-0.2in}
\label{t1}
\begin{center}
\resizebox{0.9\columnwidth}{!}{%
\begin{tabular}{lcccccc}
\toprule
Method &  Assumptions & Level & Batch size & SFO & LMO & PF* \\
\midrule
1-SFW     &\multirow{2}{*}{Average Smooth} & \multirow{2}{*}{$1$} &  $1$ &$\mathcal{O}\left(\epsilon^{-3}\right)$ & $\mathcal{O}\left(\epsilon^{-3}\right)$ & \XSolidBrush \\
SPIDER-FW   &  &   &  $\Omega\left(\epsilon^{-1}\right)$ &$\mathcal{O}\left(\epsilon^{-3}\right)$ & $\mathcal{O}\left(\epsilon^{-2}\right)$ & \XSolidBrush \\
\midrule
\multirow{2}{*}{\textbf{Theorem~\ref{thm:main}}}  &\multirow{2}{*}{Average Smooth} &   \multirow{2}{*}{$K$} & $1$ &$\mathcal{O}\left(\epsilon^{-3}\right)$ & $\mathcal{O}\left(\epsilon^{-3}\right)$ & \XSolidBrush\\
& & & $\Omega\left(\epsilon^{-1}\right)$ &$\mathcal{O}\left(\epsilon^{-3}\right)$ & $\mathcal{O}\left(\epsilon^{-2}\right)$ & \XSolidBrush\\
\midrule
\multirow{2}{*}{\textbf{Theorem~\ref{thm:main+}~(new)}}  &\multirow{2}{*}{Average Smooth} &   \multirow{2}{*}{$K$} & $1$ &$\mathcal{O}\left(\epsilon^{-3}\right)$ & $\mathcal{O}\left(\epsilon^{-3}\right)$ & \Checkmark\\
& & & $\Omega\left(\epsilon^{-1}\right)$ &$\mathcal{O}\left(\epsilon^{-3}\right)$ & $\mathcal{O}\left(\epsilon^{-2}\right)$ & \Checkmark\\
\midrule
\textbf{Theorem~\ref{thm:main-}~\textbf{(new)}}  & \multirow{1}{*}{Smooth} &   \multirow{1}{*}{$K$} & $\Omega\left(\epsilon^{-2}\right)$ &$\mathcal{O}\left(\epsilon^{-4}\right)$ & $\mathcal{O}\left(\epsilon^{-2}\right)$ &  \Checkmark \\
\midrule
\multirow{2}{*}{\textbf{Theorem~\ref{thm:fsnc}~(new)}}  &\multirow{1}{*}{Finite-sum} &   \multirow{2}{*}{$K$} & $1$ &$\mathcal{O}\left(\sqrt{m}\epsilon^{-2}\right)$ & $\mathcal{O}\left(\sqrt{m}\epsilon^{-2}\right)$ & \Checkmark\\
&\multirow{1}{*}{+ Smooth} & & $\Omega\left(\sqrt{m}\right)$ &$\mathcal{O}\left(\sqrt{m}\epsilon^{-2}\right)$ & $\mathcal{O}\left(\epsilon^{-2}\right)$ & \Checkmark\\
\bottomrule
\end{tabular}}
\end{center}
\end{table*}
\begin{table*}[t]
\caption{Summary of results for projection-free algorithms under the gradient mapping. We compare our methods with NCGS~\citep{pmlr-v80-qu18a}, SGD+ICG~\citep{Balasubramanian2018ZerothOrderNS}, and LiNASA+ICG~\citep{NEURIPS2022_7e16384b}}
\vspace{-0.2in}
\label{t1-}
\begin{center}
\resizebox{0.9\columnwidth}{!}{%
\begin{tabular}{lccccr}
\toprule
Method &  Assumptions & Level & Batch size & SFO & LMO \\
\midrule
\multirow{2}{*}{\textbf{Theorem~\ref{thm:2}}}  & \multirow{2}{*}{Average Smooth} & \multirow{2}{*}{$K$} & $1$ &$\mathcal{O}\left(\epsilon^{-1.5}\right)$ & $\mathcal{O}\left(\epsilon^{-2.5}\right)$\\
   & &  & $\Omega\left(\epsilon^{-0.5}\right)$ &$\mathcal{O}\left(\epsilon^{-1.5}\right)$ & $\mathcal{O}\left(\epsilon^{-2}\right)$\\
   \midrule
NCGS  &\multirow{2}{*}{Smooth} &  \multirow{2}{*}{$1$} &$\Omega\left(\epsilon^{-1}\right)$ &$\mathcal{O}\left(\epsilon^{-2}\right)$ & $\mathcal{O}\left(\epsilon^{-2}\right)$\\
SGD+ICG   & &  & $\Omega\left(\epsilon^{-1}\right)$ &$\mathcal{O}\left(\epsilon^{-2}\right)$ & $\mathcal{O}\left(\epsilon^{-2}\right)$\\
   \midrule
LiNASA+ICG   & \multirow{2}{*}{Smooth} &  \multirow{2}{*}{$K$} & 1 &$\mathcal{O}\left(\epsilon^{-2}\right)$ & $\mathcal{O}\left(\epsilon^{-3}\right)$ \\
\textbf{Theorem~\ref{thm:2-}~(new)}  & & & $\Omega\left(\epsilon^{-1}\right)$ &$\mathcal{O}\left(\epsilon^{-2}\right)$ & $\mathcal{O}\left(\epsilon^{-2}\right)$\\\midrule
\multirow{2}{*}{\textbf{Theorem~\ref{thm:fsnc2}~(new)}}  & \multirow{1}{*}{Finite-sum} & \multirow{2}{*}{$K$} & $1$ &$\mathcal{O}\left(\sqrt{m}\epsilon^{-1}\right)$ & $\mathcal{O}\left(\sqrt{m}\epsilon^{-2}\right)$\\
   & \multirow{1}{*}{+Smooth}&  & $\Omega\left(\sqrt{m}\right)$ &$\mathcal{O}\left(\sqrt{m}\epsilon^{-1}\right)$ & $\mathcal{O}\left(\epsilon^{-2}\right)$\\
\bottomrule
\end{tabular}}
\end{center}
\end{table*}
%##################%##################%##################%##################
\begin{table*}[t]
\caption{Summary of results for projection-free algorithms under optimal gap. Here, CVX represents convex functions and SC stands for $\lambda$-strongly convex functions. We compare our methods with 1-SFW~\citep{Zhang2019OneSS}, SPIDER-FW~\citep{pmlr-v97-yurtsever19b}, and SCGS~\citep{doi:10.1137/140992382}.}
\vspace{-0.2in}
\label{t2}
\begin{center}
\resizebox{0.9\columnwidth}{!}{%
\begin{tabular}{lccccr}
\toprule
Method &  Assumptions & Level & Batch size & SFO & LMO \\
\midrule
1-SFW &\multirow{2}{*}{CVX} & \multirow{2}{*}{$1$} & $1$ &$\mathcal{O}\left(\epsilon^{-2}\right)$ & $\mathcal{O}\left(\epsilon^{-2}\right)$ \\
SPIDER-FW & &  & $\Omega\left(\epsilon^{-1}\right)$ &$\mathcal{O}\left(\epsilon^{-2}\right)$ & $\mathcal{O}\left(\epsilon^{-1}\right)$ \\
\midrule
\multirow{2}{*}{\textbf{Theorem~\ref{thm:3}}} &\multirow{2}{*}{CVX} &   \multirow{2}{*}{$K$} & $1$ &$\mathcal{O}\left(\epsilon^{-2}\right)$ & $\mathcal{O}\left(\epsilon^{-2}\right)$ \\
 & &    & $\Omega\left(\epsilon^{-1}\right)$ &$\mathcal{O}\left(\epsilon^{-2}\right)$ & $\mathcal{O}\left(\epsilon^{-1}\right)$ \\
 \midrule
\multirow{2}{*}{\textbf{Theorem~\ref{thm:3+}~(new)}} &\multirow{1}{*}{Finite-sum} &   \multirow{2}{*}{$K$} & $1$ &$\mathcal{O}\left(\sqrt{m}\epsilon^{-1}\right)$ & $\mathcal{O}\left(\sqrt{m}\epsilon^{-1}\right)$ \\
 & \multirow{1}{*}{+CVX} &    & $\Omega\left(\sqrt{m}\right)$ &$\mathcal{O}\left(\sqrt{m}\epsilon^{-1}\right)$ & $\mathcal{O}\left(\epsilon^{-1}\right)$ \\
\midrule
\midrule
SCGS
& \multirow{1}{*}{SC} &   $1$ & $\Omega\left(\epsilon^{-1}\right)$ &$\mathcal{O}\left(\lambda^{-1}\epsilon^{-1}\right)$ & $\mathcal{O}\left(\epsilon^{-1}\right)$ \\
\midrule
\multirow{2}{*}{\textbf{Theorem~\ref{thm:4}}} & \multirow{2}{*}{SC} &   \multirow{2}{*}{$K$} & $1$ &$\mathcal{O}\left(\lambda^{-1}\epsilon^{-1}\right)$ & $\mathcal{O}\left(\epsilon^{-2}\right)$ \\
& &  & $\Omega\left(\epsilon^{-1}\right)$ &$\mathcal{O}\left(\lambda^{-1}\epsilon^{-1}\right)$ & $\mathcal{O}\left(\epsilon^{-1}\right)$ \\
\midrule
\multirow{2}{*}{\textbf{Theorem~\ref{thm:4+}~(new)}} & \multirow{1}{*}{Finite-sum} &   \multirow{2}{*}{$K$} & $1$ &$\widetilde{\mathcal{O}}\left(\sqrt{m}\lambda^{-1}\right)$ & $\widetilde{\mathcal{O}}\left(\sqrt{m}\epsilon^{-1}\right)$ \\
&\multirow{1}{*}{+SC} &  & $\Omega\left(\sqrt{m}/\lambda\right)$ &$\widetilde{\mathcal{O}}\left(\sqrt{m}\lambda^{-1}\right)$ & $\widetilde{\mathcal{O}}\left(\lambda\epsilon^{-1}\right)$ \\
\bottomrule
\end{tabular}}
\end{center}
\end{table*}
\section{Related Work}
This section reviews prior work on stochastic multi-level compositional optimization and stochastic projection-free methods.
\subsection{Stochastic Multi-Level Compositional Optimization}
Stochastic compositional optimization has been studied extensively, with most work focusing on two-level settings~\citep{wang2017stochastic,DBLP:journals/jmlr/WangLF17,Zhang2019ASC,Ghadimi2020AST,chen2021solving,qi2021online,jiang2022multiblocksingleprobe,ICML:2023:Jiang,yu2024efficient,11202712}. 
Multi-level compositional optimization was first investigated by \citet{Yang2019MultilevelSG}. Building on multi-timescale stochastic approximation~\cite{wang2017stochastic}, they proposed a multi-level stochastic gradient method with sample complexity $\mathcal{O}\left(1 / \epsilon^{(7+K)/2}\right)$ for $K$-level problems. This rate improves to $\mathcal{O}\left(1 / \epsilon^{(3+K)/4}\right)$, when the function is strongly convex.
Motivated by the NASA estimator~\citep{Ghadimi2020AST}, \citet{balasubramanian2020stochastic} used a linearized averaging estimator for tracking inner function values, obtaining $\mathcal{O}\left(1 / \epsilon^{4}\right)$ sample complexity for non-convex objectives.
This rate was also obtained in a concurrent work~\citep{chen2021solving} by employing variance reduction to evaluate the function value.

Later, \citet{Zhang2021MultiLevelCS} applied nested variance reduction to approximate gradients and improved the sample complexity to the optimal $\mathcal{O}(\epsilon^{-3})$. 
However, their method requires a large and increasing batch size of $\mathcal{O}(\epsilon^{-1})$.
To address this issue, \citet{jiang2022optimal} proposed SMVR, which attains the same optimal rate with a constant batch size, and further achieves $\mathcal{O}(\epsilon^{-2})$ for convex objectives and $\mathcal{O}\left((\lambda\epsilon)^{-1}\right)$ for $\lambda$-strongly convex objectives.
More recently, \citet{gao2023stochastic} studied decentralized stochastic multi-level optimization and established level-independent convergence in decentralized settings.
Despite these advances, the above methods apply only to unconstrained problems.

\subsection{Stochastic Projection-Free Algorithms}
The most classical projection-free method,  Frank-Wolfe algorithm~\cite{Frank1956AnAF}, was initially developed for smooth convex optimization  with polyhedral domains and later extended to general compact convex sets by~\citet{pmlr-v28-jaggi13}. 
In the stochastic setting, \citet{hazan2012} proposed a projection-free method for online smooth convex optimization, and \citet{pmlr-v48-hazana16} incorporated variance reduction into stochastic Frank-Wolfe.
Inspired by accelerated gradient methods~\citep{Nesterov1983AMF}, \citet{doi:10.1137/140992382} introduced stochastic conditional gradient sliding (SCGS), achieving SFO complexity of $\mathcal{O}(\lambda^{-1}\epsilon^{-1})$ and LMO complexity of $\mathcal{O}(\epsilon^{-1})$ for smooth $\lambda$-strongly convex optimization.
Besides that, projection-free methods have also been widely studied in online convex optimization~\cite{Hazan20,ICML:2020:Wan,JMLR:2022:Wan,Zak_SC22,Garber23}.

For non-convex objectives, \citet{Reddi2016StochasticFM} proposed SVFW, which achieves SFO complexity $\mathcal{O}(\epsilon^{-10/3})$ and LMO complexity $\mathcal{O}(\epsilon^{-2})$ under the Frank-Wolfe gap.
Using the SPIDER variance-reduction technique~\citep{Fang2018SPIDERNN}, \citet{pmlr-v97-yurtsever19b} developed SPIDER-FW and improved the SFO complexity to $\mathcal{O}(\epsilon^{-3})$, at the cost of a large batch size $\mathcal{O}(\epsilon^{-1})$.
To avoid large batches, \citet{Zhang2019OneSS} proposed the one-sample stochastic Frank-Wolfe algorithm (1-SFW), which attains the same SFO complexity with LMO complexity of $\mathcal{O}(\epsilon^{-3})$.
Beyond the Frank-Wolfe gap, \citet{pmlr-v80-qu18a} and \citet{Balasubramanian2018ZerothOrderNS} analyzed projection-free methods under the gradient mapping criterion and obtained $\mathcal{O}(\epsilon^{-2})$ complexity for both SFO and LMO, using a batch size $\mathcal{O}(\epsilon^{-1})$ per iteration.

In the multi-level setting, \citet{NEURIPS2022_7e16384b} proposed a projection-free conditional gradient-type algorithm that combines linearized NASA with inexact conditional gradient~\citep{Balasubramanian2018ZerothOrderNS}.
It achieves SFO complexity $\mathcal{O}(\epsilon^{-2})$ and LMO complexity $\mathcal{O}(\epsilon^{-3})$ under the gradient mapping criterion.
However, its SFO rate does not match the $\mathcal{O}(\epsilon^{-1.5})$ complexity achieved by variance-reduced methods for unconstrained multi-level optimization. Moreover, the analysis is limited to non-convex objectives for the gradient mapping criterion, which motivates the methods developed in this paper.

\section{Projection-free Variance Reduction for Non-convex Functions}
This section presents projection-free variance-reduced algorithms for stochastic multi-level compositional optimization in the non-convex setting. We first state the assumptions and then provide guarantees under two criteria: the Frank-Wolfe gap and the gradient mapping.
\subsection{Assumptions}
We begin with standard assumptions commonly used in  compositional optimization, variance reduction and projection-free analysis~\citep{Zhang2019OneSS,8472787,NEURIPS2020_7a53928f,pmlr-v119-huang20j,Zhang2021MultiLevelCS,jiang2022optimal}.
\begin{ass} \label{asm:0} 
We suppose $F\left(\x_{1}\right)-\min_{\x \in \X} F(\x) \leq \Delta_{F}$ for the initial point $\x_{1}$. The feasible set $\X$ is closed and convex with bounded diameter, i.e., $\max _{\x, \y \in \mathcal{X}}\Norm{\x-\y} \leq D$.
\end{ass}

\begin{ass} \label{asm:1} 
All functions $f_1, \dotsc, f_K$ are $L_f$-Lipschitz continuous, and their Jacobians $\nabla f_1, \dotsc, \nabla f_K$ are $L_J$-Lipschitz continuous on the bounded domain $\mathcal{X}$.
\end{ass}
\textbf{Remark:} Under Assumption~\ref{asm:1}, the objective $F$ is $L_F$-smooth, where
$L_F \coloneqq L_f^{2K-1}L_J\sum_{i=1}^K\frac{1}{L_f^i}$ \citep{jiang2022optimal}.

\begin{ass}\label{asm:stochastic2} For $1\leq i\leq K$, we assume that:
\begin{gather*}
       \E_{\xi_t^i}\left[f_i(\x;\xi_t^i)\right] = f_i(\x), \quad
	  \E_{\xi_t^i}\left[\nabla f_i(\x;\xi_t^i)\right] = \nabla f_i(\x),\\
 \E_{\xi_t^i}\left[\Norm{f_i(\x;\xi_t^i) - f_i(\x)}^2 \right]\leq \sigma^2, \quad
	\E_{\xi_t^i}\left[\Norm{\nabla f_i(\x;\xi_t^i) - \nabla f_i(\x) }^2 \right] \leq  \sigma_J^2,
\end{gather*}
where $\{\xi_t^i\}_{i=1}^K$ are mutually independent.
\end{ass}

\begin{ass}\label{asm:stoc_smooth3} (Average smoothness and Lipschitz)
For $1\leq i\leq K$, we assume that:
	\begin{align*}
	\E_{\xi_t^i}\left[\Norm{f_i(\x;\xi_t^i)-  f_i(\y;\xi_t^i)}^2\right] &\leq \LL_f^2 \Norm{\x - \y}^2,\\
		\E_{\xi_t^i}\left[\Norm{\nabla f_i(\x;\xi_t^i)- \nabla f_i(\y;\xi_t^i)}^2\right] &\leq  \LL_J^2 \Norm{\x - \y}^2.
	\end{align*}
\end{ass}

\subsection{Results for Frank-Wolfe Gap}
We first analyze sample complexity under the Frank-Wolfe gap criterion.
The primary procedure of our algorithm involves estimating the gradient of the objective function and then employing the Frank-Wolfe method to replace the projection operation. 
Note that the gradient of the multi-level function exhibits a nested structure; thus, the estimation error accumulates as the level becomes deeper. 
To control this accumulation, we adopt the STORM variance-reduction technique~\citep{cutkosky2019momentum} to track both inner function values and the overall gradient.
At iteration $t$, for each level $i$, we draw a mini-batch $\{\xi_t^{i,1},\ldots,\xi_t^{i,B_1}\}$ of size $B_1$ and maintain a variance-reduced estimator $\u_t^i$ to track the inner function value $f_i(\cdot)$ via
\begin{equation}\label{VR:u}
    \begin{aligned}
    \u_t^i = (1-\alpha)\u_{t-1}^i +   \frac{1}{B_1} \sum_{j=1}^{B_1} f_i(\u_t^{i-1};\xi_t^{i,j})
    -  (1-\alpha)\frac{1}{B_1} \sum_{j=1}^{B_1}f_i(\u_{t-1}^{i-1};\xi_t^{i,j}) .
\end{aligned}
\end{equation}
This evaluation ensures that the estimation error is reduced over time. We then build a similar variance-reduced estimator $\v_t$ for evaluating $\nabla F(\x_t)$ as:
\begin{equation}\label{VR:v}
\begin{aligned}
\v_t = (1-\alpha)\v_{t-1} +\frac{1}{B_1} \sum_{j=1}^{B_1} \left[ \prod_{i=1}^K \nabla f_i(\u_t^{i-1};\xi_t^{i,j}) \right] 
- (1-\alpha)\frac{1}{B_1} \sum_{j=1}^{B_1} \left[\prod_{i=1}^K \nabla f_i(\u_{t-1}^{i-1};\xi_t^{i,j})\right] .
\end{aligned}
\end{equation}
After obtaining the gradient estimation, we follow the framework of the Frank-Wolfe algorithm \citep{pmlr-v28-jaggi13}, but use the estimator $\v_t$ to replace the gradient required in the original algorithm as follows:
\begin{gather*}
    \z_{t} = \arg \min_{\x \in \X} \langle \x,\v_{t} \rangle, \quad \x_{t+1} = \x_t + \eta (\z_t - \x_t).
\end{gather*}
This yields our Projection-free Multi-level Variance Reduction (PMVR) method, summarized in Algorithm~\ref{alg:1}~(PMVR).
For initialization (when iteration $t=1$), we use mini-batch averages: $\u_{1}^i = \frac{1}{B_0} \sum_{j=1}^{B_0}  f_i(\u_{1}^{i-1};\xi^{i,j}_1)$ and $\v_{1} = \frac{1}{B_0} \sum_{j=1}^{B_0} \left[ \prod_{i=1}^K  \nabla f_i(\u_{1}^{i-1};\xi_1^{i,j}) \right]$, where $B_0$ is the initial batch size.
\begin{algorithm}[!tb]
	\caption{PMVR/PMM Algorithm}
	\label{alg:1}
	\begin{algorithmic}[1]
	\STATE {\bfseries Input:} parameters $T$, $\eta$, $\alpha$, initial points $\left(\x_1,\u_1,\v_1\right)$ 
		\FOR{time step $t = 1$ {\bfseries to} $T$}
		\STATE Set $\u_t^0 = \x_t$
		\FOR{level $i = 1$ {\bfseries to} $K$}
		\STATE (PMVR) Compute the function estimator $\u_t^i$ according to equation~(\ref{VR:u})
        \STATE (PMM) Compute the function estimator $\u_t^i$ according to equation~(\ref{ss1})
        \ENDFOR
        \STATE (PMVR) Compute the gradient estimator $\v_t$ according to equation~(\ref{VR:v}) 
        \STATE (PMM)  Compute the gradient estimator $\v_t$ according to equation~(\ref{ss2}) 
        \STATE (v1) Compute $\z_{t} = \arg \min_{\x \in \X} \langle \x,\v_{t} \rangle$
        \STATE (v2) Compute $\z_t = \operatorname{IFW}(\x_t,\beta)$
		\STATE Update the weight: $\x_{t+1} = \x_t + \eta (\z_t - \x_t) $
		\ENDFOR
	\STATE Choose $\tau$ uniformly at random from $\{1, \ldots, T\}$
	\STATE {\bfseries Return} $(\x_{\tau},\u_{\tau},\v_{\tau})$
	\end{algorithmic}
\end{algorithm}
We now present the complexity of Algorithm~\ref{alg:1}~(PMVR) in terms of the Frank-Wolfe gap $\F(\cdot)$ in equation~(\ref{FW}). 
Noting that using larger batches reduces the number of iterations and thus can improve the LMO complexity. Accordingly, we provide both constant-batch and large-batch guarantees.
\begin{theorem}\label{thm:main}
Under Assumptions~\ref{asm:0},\ref{asm:1},\ref{asm:stochastic2},\ref{asm:stoc_smooth3}, PMVR-v1 enjoys the following guarantees:
\begin{compactenum}[(i)]
\item By setting $B_1 = \Omega(1)$, $\eta = \Theta(\epsilon^2)$, and $\alpha = \Theta(\epsilon^2)$, we can ensure that  $\E\left[\F(\x_\tau)\right] \leq \epsilon$ within $\mathcal{O}(\epsilon^{-3})$ iterations.
\item By setting $B_1 = \Omega(\epsilon^{-1})$, $\eta = \Theta(\epsilon)$, and $\alpha = \Theta(\epsilon)$, we can ensure that $\E\left[\F(\x_\tau)\right] \leq \epsilon$ within $\mathcal{O}(\epsilon^{-2})$ iterations.
\end{compactenum}
\end{theorem}
\textbf{Remark:} 
(i) With constant batch size, the theorem implies $\mathcal{O}(\epsilon^{-3})$ complexity for both SFO and LMO calls, matching the rates for projection-free single-level problems~\citep{Zhang2019OneSS}.
(ii) With a large batch size $B_1=\Omega(\epsilon^{-1})$, we obtain SFO complexity of $\mathcal{O}(\epsilon^{-3})$ and LMO complexity of $\mathcal{O}(\epsilon^{-2})$, consistent with the existing single-level projection-free results~\citep{pmlr-v97-yurtsever19b}.

\subsection{Parameter-free Variants}\label{pf}
Although we developed projection-free algorithms with provable guarantees, they require problem-dependent parameters to set up hyperparameters. 
In this subsection, we present a fully parameter-free algorithm that achieves similar rates without knowing problem constants.
Specifically, we design a multi-stage scheme with stages $s = \{1, 2, \ldots, S\}$. At the beginning of each stage, we reset the hyperparameters~$\{T_s, \eta_s, \alpha_s\}$, and initialize each stage with the output of the previous one. In each stage $s$, the algorithm is executed for $T_s = 2^{s-1}$ iterations, effectively doubling the iteration numbers after each stage. 
 Equivalently, at iteration $t$, the current stage index can be identified as $s(t)=1+\lfloor \log_2 t \rfloor$, with $T_{s(t)}=2^{\lfloor \log_2 t \rfloor}$. The complete algorithm is summarized in Algorithm~\ref{alg:3}. Next, we present the theoretical guarantees for this stage-wise algorithm.
\begin{algorithm}[!tb]
	\caption{Stage-wise PMVR}
	\label{alg:3}
	\begin{algorithmic}[1]
	\STATE {\bfseries Input:} initial points $\left(\x_0,\u_0,\v_0\right)$
		\FOR{stage $s = 1$ {\bfseries to} $S$}
		\STATE $(\x_{s},\u_{s},\v_{s})$ = PMVR~(with parameters $T_{s}$, $\eta_s$, $\alpha_s$ and initialization $\left(\x_{s-1},\u_{s-1},\v_{s-1}\right)$)
		\ENDFOR
	\STATE Return $\x_{S}$
	\end{algorithmic}
\end{algorithm}
\begin{theorem}\label{thm:main+}
Under Assumptions~\ref{asm:0},\ref{asm:1},\ref{asm:stochastic2},\ref{asm:stoc_smooth3}, Stage-wise PMVR-v1 guarantees the following:
\begin{compactenum}[(i)]
    \item  By setting $T_s=2^{s-1}$, $B_1^s = 1$, and $\eta_s = \alpha_s = T_s^{-2/3}$, we can ensure that  $\E\left[\F(\x_\tau)\right] \leq \epsilon$ within $\mathcal{O}(\epsilon^{-3})$ iterations.
    \item By setting $T_s=2^{s-1}$, $B_1^s = \sqrt{T_s}$, and $\eta_s = \alpha_s = T_s^{-1/2}$, we can ensure $\E\left[\F(\x_\tau)\right] \leq \epsilon$ within $\mathcal{O}(\epsilon^{-2})$ iterations.
\end{compactenum}
\end{theorem}
\textbf{Remark:} The stage-wise schedule matches the SFO and LMO complexities in Theorem~\ref{thm:main} while avoiding problem-dependent constants to set hyper-parameters.

\subsection{Results for Gradient Mapping}
We now analyze complexity under the gradient mapping criterion $\G(\cdot)$ defined in equation~(\ref{GM}).
To deal with the gradient mapping, our PMVR-v1 algorithm only requires minimal modifications to fit this criterion. Using Proposition~2 of~\citet{NEURIPS2022_7e16384b}, gradient mapping can be decomposed into two components:
\begin{align*}
  \mathcal{G}(\x_t) \leq - 4 \beta \min _{\w \in \X} g(\w, \x_t, \v_t)+2\left\| \nabla F(\x_t)- \v_t \right\|^{2},  
\end{align*}  
where $ g(\w, \x_t, \v_t)= \langle \v_t, \w-\x_t \rangle+\frac{\beta}{2}\|\w-\x_t\|^{2}$. 
Since our PMVR method has already employed variance-reduced techniques to ensure that the gradient estimation error $\E\left[\Norm{\nabla F(\x_t)- \v_t }^{2}\right]$ decreases over time, we can reuse this part and focus on controlling the $-\min _{\w \in \X} g(\w, \x_t, \v_t)$ term. 
To address this constrained quadratic minimization problem, we design a sub-algorithm, modified from the Frank-Wolfe method~\citep{pmlr-v28-jaggi13}, which runs for $N$ loops. At inner iteration $n$, we compute:
\begin{gather*}
    \s
      = \arg \min _{\hat{\s} \in \X}\left\langle \v_t+\beta\left(\w_n-\x_t\right), \hat{\s}\right\rangle, \quad 
      \w_{n+1}=\left(1-\gamma_{t}\right) \w_{n}+\gamma_{t} \s.
\end{gather*}
Note that $\v_t+\beta\left(\w_n-\x_t\right)$ is the gradient of $g(\w_n,\x_t,\v_t)$ with respect to $\w$ at $\w=\w_n$.
Thus, we obtain the PMVR-v2 method by replacing Step~10 of Algorithm~\ref{alg:1} with Algorithm~\ref{alg:2} and can provide the guarantees for gradient mapping in the following theorem.
\begin{algorithm}[!tb]
	\caption{Inner Frank-Wolfe Method~(IFW)}
	\label{alg:2} 
	\begin{algorithmic}[1]
 \STATE {\bfseries Input:} initial point $\x_t$, parameter $\beta$
         \STATE Initialize $\w_1 = \x_t$
        \FOR{inner step $n = 1$ {\bfseries to} $N$}
	\STATE Compute $ \s
      = \arg \min _{\hat{\s} \in \X}\left\langle \v_t+\beta\left(\w_n-\x_t\right), \hat{\s}\right\rangle$
	\STATE Set $\w_{n+1}=\left(1-\gamma_{t}\right) \w_{n}+\gamma_{t} \s$
	\ENDFOR
     \STATE {\bfseries Return} $\z_t = \w_{N+1}$
\end{algorithmic}
\end{algorithm}
\begin{theorem}\label{thm:2} Under Assumptions~\ref{asm:0},\ref{asm:1},\ref{asm:stochastic2},\ref{asm:stoc_smooth3}, our PMVR-v2 enjoys the following guarantees:
\begin{compactenum}[(i)]
    \item By setting $B_1 = \Omega(1)$, $N = \Omega(\epsilon^{-1})$, $\eta = \Theta(\sqrt{\epsilon})$, and $\alpha = \Theta(\epsilon)$, we can ensure $ \E\left[\G (\x_\tau) \right]\leq \epsilon$ in $\mathcal{O}(\epsilon^{-1.5})$ iterations.
    \item  By setting $B_1 = \Omega(\epsilon^{-0.5})$, $N = \Omega(\epsilon^{-1})$, $\eta = \Theta(1)$, and $\alpha = \Theta(\sqrt{\epsilon})$, we can ensure $\E\left[\G (\x_\tau) \right] \leq \epsilon$ in $\mathcal{O}(\epsilon^{-1})$ iterations.
\end{compactenum}
\end{theorem}
\textbf{Remark:} 
(i) With constant batch size, PMVR-v2 yields $\mathcal{O}(\epsilon^{-1.5})$ SFO complexity and $\mathcal{O}(\epsilon^{-2.5})$ LMO complexity, improving upon LiNASA+ICG~\citep{NEURIPS2022_7e16384b} (which attains $\mathcal{O}(\epsilon^{-2})$ SFO and $\mathcal{O}(\epsilon^{-3})$ LMO complexities).
(ii) With $B_1=\Omega(\epsilon^{-0.5})$, PMVR-v2 achieves $\mathcal{O}(\epsilon^{-1.5})$ SFO and $\mathcal{O}(\epsilon^{-2})$ LMO complexities.
These rates improve upon single-level methods such as NCGS~\citep{pmlr-v80-qu18a} and SGD+ICG~\citep{Balasubramanian2018ZerothOrderNS}, which use larger batch size $\mathcal{O}(\epsilon^{-1})$ and have worse SFO complexity $\mathcal{O}(\epsilon^{-2})$.
Notably, our SFO complexity of $\mathcal{O}(\epsilon^{-1.5})$ also matches the lower bound for stochastic unconstrained optimization~\citep{Arjevani2019LowerBF}. 

\section{Projection-free Momentum-based Methods for Non-convex Functions}
In this section, we develop momentum-based projection-free methods for non-convex multi-level compositional optimization under weaker smoothness assumptions. In particular, unlike the variance-reduced methods in the previous section, the results here do not rely on the average smoothness/Lipschitz assumptions. 
We first provide guarantees under the Frank-Wolfe gap and then extend the approach to the gradient mapping criterion.
\subsection{Results for Frank-Wolfe Gap}
Although the variance-reduced methods achieve improved complexity, their analysis relies on average smoothness and average Lipschitz conditions (Assumption~\ref{asm:stoc_smooth3}), which are stronger than the standard smoothness/Lipschitz assumptions typically used in stochastic optimization. 
Momentum-based methods are known to yield convergence guarantees under standard smoothness assumptions. We therefore develop a simpler momentum-based projection-free algorithm and provide its theoretical guarantees.

The momentum-based method differs from PMVR primarily in the construction of the function and gradient estimators. 
Instead of the STORM-type variance-reduced updates, we use exponential moving averages. For each level $i$, we maintain an estimator $\u_t^i$ to track the inner function value $f_i(\cdot)$ via
\begin{equation}
\begin{aligned}\label{ss1}
    \u_t^i = (1-\alpha)\u_{t-1}^i + \alpha  \frac{1}{B_1} \sum_{j=1}^{B_1} f_i(\u_t^{i-1};\xi_t^{i,j}).
\end{aligned}
\end{equation}
This update is simpler and controls the estimation error over time without Assumption~\ref{asm:stoc_smooth3}. Similarly, we employ a momentum-based estimator $\v_t$ to evaluate the overall gradient by
\begin{equation}
\begin{aligned}\label{ss2}
\v_t = (1-\alpha)\v_{t-1} +\alpha \frac{1}{B_1} \sum_{j=1}^{B_1} \left[ \prod_{i=1}^K \nabla f_i(\u_t^{i-1};\xi_t^{i,j}) \right].
\end{aligned}
\end{equation}
The remaining steps follow PMVR, replacing projections with a Frank-Wolfe update. The full method is given in Algorithm~\ref{alg:1}~(PMM). Next, we present the complexity guarantees with respect to the Frank-Wolfe gap.
\begin{theorem}\label{thm:main-}
Under Assumptions~\ref{asm:0},\ref{asm:1},\ref{asm:stochastic2}, by setting $B_1 = T$, $\eta = T^{-1/2}$, and $\alpha = 1/2$, our PMM-v1 algorithm ensures $\E\left[\F(\x_\tau)\right] \leq \mathcal{O}(1/\sqrt{T})$.
\end{theorem}
\textbf{Remark:} (i) To ensure $\E\left[\F(\x_\tau)\right] \leq \epsilon$, we require $B_1=T=\Omega(\epsilon^{-2})$.  
Therefore, the resulting SFO and LMO complexities are $\mathcal{O}(\epsilon^{-4})$ and $\mathcal{O}(\epsilon^{-2})$, respectively.
(ii) As in Section~\ref{pf}, the above result can be made fully parameter-free via a stage-wise schedule with $B_1^s=T_s$, $\eta_s=T_s^{-1/2}$, and $\alpha_s=1/2$ for stage-wise PMM-v1.

\subsection{Results for Gradient Mapping}
We consider the gradient mapping criterion $\G(\cdot)$ in this subsection. As in PMVR-v2, we replace Step 10 in Algorithm~\ref{alg:1}~(PMM) with the inner Frank-Wolfe procedure (Algorithm~\ref{alg:2}), yielding PMM-v2. Then, we can obtain the guarantees for gradient mapping below.
\begin{theorem}\label{thm:2-}
Under Assumptions~\ref{asm:0},\ref{asm:1},\ref{asm:stochastic2}, by setting $B_1 = N = T$, $\eta = \Theta(1)$, and $\alpha = 1/2$, our PMM-v2 algortithm can ensure $\E\left[\G (\x_\tau) \right] \leq \mathcal{O}(1/T)$.
\end{theorem}
\textbf{Remark:} To ensure $\E\left[\G (\x_\tau) \right] \leq \epsilon$, we require $B_1=N=T=\Omega(\epsilon)$. 
Therefore, the SFO complexity $B_1T$ and the LMO complexity $NT$ are both $\mathcal{O}(\epsilon^{-2})$.
Notably, the $\mathcal{O}(\epsilon^{-2})$ SFO complexity matches the lower bound for stochastic unconstrained optimization under standard smoothness assumptions~\citep{Arjevani2019LowerBF} and is thus optimal.
\section{The Proposed Methods for Convex and Strongly Convex Functions}
In addition to the non-convex analyses in previous sections, we study the cases where the objective function $F$ is convex or strongly convex, defined as follows.
\begin{ass}
A function $F:\X \mapsto \mathbb{R} $ is convex if
\begin{align*}
  F(\mathbf{y}) \geq F(\mathbf{x})+\langle\nabla F(\mathbf{x}), \mathbf{y}-\mathbf{x}\rangle, \quad \forall \mathbf{x}, \mathbf{y} \in \X.
\end{align*}
\end{ass}
\begin{ass}\label{def:sc}
A function $F:\X \mapsto \mathbb{R} $ is $\lambda$-strongly convex if
\begin{align*}
  F(\mathbf{y}) \geq F(\mathbf{x})+\langle\nabla F(\mathbf{x}), \mathbf{y}-\mathbf{x}\rangle+\frac{\lambda}{2}\|\mathbf{y}-\mathbf{x}\|^{2}, \quad \forall \mathbf{x}, \mathbf{y} \in \X.
\end{align*}
\end{ass}
When $F$ is convex (or strongly convex), any local optimal point is globally optimal. Accordingly, we measure progress using the optimal gap in equation~(\ref{OP}). Next, we consider the case for convex and strongly convex functions, respectively.
\subsection{Convex Objectives}
First, we investigate the case for convex functions. 
To obtain optimal-gap guarantees, we adopt the stage-wise warm-start framework of Algorithm~\ref{alg:3}. 
Specifically, we run Algorithm~\ref{alg:1}~(PMVR) for $S$ stages. At stage $s$ we use a new set of parameters $\{T_s,\eta_s,\alpha_s\}$ and initialize the algorithm with the output from the previous stage $(\x_{s-1}, \u_{s-1}, \v_{s-1})$.
By decreasing $\eta_s$ and $\alpha_s$ while increasing $T_s$ across stages, we can ensure geometric decay of the expected optimal gap, namely $\E\left[F(\x_s) - F_{\star}\right] \leq \frac{1}{2} \E\left[F(\x_{s-1}) - F_{\star}\right]$. 
Let $S = \mathcal{O}(\log (\frac{\epsilon_1}{\epsilon}))$ and define $\epsilon_s = \epsilon_1 /2^{s}$, where $\epsilon_1$ is a positive constant. We can obtain the guarantees for the optimal gap in the following theorem.
\begin{theorem}\label{thm:3}
For convex objectives, under Assumptions~\ref{asm:0},\ref{asm:1},\ref{asm:stochastic2},\ref{asm:stoc_smooth3}, Stage-wise PMVR-v1 enjoys the following guarantees:
\begin{compactenum}[(i)]
    \item By setting $B_1 = \Omega(1)$,  $\eta_s = \Theta(\epsilon_s^2)$, $\alpha_s = \Theta(\epsilon_s^2)$, and $T_s = \Omega(\epsilon_s^{-2})$, we can ensure that $\E\left[F(\x_{S})\right]-\min_{\hat{\x}} F(\hat{\x}) \leq \epsilon$ in $\mathcal{O}(\epsilon^{-2})$ iterations.
    \item By setting $B_1=\Omega(\epsilon_s^{-1})$, $\eta_s = \Theta(\epsilon_s)$, $\alpha_s = \Theta(\epsilon_s)$, and $T_s = \Omega(\epsilon_s^{-1})$, we can ensure $\E\left[F(\x_{S})\right]-\min_{\hat{\x}} F(\hat{\x}) \leq \epsilon$ in $\mathcal{O}(\epsilon^{-1})$ iterations.
\end{compactenum}

\end{theorem}
\textbf{Remark:} (i) With constant batch size, the above result implies $\mathcal{O}(\epsilon^{-2})$ complexity for both SFO and LMO, matching the best-known rates for single-level projection-free methods~\citep{Zhang2019OneSS}. 
The $\mathcal{O}(\epsilon^{-2})$ SFO rate is already optimal for stochastic convex optimization~\citep{Agarwal2012InformationTheoreticLB}.
(ii) With a large batch size $B_1=\Omega(\epsilon^{-1})$, the method achieves the same SFO complexity and improves the LMO complexity to $\mathcal{O}(\epsilon^{-1})$, consistent with single-level projection-free methods such as SPIDER-FW~\citep{pmlr-v97-yurtsever19b}.

\subsection{Strongly Convex Objectives} 
For strongly convex objectives, we incorporate a quadratic regularization term in the conditional-gradient step.
 Specifically, instead of using $\arg \min_{\x \in \X}\langle \x,\v_t\rangle$ in Algorithm~\ref{alg:1}, we aim to approximately solve the strongly convex quadratic surrogate
\begin{align*}
    \min _{\w \in \X} g(\w,\x_t,\v_t) \coloneqq \left\{\langle \w,\v_t \rangle+\frac{\lambda}{4}\|\w-\x_t\|^{2}\right\}.
\end{align*} 
We solve this sub-problem using only LMO calls by running a Frank-Wolfe procedure~\citep{pmlr-v28-jaggi13} for $N$ inner iterations. At inner iteration $n$, we compute
\begin{equation}
   \begin{split}
       \s  = &\arg \min _{\hat{\s} \in \X}\left\langle \v_t+\frac{\lambda}{2}\left(\w_n-\x_t\right), \hat{\s}\right\rangle, \quad
\w_{n+1}=\left(1-\gamma_{t}\right) \w_{n}+\gamma_{t} \s,
\end{split}
\end{equation}
where $\v_t+\frac{\lambda}{2}(\w_n-\x_t)$ is the gradient of $g(\w,\x_t,\v_t)$ with respect to $\w$ at $\w=\w_n$.
Overall, we retain the stage-wise warm-start design of Algorithm~\ref{alg:3} and use the PMVR-v2 variant with $\beta=\lambda/2$.
With $\epsilon_s=\epsilon_1/2^s$ and $\epsilon_1 \ge 0$, we obtain the following guarantees.
\begin{theorem}\label{thm:4}
For strongly convex objective functions, under Assumptions~\ref{asm:0},\ref{asm:1},\ref{asm:stochastic2},\ref{asm:stoc_smooth3}, our proposed Stage-wise PMVR-v2 algorithm enjoys the following guarantees:
\begin{compactenum}[(i)]
    \item  Setting $B_1 = \Omega(1)$, $N=\Omega(\lambda \epsilon^{-1})$,  $\eta_s = \Theta(\lambda \epsilon_s)$, $\alpha_s = \Theta(\lambda \epsilon_s)$, and $T_s = \Omega(\lambda^{-1}\epsilon_s^{-1})$, we ensure $\E\left[F(\x_{S})\right]-\min_{\hat{\x}} F(\hat{\x}) \leq \epsilon$ in $S = \mathcal{O}(\log (\frac{\epsilon_1}{\epsilon}))$ stages.
    \item 
By setting $B_1 = \Omega(\epsilon_s^{-1})$, $N= \Omega(\lambda \epsilon^{-1})$,  $\eta_s = \Theta(\lambda)$, $\alpha_s = \Theta(\lambda)$, and $T_s = \Omega(\lambda^{-1})$, we can ensure $\E\left[F(\x_{S})\right]-\min_{\hat{\x}} F(\hat{\x}) \leq \epsilon$ within $S = \mathcal{O}(\log (\frac{\epsilon_1}{\epsilon}))$ stages.
\end{compactenum}
\end{theorem}
\textbf{Remark:} 
(i) With constant batch size, we obtain SFO complexity of $\mathcal{O}(\lambda^{-1}\epsilon^{-1})$ and LMO complexity of $\mathcal{O}(\epsilon^{-2})$. 
Notably, the $\mathcal{O} (\lambda^{-1}\epsilon^{-1})$ SFO complexity we obtained is already optimal for stochastic unconstrained strongly convex problems~\citep{Agarwal2012InformationTheoreticLB}.
(ii) With a large batch size, the SFO complexity remains of the same order, while the LMO complexity improves to $\mathcal{O}(\epsilon^{-1})$, matching the known single-level projection-free results for strongly convex objectives~\citep{doi:10.1137/140992382}.

\section{The Proposed Methods for Finite-sum Structure}
In this section, we consider the case where each level has a finite-sum structure. We first present guarantees for non-convex objectives, and then provide the results for convex and strongly convex settings. 
\subsection{Results for Non-convex Functions}
For the multi-level finite-sum problem, we assume that the function in each level $i$ can be written as $f_i(\x) = \frac{1}{m}\sum_{j=1}^m f_{i}(\x;\xi^{i,j})$.
This structure allows us to periodically compute exact function values and Jacobians, which can be reused to build more accurate estimators.
Specifically, at iteration $t=1$ and every $I$ iterations (i.e., when $t \bmod I = 0$), we set a snapshot index $\tau=t$ and compute the exact quantities $f_i(\u_\tau^{i-1})$ and $\nabla f_i(\u_\tau^{i-1})$ for each level $i$.
At other iterations, we construct estimators using the snapshot information. For the inner functions, we use
\begin{equation}\label{fs:u}
    \begin{aligned}
    \u_t^i = (1-\alpha)\u_{t-1}^i + \alpha \h_t^i
    +  (1-\alpha)\frac{1}{B_1} \sum_{j=1}^{B_1}\left(f_i(\u_{t}^{i-1};\xi_t^{i,j}) - f_i(\u_{t-1}^{i-1};\xi_t^{i,j})\right),
\end{aligned}
\end{equation}
where $\h_t^i$ is an unbiased estimator of $f_i(\u_t^{i-1})$ defined by
\begin{equation}
    \begin{aligned}
    \h_t^i =  \frac{1}{B_1} \left(\sum_{j=1}^{B_1} f_i(\u_t^{i-1};\xi_t^{i,j}) - f_i(\u_\tau^{i-1};\xi_t^{i,j})\right) +f_i(\u_\tau^{i-1}).
\end{aligned}
\end{equation}
Using this new estimator, we can ensure faster decay of the estimation error with appropriate parameter choices. 
We also employ a similar design of the gradient estimator, such that
\begin{equation}\label{fs:v}
    \begin{aligned}
    \v_t = (1-\alpha)\v_{t-1}+\alpha \z_t+(1-\alpha)\left(\frac{1}{B_1} \sum_{j=1}^{B_1} \left[\prod_{i=1}^K \nabla f_i(\u_t^{i-1};\xi_t^{i,j})-\prod_{i=1}^K \nabla f_i(\u_{t-1}^{i-1};\xi_t^{i,j})\right]\right) 
\end{aligned}
\end{equation}
where $\z_t$ is an unbiased estimator of $\prod_{i=1}^K\nabla f_i(\u_t^{i-1})$ given by
\begin{equation}
    \begin{aligned}
    \z_t = \frac{1}{B_1} \sum_{j=1}^{B_1} \left[\prod_{i=1}^K \nabla f_i(\u_t^{i-1};\xi_t^{i,j})-\prod_{i=1}^K \nabla f_i(\u_{\tau}^{i-1};\xi_t^{i,j})\right] +\prod_{i=1}^K\nabla f_i(\u_\tau^{i-1}).
\end{aligned}
\end{equation}
The resulting method is summarized in Algorithm~\ref{alg:fs} and is referred to as PMFS (Projection-free Multi-level method for Finite-sum Structure). Next, we state the theoretical guarantee for the proposed methods under the Frank-Wolfe gap.
\begin{algorithm}[!tb]
	\caption{PMFS Algorithm}
	\label{alg:fs}
	\begin{algorithmic}[1]
	\STATE {\bfseries Input:} parameters $T$, $\eta$, $\alpha$, initial points $\left(\x_1,\u_1,\v_1\right)$ 
		\FOR{time step $t = 1$ {\bfseries to} $T$}
		\STATE Set $\u_t^0 = \x_t$
        \IF{$t ==1$ or $t \mod I ==0$}
        \STATE Set $\tau =t$ and compute $ f_i(\u_t^i)$ and $\nabla f_i(\u_t^i)$ for every layer $i$
        \ENDIF
		\FOR{level $i = 1$ {\bfseries to} $K$}
		\STATE Compute the function estimator $\u_t^i$ according to equation~(\ref{fs:u})
        \ENDFOR
        \STATE Compute the gradient estimator $\v_t$ according to equation~(\ref{fs:v})
        \STATE (v1) Compute $\z_{t} = \arg \min_{\x \in \X} \langle \x,\v_{t} \rangle$
        \STATE (v2) Compute $\z_t = \operatorname{IFW}(\x_t,\beta)$
		\STATE Update the weight: $\x_{t+1} = \x_t + \eta (\z_t - \x_t) $
		\ENDFOR
	\STATE Choose $\tau$ uniformly at random from $\{1, \ldots, T\}$
	\STATE {\bfseries Return} $(\x_{\tau},\u_{\tau},\v_{\tau})$
	\end{algorithmic}
\end{algorithm}

\begin{theorem}\label{thm:fsnc}
Under Assumptions~\ref{asm:0},\ref{asm:1},\ref{asm:stochastic2},\ref{asm:stoc_smooth3}, our PMFS-v1 enjoys the following guarantees:
\begin{compactenum}[(i)]
    \item By setting $B_1 = 1$, $I=m/B_1$, $\alpha = {B_1}/{m}$ and $\eta = {m^{-1/4}T^{-1/2}}$, we can ensure that  $\E\left[\F(\x_\tau)\right] \leq \epsilon$ within $\mathcal{O}(\sqrt{m}\epsilon^{-2})$ iterations.
    \item By setting $B_1 = \sqrt{m}$, $I=m/B_1$, $\alpha = {B_1}/{m}$ and $\eta = T^{-1/2}$, we can ensure that $\E\left[\F(\x_\tau)\right] \leq \epsilon$ within $\mathcal{O}(\epsilon^{-2})$ iterations.
\end{compactenum}
\end{theorem}
\textbf{Remark:} 
(i) With constant batch size, both the SFO and LMO complexities are $\mathcal{O}(\sqrt{m}\epsilon^{-2})$.
(ii) With batch size $B_1=\sqrt{m}$, the SFO complexity is $\mathcal{O}(\sqrt{m}\epsilon^{-2})$ and the LMO complexity improves to $\mathcal{O}(\epsilon^{-2})$, consistent with projection-free results for single-level finite-sum objectives~\citep{beznosikov2024sarah}.
(iii) As in Section~\ref{pf}, the result can be made parameter-free via a stage-wise schedule, e.g., $\eta_s = m^{-1/4}T_s^{-1/2}$ (constant batch) and $\eta_s = T_s^{-1/2}$ (large batch) for Stage-wise PMFS-v1.

We next consider the gradient mapping criterion $\G(\cdot)$. As in PMVR-v2, we only need to replace Step~11 of Algorithm~\ref{alg:fs} with the inner Frank-Wolfe subroutine (Algorithm~\ref{alg:2}), yielding PMFS-v2. We can obtain the guarantees for gradient mapping below.
\begin{theorem}\label{thm:fsnc2} Under Assumptions~\ref{asm:0},\ref{asm:1},\ref{asm:stochastic2},\ref{asm:stoc_smooth3}, our PMFS-v2 enjoys the following guarantees:
\begin{compactenum}[(i)]
    \item By setting $B_1 = \Omega(1)$, $I=m/B_1$, $N = \Omega(\epsilon^{-1})$, $\eta = \Theta(1/\sqrt{m})$, and $\alpha = \Theta(1/m)$, we can ensure $ \E\left[\G (\x_\tau) \right]\leq \epsilon$ in $\mathcal{O}(\sqrt{m}\epsilon^{-1})$ iterations.
    \item By setting $B_1 = \Omega(\sqrt{m})$, $I=m/B_1$, $N = \Omega(\epsilon^{-1})$, $\eta = \Theta(1)$, and $\alpha = \Theta(1/\sqrt{m})$, we can ensure $\E\left[\G (\x_\tau) \right] \leq \epsilon$ in $\mathcal{O}(\epsilon^{-1})$ iterations.
\end{compactenum}
\end{theorem}
\textbf{Remark:} 
(i) With constant batch size, the resulting SFO complexity is $\mathcal{O}(\sqrt{m}\epsilon^{-1})$ and the LMO complexity is $\mathcal{O}(\sqrt{m}\epsilon^{-2})$.
(ii) With batch size $\mathcal{O}(\sqrt{m})$, the SFO complexity remains $\mathcal{O}(\sqrt{m}\epsilon^{-1})$, while the LMO complexity improves to $\mathcal{O}(\epsilon^{-2})$.

\subsection{Results for Convex and Strongly Convex Functions}
We now provide optimal-gap guarantees for convex and strongly convex objectives, using the same stage-wise framework, i.e., Algorithm~\ref{alg:3}. 
At stage $s$, we run Algorithm~\ref{alg:fs} with parameters $(T_s,\eta_s,\alpha_s)$ and initialize it with the previous stage output $(\x_{s-1},\u_{s-1},\v_{s-1})$.
Let $\epsilon_s = \epsilon_1 / 2^{s}$, where $\epsilon_1\ge 0$ is a constant. We obtain the following optimal-gap guarantees.
\begin{theorem}\label{thm:3+}
For convex functions, under Assumptions~\ref{asm:0},\ref{asm:1},\ref{asm:stochastic2},\ref{asm:stoc_smooth3}, our Stage-wise PMFS-v1 enjoys the following guarantees:
\begin{compactenum}[(i)]
    \item Setting $B_1 = \Omega(1)$, $I=m/B_1$,  $\eta_s = \Theta(\epsilon_s/\sqrt{m})$, $\alpha_s = \Theta(1/m)$, and $T_s = \Omega(\sqrt{m}\epsilon_s^{-1})$, we can ensure $\E\left[F(\x_{S})\right]-\min_{\hat{\x}} F(\hat{\x}) \leq \epsilon$ in $\mathcal{O}(\sqrt{m}\epsilon^{-1})$ iterations.
    \item By setting $B_1=\Omega(\sqrt{m})$, $I=m/B_1$, $\eta_s = \Theta(\epsilon_s)$, $\alpha_s = \Theta(1/\sqrt{m})$, and $T_s = \Omega(\epsilon_s^{-1})$, we can ensure $\E\left[F(\x_{S})\right]-\min_{\hat{\x}} F(\hat{\x}) \leq \epsilon$ in $\mathcal{O}(\sqrt{m}\epsilon^{-1})$ iterations.
\end{compactenum}
\end{theorem}
\textbf{Remark:} 
(i) With constant batch size, both SFO and LMO complexities are $\mathcal{O}(\sqrt{m}\epsilon^{-1})$.
(ii) With batch size $\mathcal{O}(\sqrt{m})$, the SFO complexity remains $\mathcal{O}(\sqrt{m}\epsilon^{-1})$ and the LMO complexity improves to $\mathcal{O}(\epsilon^{-1})$.

For strongly convex objectives, we retain the stage-wise design of Algorithm~\ref{alg:3} with the PMFS-v2 variant and set $\beta=\lambda/2$. The resulting optimal-gap guarantees are as follows.
\begin{theorem}\label{thm:4+}
For strongly convex function, under Assumptions~\ref{asm:0},\ref{asm:1},\ref{asm:stochastic2},\ref{asm:stoc_smooth3}, our proposed Stage-wise PMFS-v2 algorithm enjoys the following guarantees:
\begin{compactenum}[(i)]
    \item Setting $B_1 = \Omega(1)$,  $I=m/B_1$, $N=\Omega(\lambda \epsilon^{-1})$,  $\eta_s = \Theta(\lambda /\sqrt{m})$, $\alpha_s = \Theta(1/m)$, and $T_s = \Omega(\sqrt{m}/\lambda)$, we ensure $\E\left[F(\x_{S})\right]-\min_{\hat{\x}} F(\hat{\x}) \leq \epsilon$ in $\mathcal{O}(\log (\frac{\epsilon_1}{\epsilon}))$ stages.
    \item By setting $B_1 = \Omega(\sqrt{m}/\lambda)$, $I=m/B_1$, $N= \Omega(\lambda \epsilon^{-1})$,  $\eta_s = \Theta(1)$, $\alpha_s = \Theta(1/(\lambda\sqrt{m}))$, and $T_s = \Omega(1)$, we can ensure $\E\left[F(\x_{S})\right]-\min_{\hat{\x}} F(\hat{\x}) \leq \epsilon$ within $\mathcal{O}(\log (\frac{\epsilon_1}{\epsilon}))$ stages.
\end{compactenum}
\end{theorem}
\textbf{Remark:}  
(i) With constant batch size, the SFO complexity is $\mathcal{O}(\sqrt{m}/\lambda\log(1/\epsilon))$ and the LMO complexity is $\mathcal{O}(\sqrt{m}/\epsilon\log(1/\epsilon))$. (ii) With a large batch size, the SFO complexity remains of the same order, while the LMO complexity improves to $\mathcal{O}(\lambda/\epsilon\log(1/\epsilon))$, consistent with single-level strongly convex projection-free results~\citep{doi:10.1137/140992382}.

\section{Experiments}\label{sec:4}
In this section, we evaluate the proposed methods on three benchmark problems. We compare against existing stochastic multi-level algorithms, including A-TSCGD~\citep{Yang2019MultilevelSG}, NLASG~\citep{balasubramanian2020stochastic}, Nested-SPIDER~\citep{Zhang2021MultiLevelCS}, SCSC~\citep{chen2021solving}, and SMVR~\citep{jiang2022optimal}.
For our methods, we tune the momentum parameter $\alpha$ over $\{0.01,0.03,0.05,0.1,0.3\}$ and choose the inner-loop length $N$ for PMVR-v2 from $\{10,50,100\}$.
For the remaining methods, we try the hyperparameters suggested in their original papers when available, and also perform a grid search to select the best parameters.
The learning rates are tuned over $\{0.001,0.005, 0.01,0.05,0.1\}$. All experiments are conducted on a personal laptop.

\subsection{Matrix Optimization with Low-Rank Constraints}
Following prior work on multi-level optimization~\citep{NEURIPS2022_7e16384b}, we conduct experiments on matrix optimization with low-rank constraints. Specifically, we consider the matrix-valued single-index model~\citep{pmlr-v70-yang17a} with a low-rank constraint:
\begin{gather*}
    y=\left|\langle A, B^{\star}\rangle_{F}\right|^{2} + \epsilon, \quad
    \text{rank}(B^{\star}) \leq s,
\end{gather*}
where $A,B \in \R^{m\times n}$, $\epsilon \sim \mathcal{N}\left(0, \sigma^{2}\right),\langle\cdot, \cdot\rangle_{F}$ denotes the Frobenius inner product, and $s$ is a positive integer smaller than both $m$ and $n$. 
To recover a low-rank matrix $B^\star$ given $A$ and $y$, we minimize the mean squared loss function under a nuclear norm constraint:
\begin{align*}
    \min_{B} F(B)&=\E\left[\left(y-\left|\langle A, B\rangle_{F}\right|^{2}\right)^{2}\right] \quad \text{s.t. } \ \|B\|_{\star} \leq s.
\end{align*}
In this problem, projection onto the nuclear-norm ball typically requires a full singular value decomposition~(SVD), whereas each Frank-Wolfe linear minimization step only requires computing the singular vector pair corresponding to the largest singular value, which is substantially cheaper~\citep{pmlr-v28-jaggi13}.

In line with the setup in~\citet{NEURIPS2022_7e16384b}, we set $B^{\star}= v v^{\top} /\left\|v v^{\top}\right\|_{\star}$, where $v$ is randomly sampled from the normal distribution via \texttt{numpy.random.normal()}. We also generate matrix $A$ by $A=I+E$, with $E_{i, j} \stackrel{i . i . d .}{\sim} \mathcal{N}\left(0, 0.3\right)$. 
Figure~\ref{fig:2} reports the Frank-Wolfe gap and the gradient mapping versus running time. All curves are averaged over 50 runs. 
As can be seen, our methods demonstrate a more rapid decrease in both criteria compared to other approaches, especially for the gradient mapping criterion, demonstrating the superiority of our proposed methods.
\begin{figure*}[t]
	\begin{center}
\includegraphics[width=0.37\textwidth]{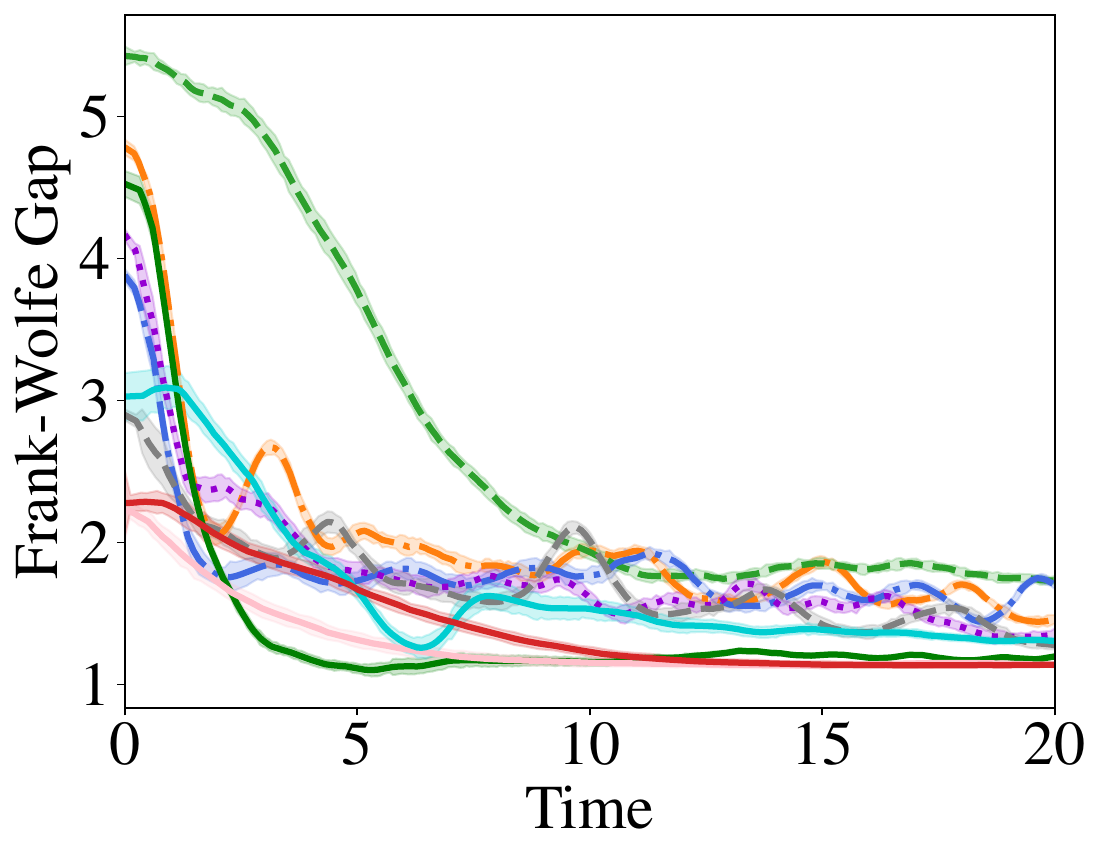}
			\hspace{0.5in}
   \includegraphics[width=0.37\textwidth]{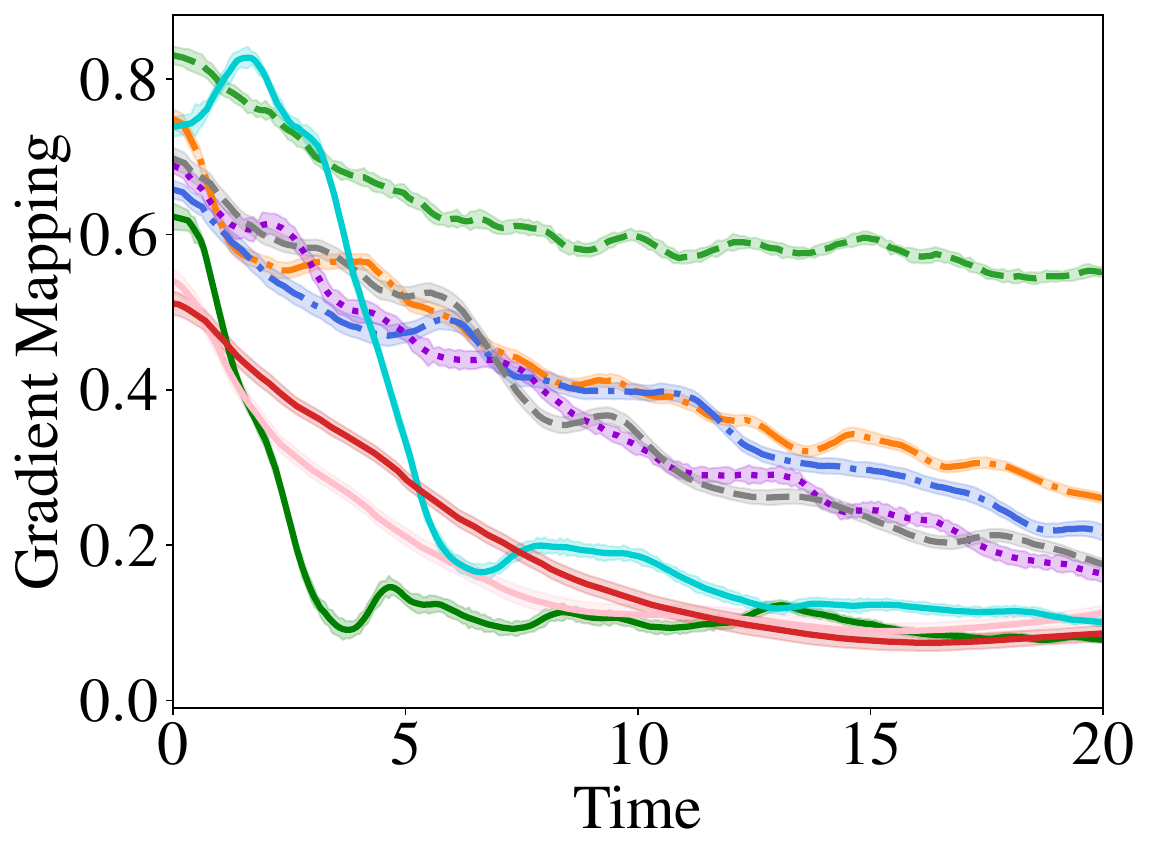}
   \vspace{-0.1in}
		\subfigure{
\includegraphics[width=0.99\textwidth]{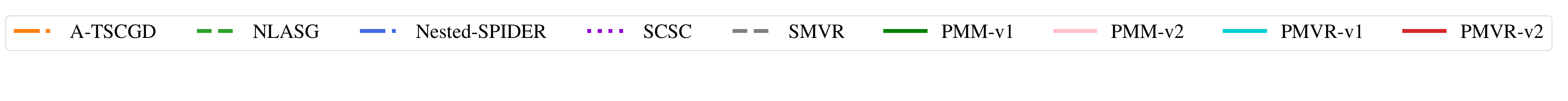}
		}
        \vspace{-0.35in}
		\caption{Results for matrix optimization with low-rank constraints.}
        \vspace{-0.2in}
		\label{fig:2}
	\end{center}
\end{figure*}
\subsection{Mean-variance Risk-averse Optimization}
We next consider mean-variance risk-averse portfolio optimization~\citep{Shapiro2009LecturesOS}, a commonly used benchmark for multi-level methods~\citep{Yang2019MultilevelSG,balasubramanian2020stochastic,Zhang2021MultiLevelCS,chen2021solving}. 
\begin{figure*}[t]
	\begin{center}
		\subfigure{
			\includegraphics[width=0.3\textwidth]{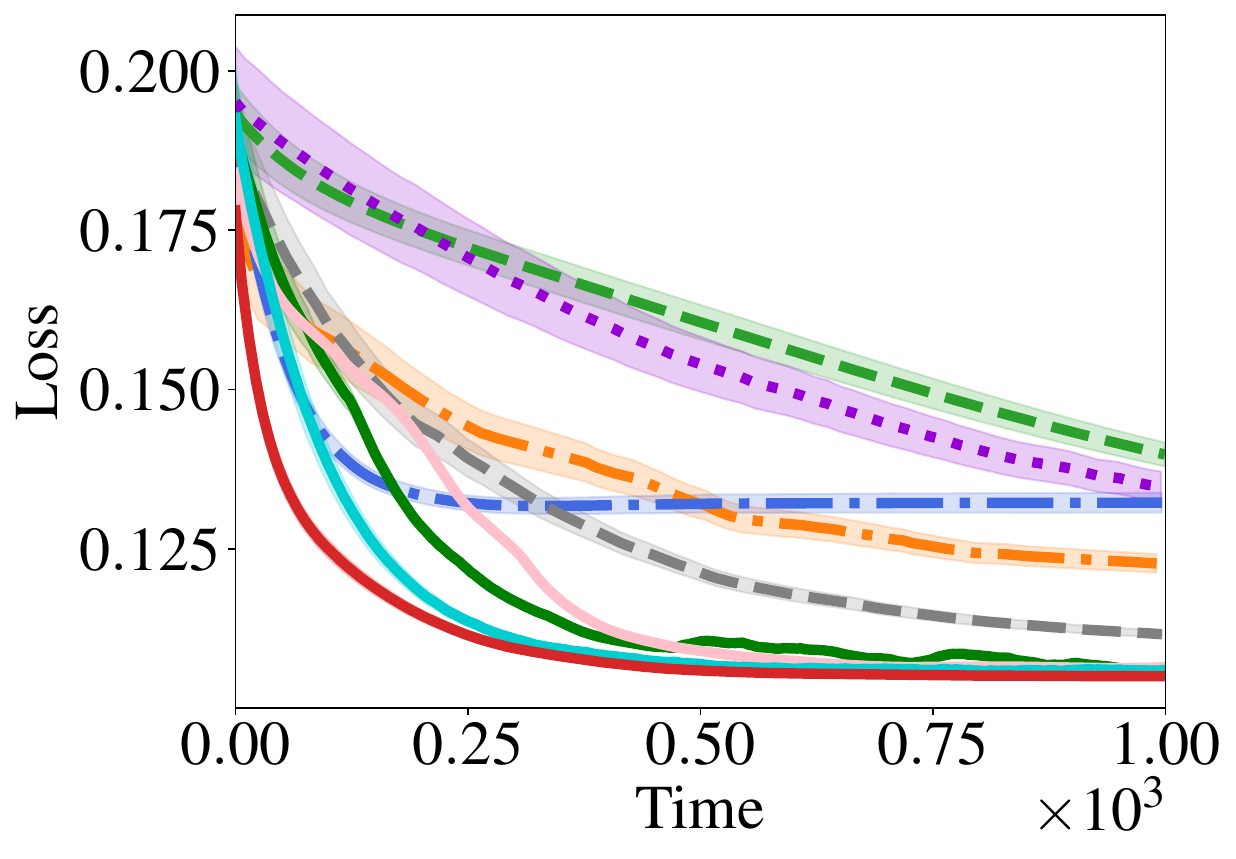}}
   \setcounter{subfigure}{0}
   \subfigure[ Industry-10]{
			\includegraphics[width=0.3\textwidth]{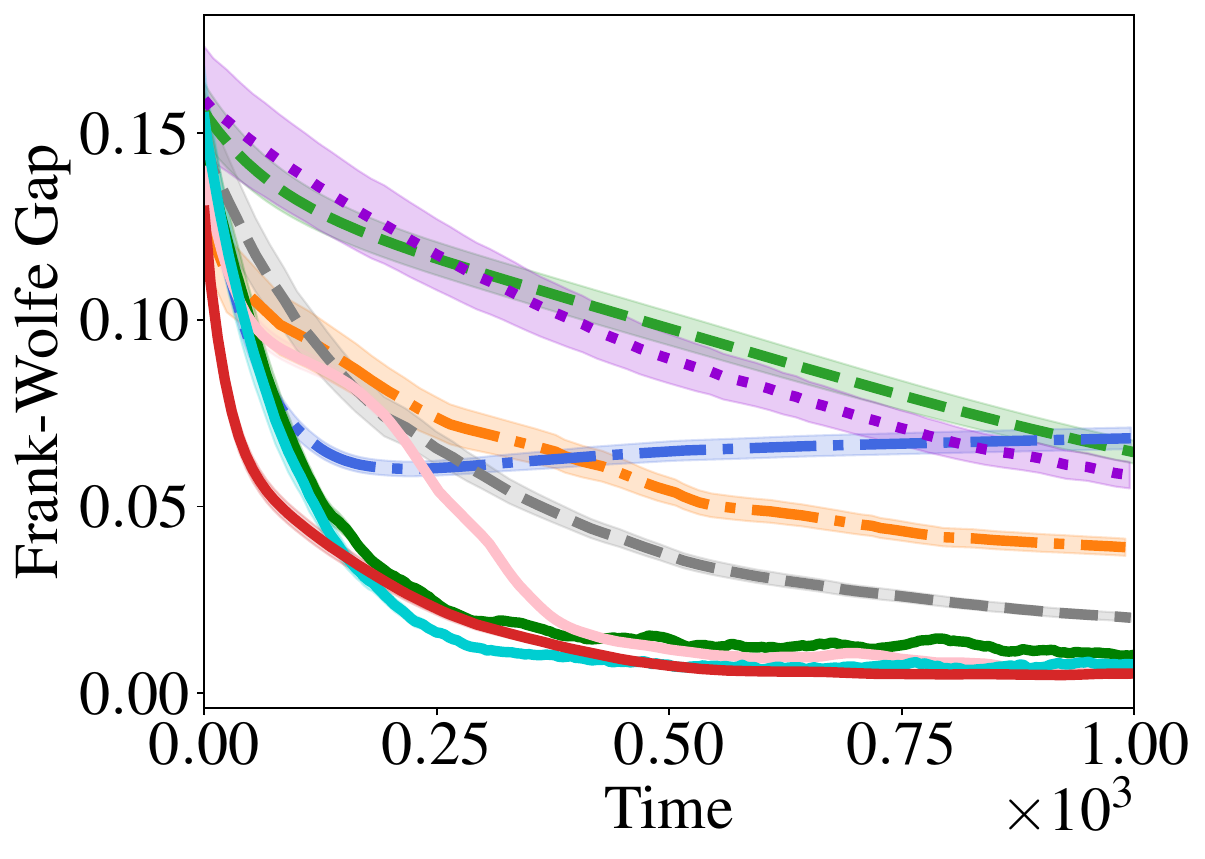}}
   \subfigure{
			\includegraphics[width=0.3\textwidth]{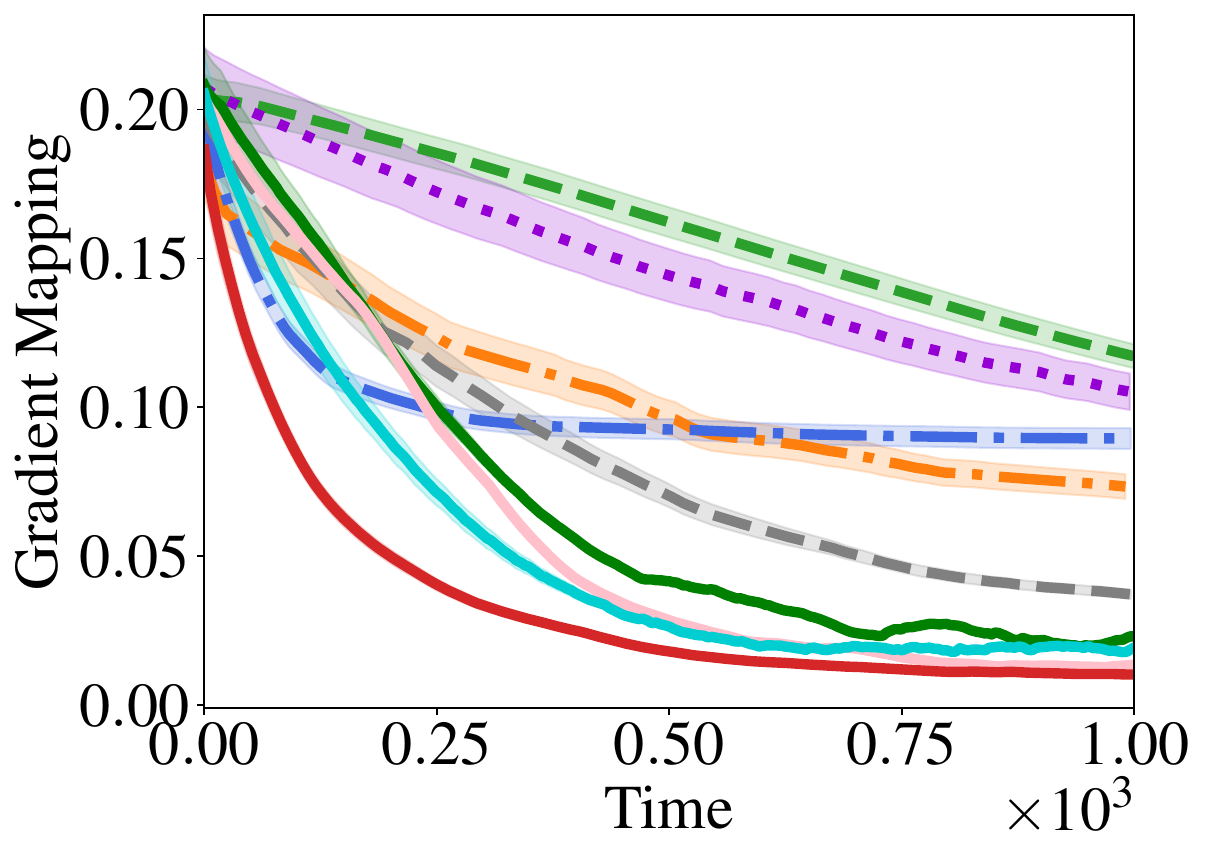}
		}
		%\vskip -0.05in
  \setcounter{subfigure}{0}
		\subfigure{
			\includegraphics[width=0.3\textwidth]{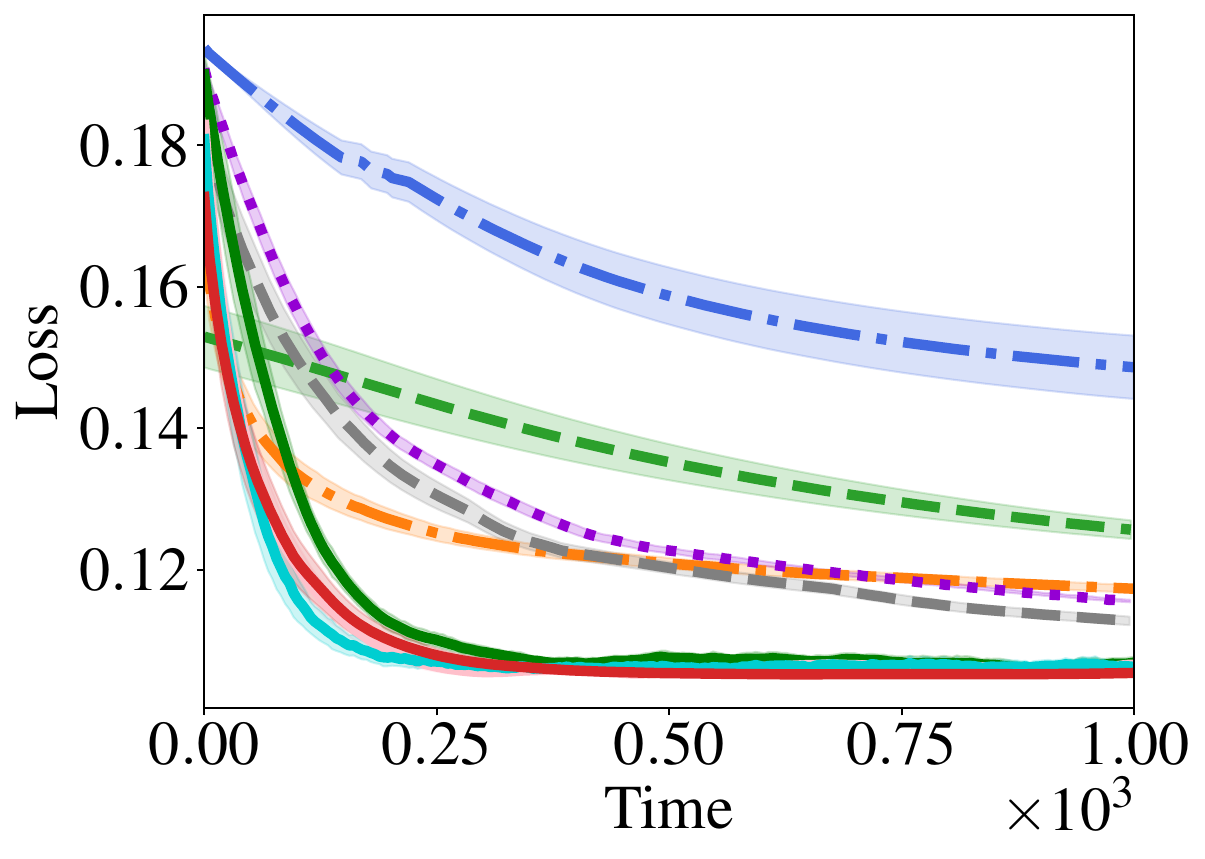}}
   \subfigure[ Industry-12]{
			\includegraphics[width=0.3\textwidth]{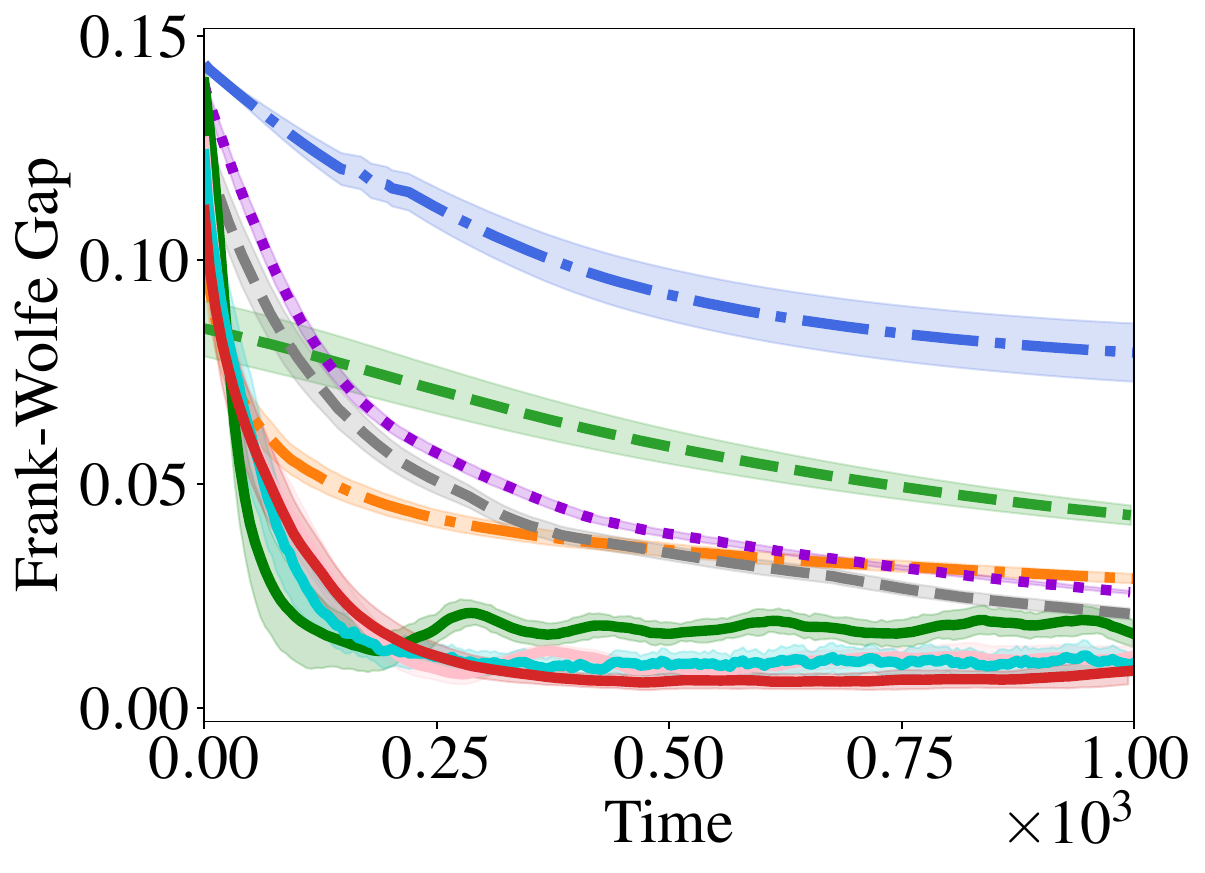}}
   \subfigure{
			\includegraphics[width=0.3\textwidth]{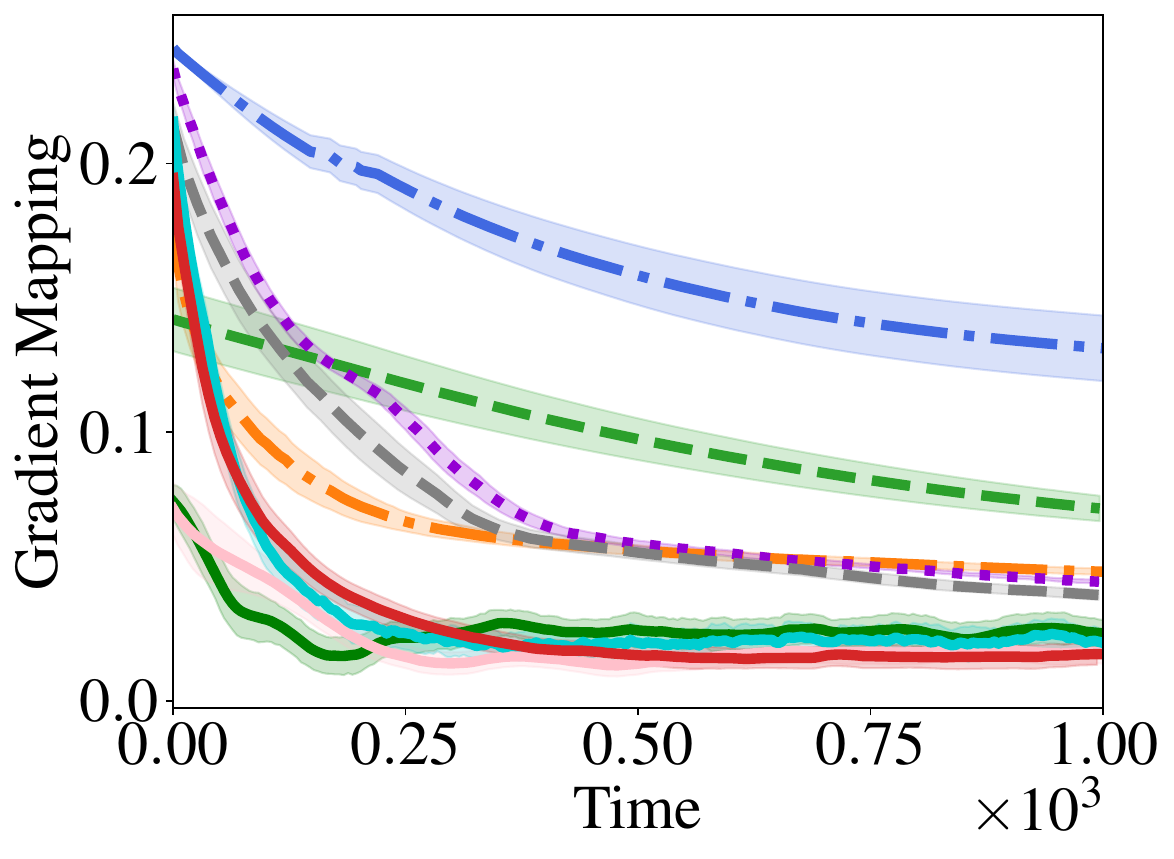}
		}

   \setcounter{subfigure}{2}
     \includegraphics[width=0.3\textwidth]{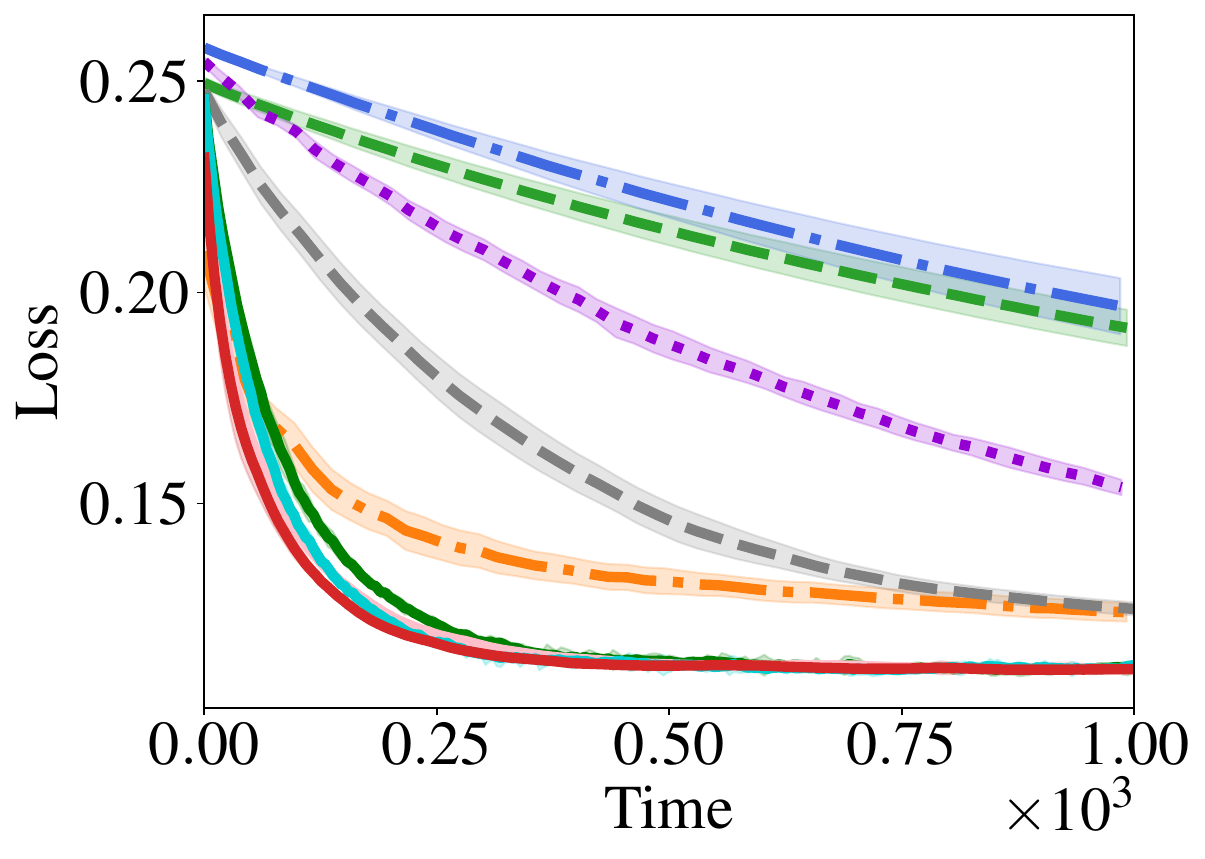}
   \subfigure[ Industry-17]{
			\includegraphics[width=0.3\textwidth]{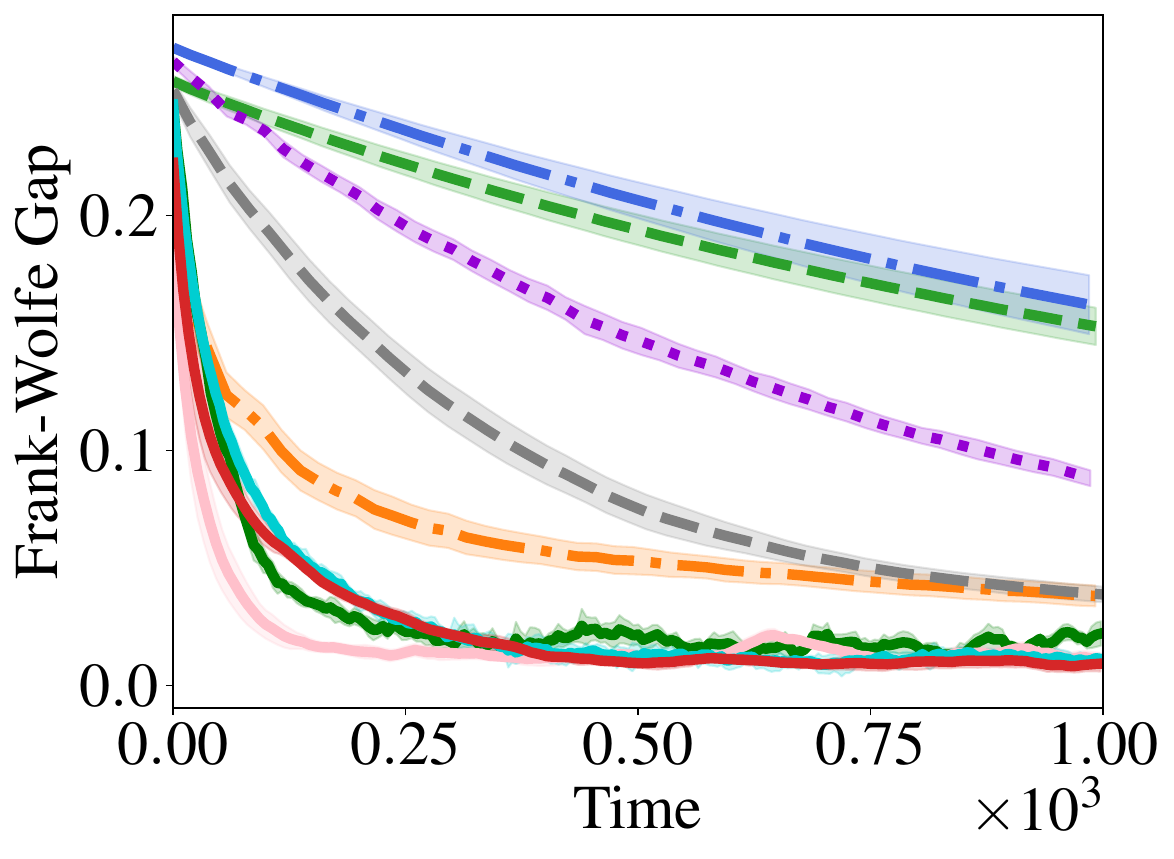}}
   \subfigure{
			\includegraphics[width=0.3\textwidth]{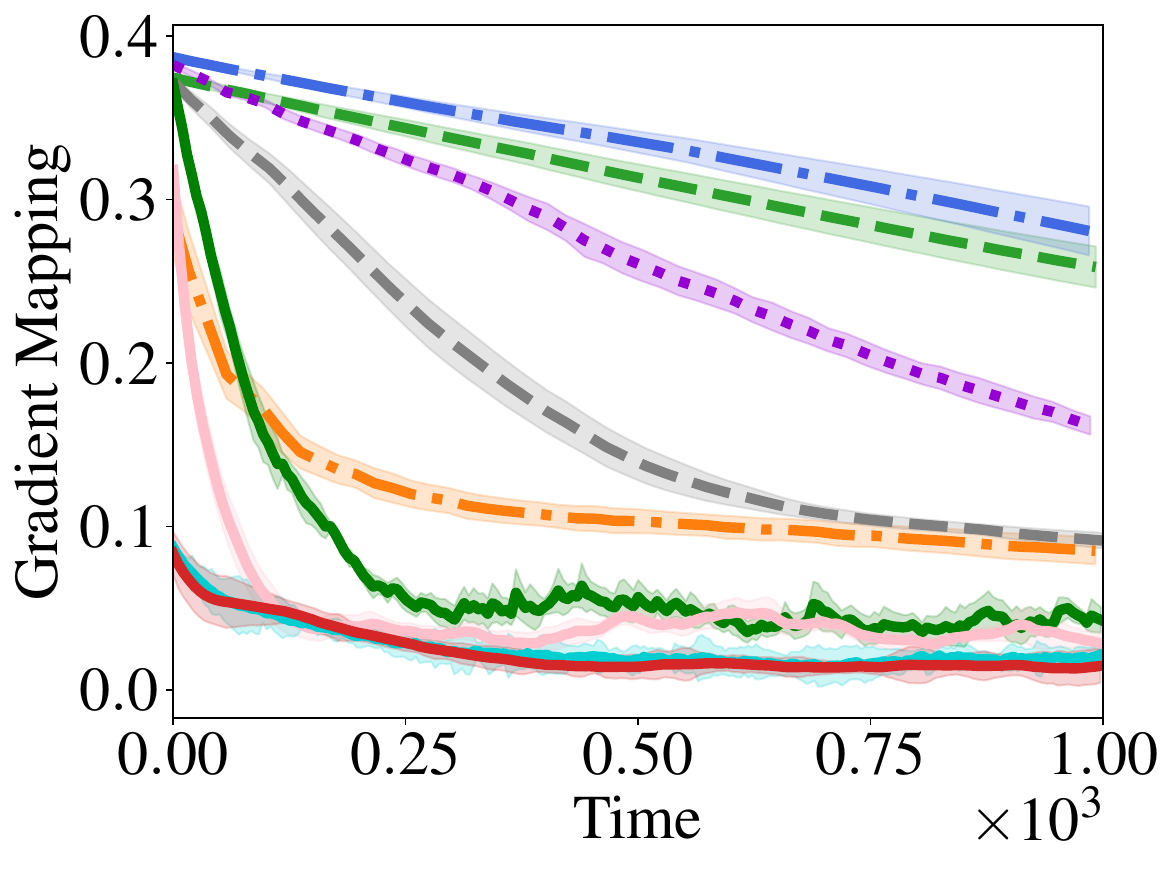}
		}
         \vspace{-0.1in}
		\subfigure{
			\includegraphics[width=0.99\textwidth]{./legend.pdf}
		}
  \vspace{-0.35in}
		\caption{Results for mean-variance risk-averse portfolio optimization.}
        \vspace{-0.3in}
		\label{fig:1}
	\end{center}
\end{figure*}
Suppose we have $d$ assets to invest over periods $l=\{1, \ldots, L\}$, and $\r_{l} \in \mathbb{R}^{d}$ represents the payoff of $d$ assets at time step $l$. The goal is to maximize investment returns and minimize risk simultaneously. A suitable approach for this purpose is the mean-variance risk-averse optimization model, where risk is defined as the variance. The problem can be written as
\begin{gather*}
\min _{\x \in \mathcal{X}} F(\x) =  -\frac{1}{L} \sum_{l=1}^{L}\left\langle \r_{l}, \x\right\rangle+ \frac{\lambda}{L} \sum_{l=1}^{L}\left(\left\langle \r_{l}, \x\right\rangle-\left\langle \bar{\r}, \x\right\rangle\right)^{2},
\end{gather*}
where $\bar{\r} = \frac{1}{L}\sum_{l=1}^{L} \r_{l}$, and the decision variable $\x$ denotes the investment quantities in $d$ assets. 
The domain $\X$ is a simplex, ensuring that $\Norm{\x}_{1}=1$ for any $\x \in \X$. This problem can be modeled as a stochastic two-level constrained compositional optimization problem,  with each layer expressed as 
\begin{gather*}
     f_{1}(\x)=\left(-\frac{1}{L} \sum_{l=1}^{L}\langle \r_{l}, \x\rangle, \x\right), \quad
     f_{2}(\y_1, \y_2)=\y_1 + \frac{\lambda}{L} \sum_{l=1}^{L}\left(\left\langle \r_{l}, \y_2\right\rangle+\y_1\right)^{2}, 
\end{gather*}
where $f_1(\cdot)$ is the inner function and $f_2(\cdot)$ is the outer function such that $F(\x)=f_2(f_1(\x))$.
\begin{figure*}[!ht]
	\begin{center}
		\subfigure{
			\includegraphics[width=0.3\textwidth]{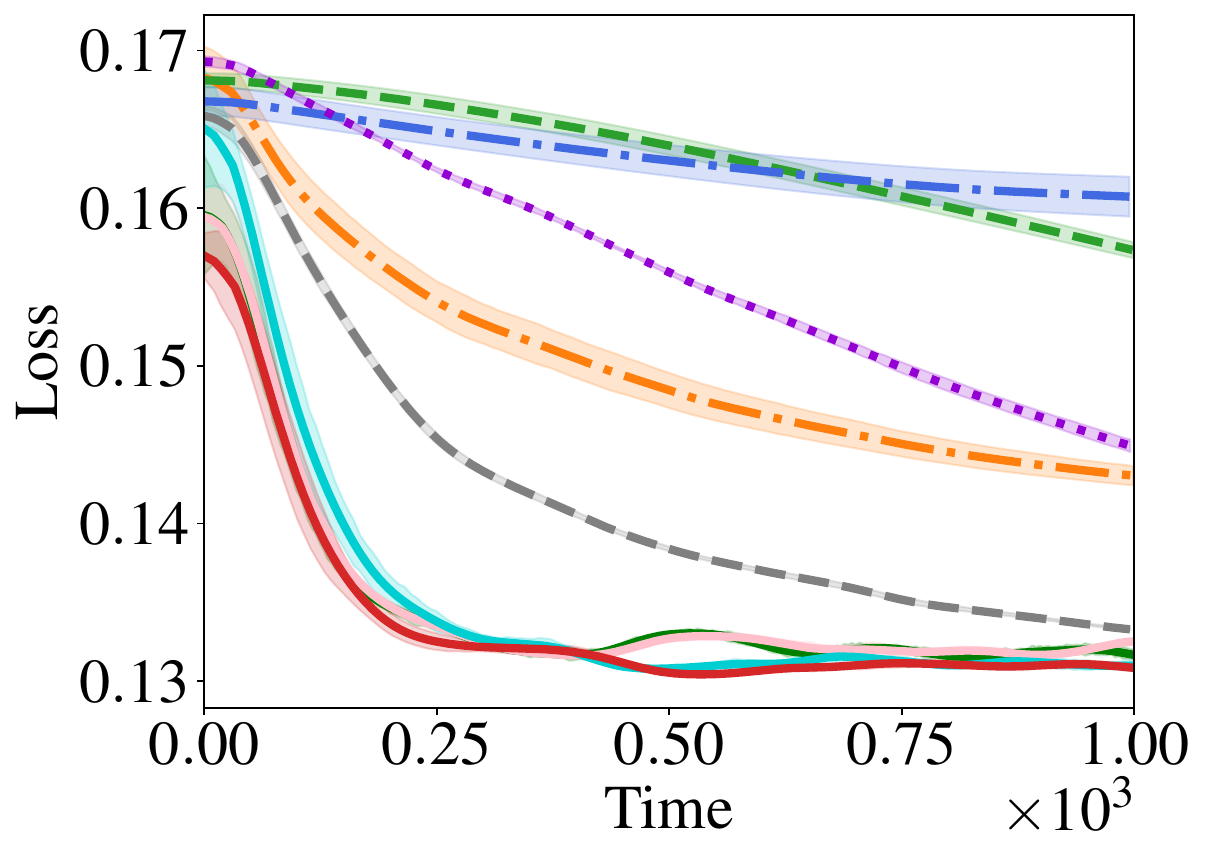}}
   \setcounter{subfigure}{0}
   \subfigure[ Industry-10]{
			\includegraphics[width=0.3\textwidth]{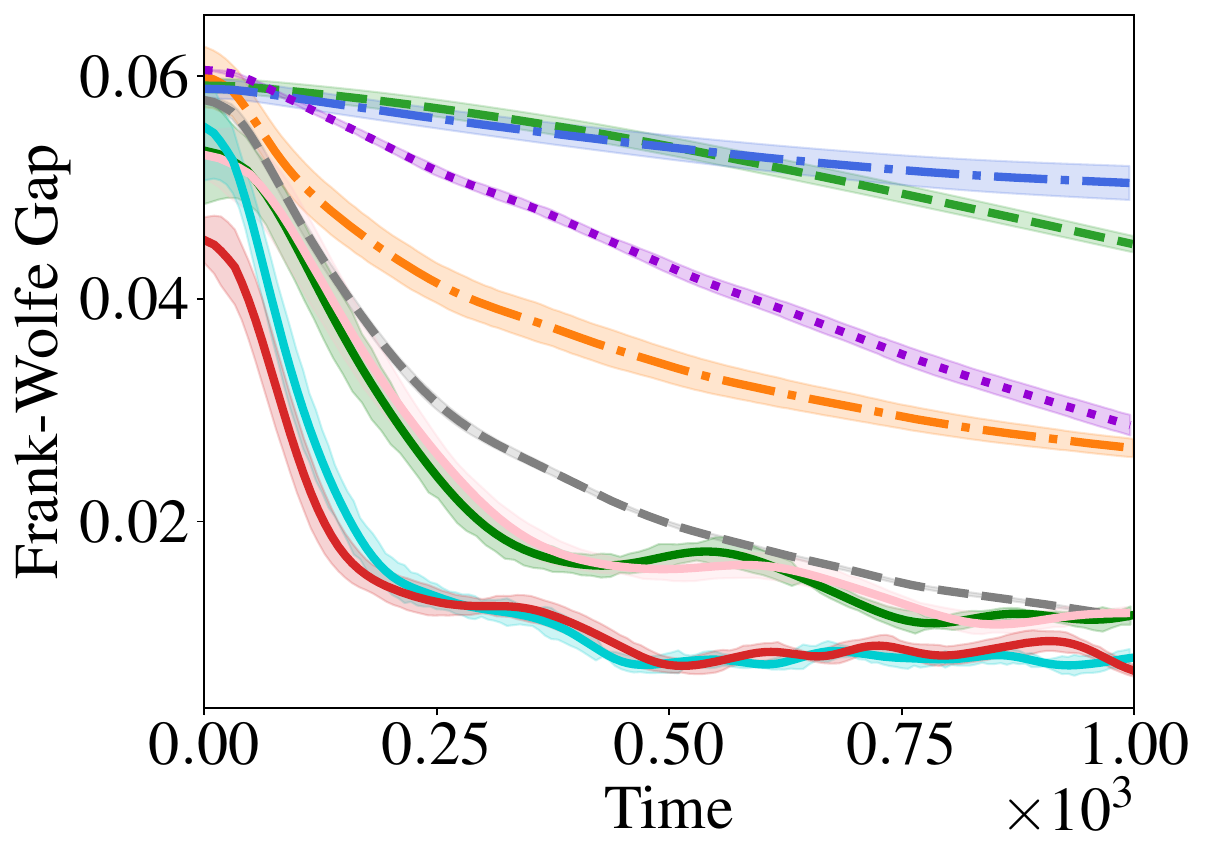}}
   \subfigure{
			\includegraphics[width=0.3\textwidth]{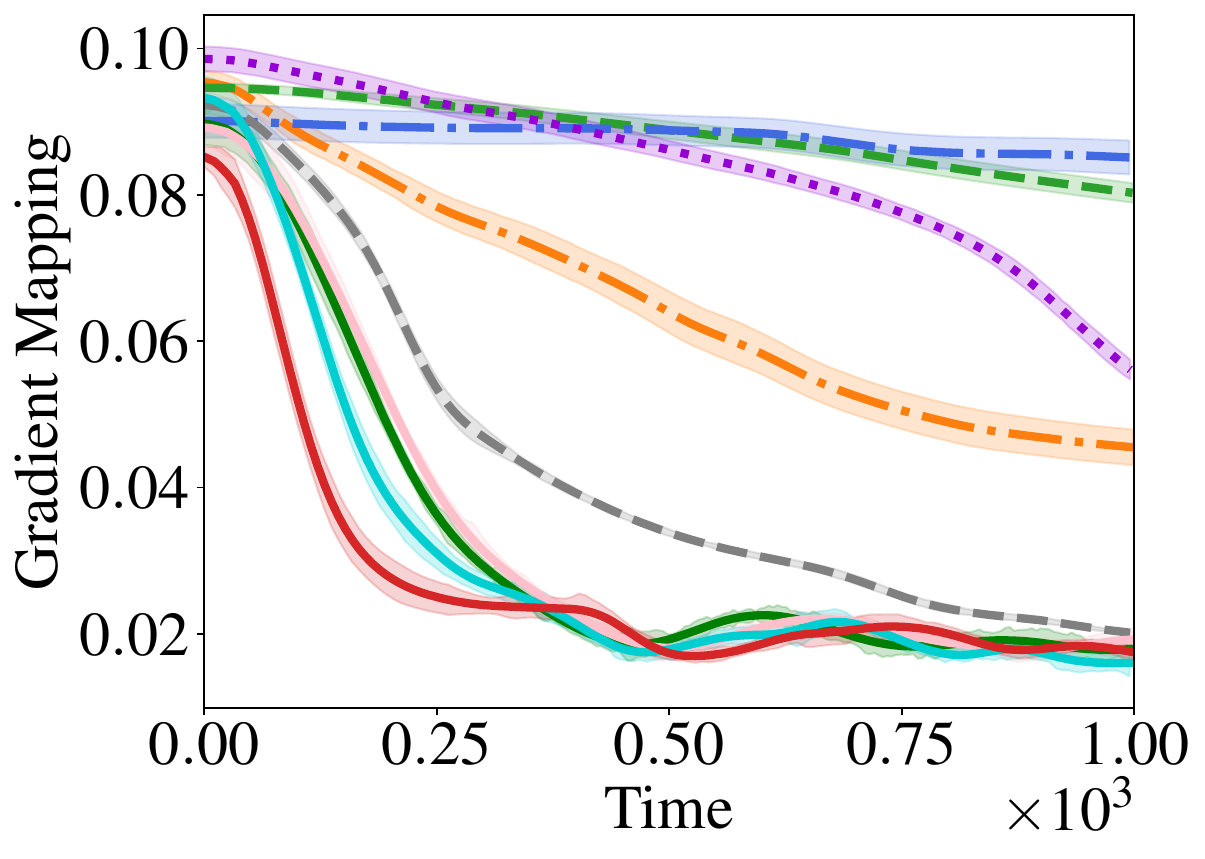}
		}
		%\vskip -0.05in
  \setcounter{subfigure}{0}
		\subfigure{
			\includegraphics[width=0.3\textwidth]{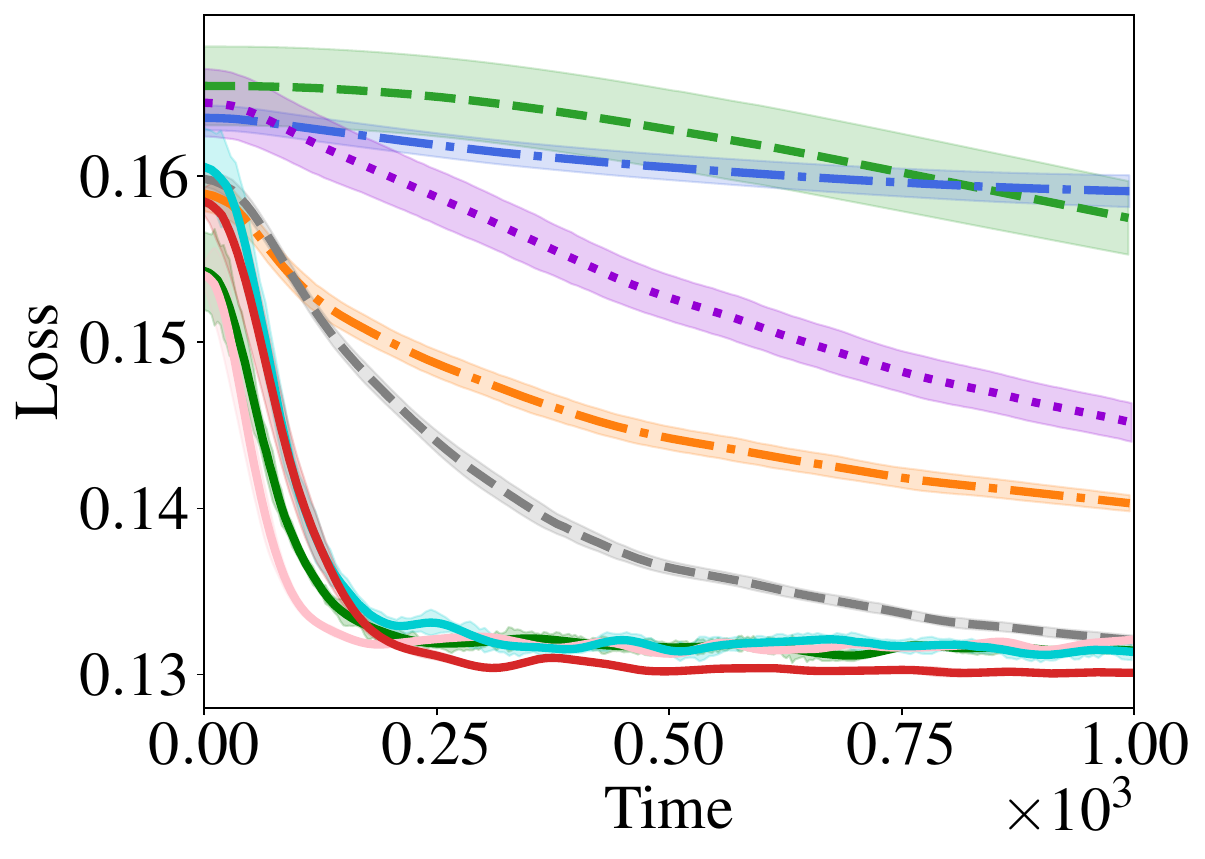}}
   \subfigure[ Industry-12]{
			\includegraphics[width=0.3\textwidth]{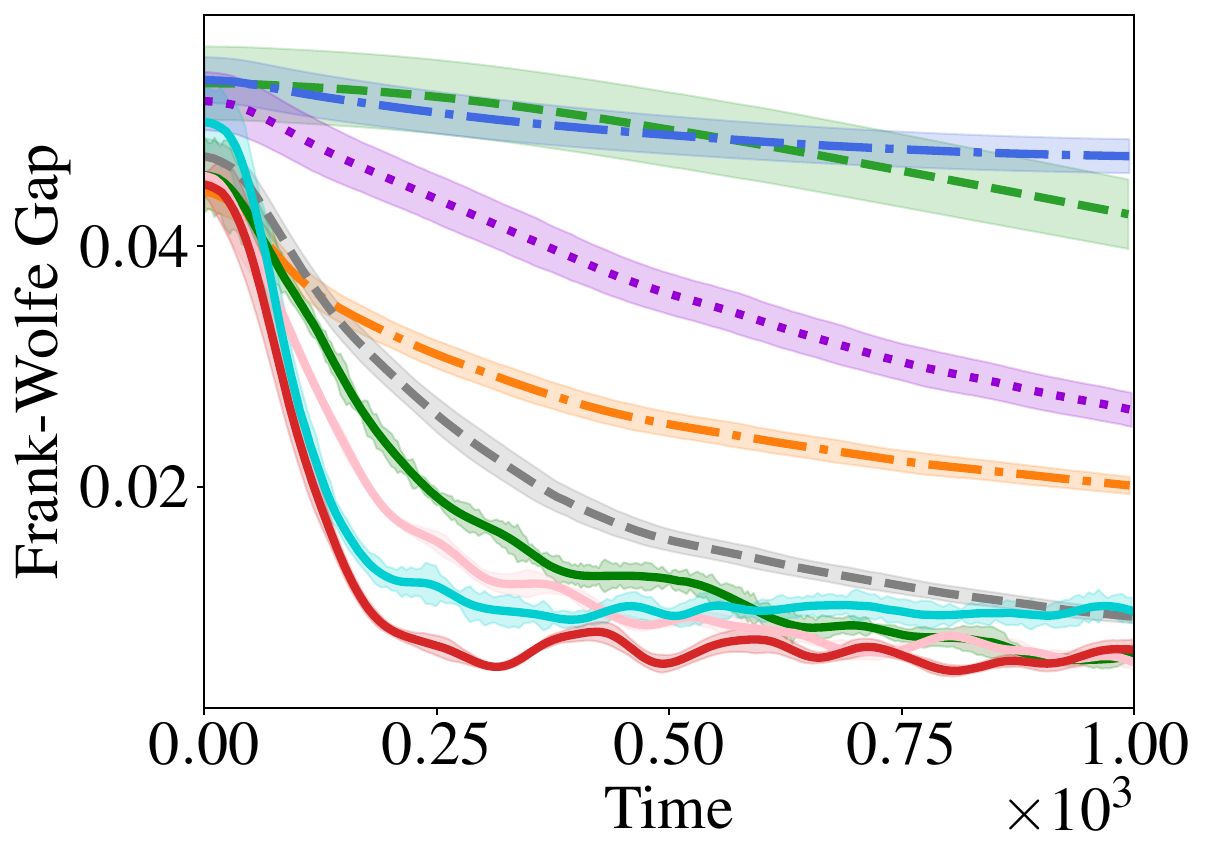}}
   \subfigure{
			\includegraphics[width=0.3\textwidth]{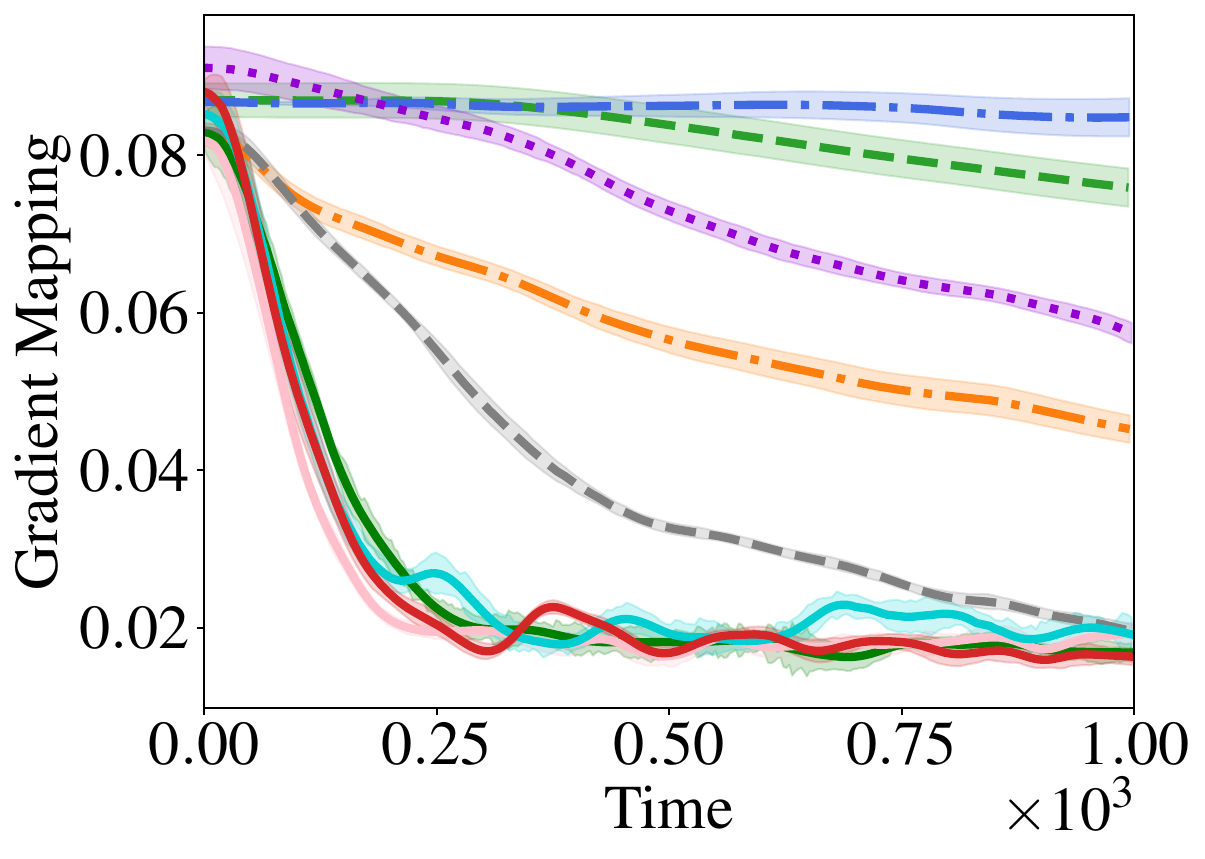}
		}
    \setcounter{subfigure}{1}
		\subfigure{
			\includegraphics[width=0.3\textwidth]{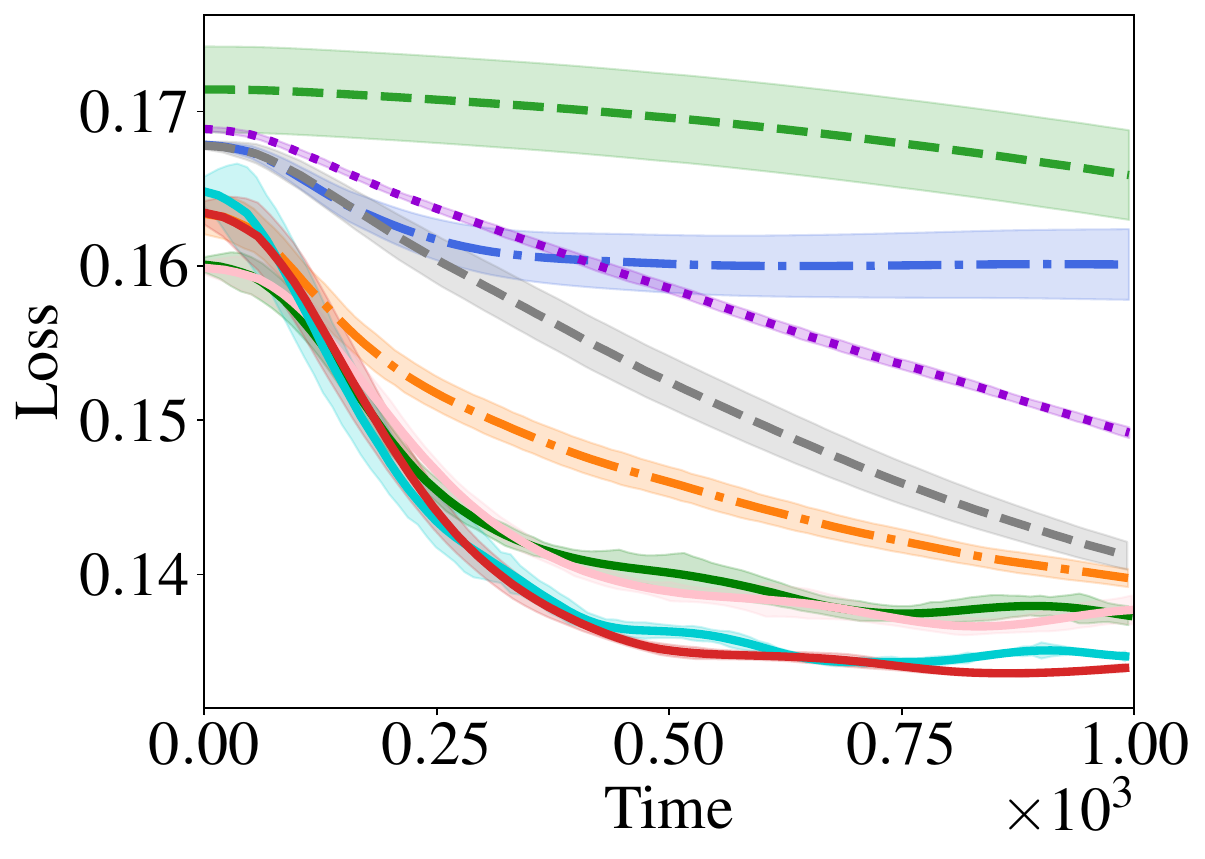}}
   \subfigure[ Industry-17]{
			\includegraphics[width=0.3\textwidth]{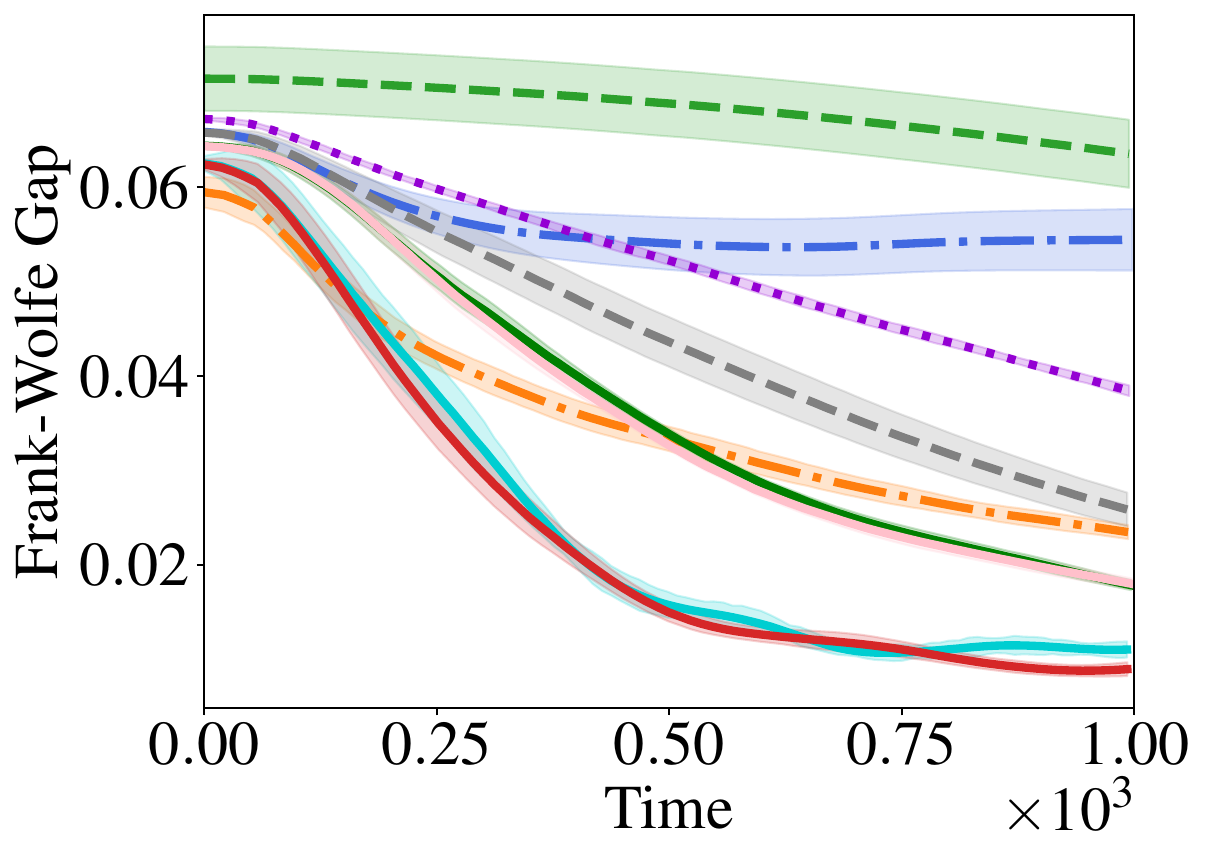}}
   \subfigure{
			\includegraphics[width=0.3\textwidth]{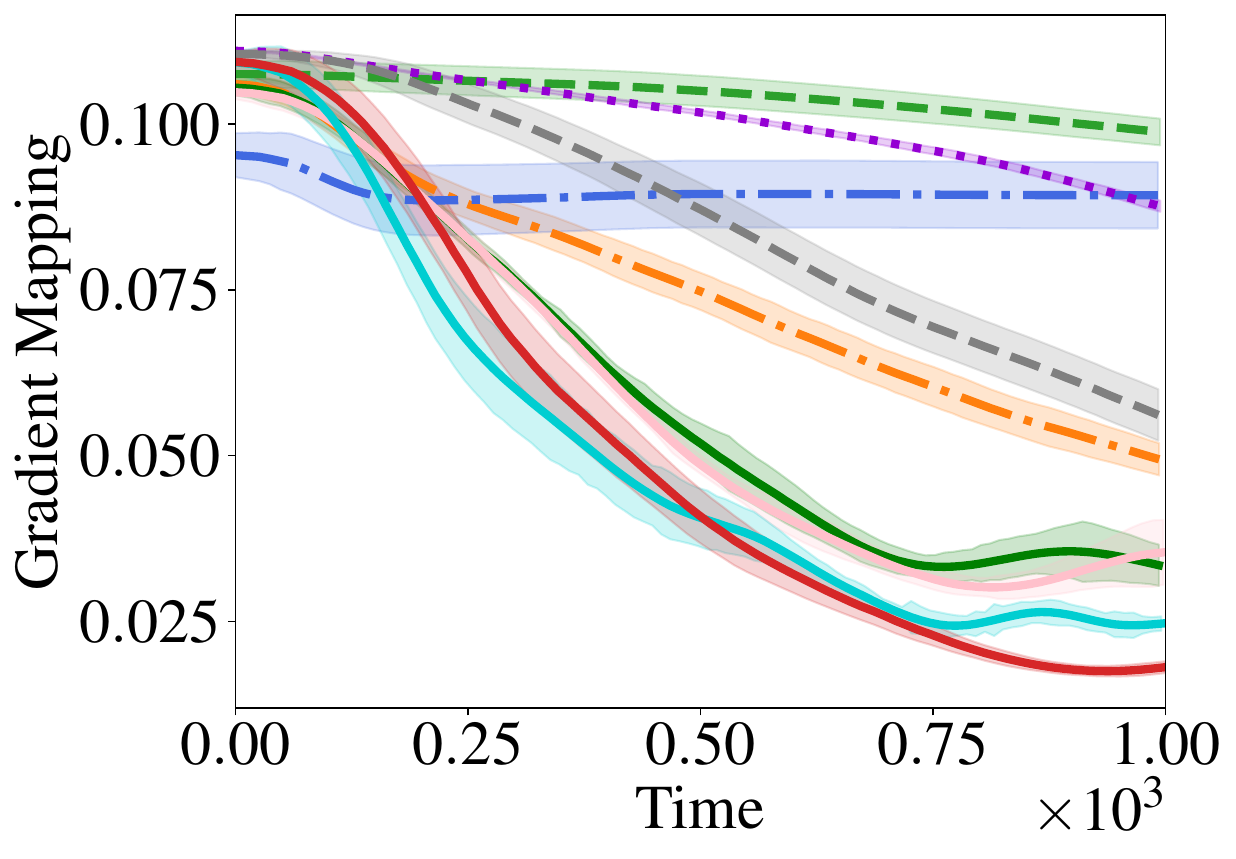}
		}
         \vspace{-0.1in}
		\subfigure{
			\includegraphics[width=0.99\textwidth]{./legend.pdf}
		}
         \vspace{-0.3in}
		\caption{Results for mean-deviation risk-averse portfolio optimization.}
		\label{fig:3}
	\end{center}
\end{figure*}

For experimental validation, we use real-world datasets Industry-10, Industry-12, and Industry-17 from the Kenneth R. French Data Library\footnote{https://mba.tuck.dartmouth.edu/pages/faculty/ken.french/}. These datasets contain payoffs for 10, 12, and 17 industrial assets over 25,105 consecutive periods. 
For projection-based methods, we implement simplex projection using a well-known efficient projection method~\cite{Duchi2008EfficientPO}. 
Figure~\ref{fig:1} reports the objective value $F(\x)$, the Frank-Wolfe gap, and the gradient mapping, averaged over 50 runs.
It is observed that our proposed methods~(PMVR-v1,v2, and PMM-v1,v2) tend to converge more rapidly compared to other algorithms in  all tasks. The loss value, Frank-Wolfe gap, and gradient mapping of our PMVR-v2 decrease most quickly in both datasets, validating the effectiveness of the proposed method. 
\begin{figure*}[!t]
	\begin{center}
		\subfigure{
			\includegraphics[width=0.3\textwidth]{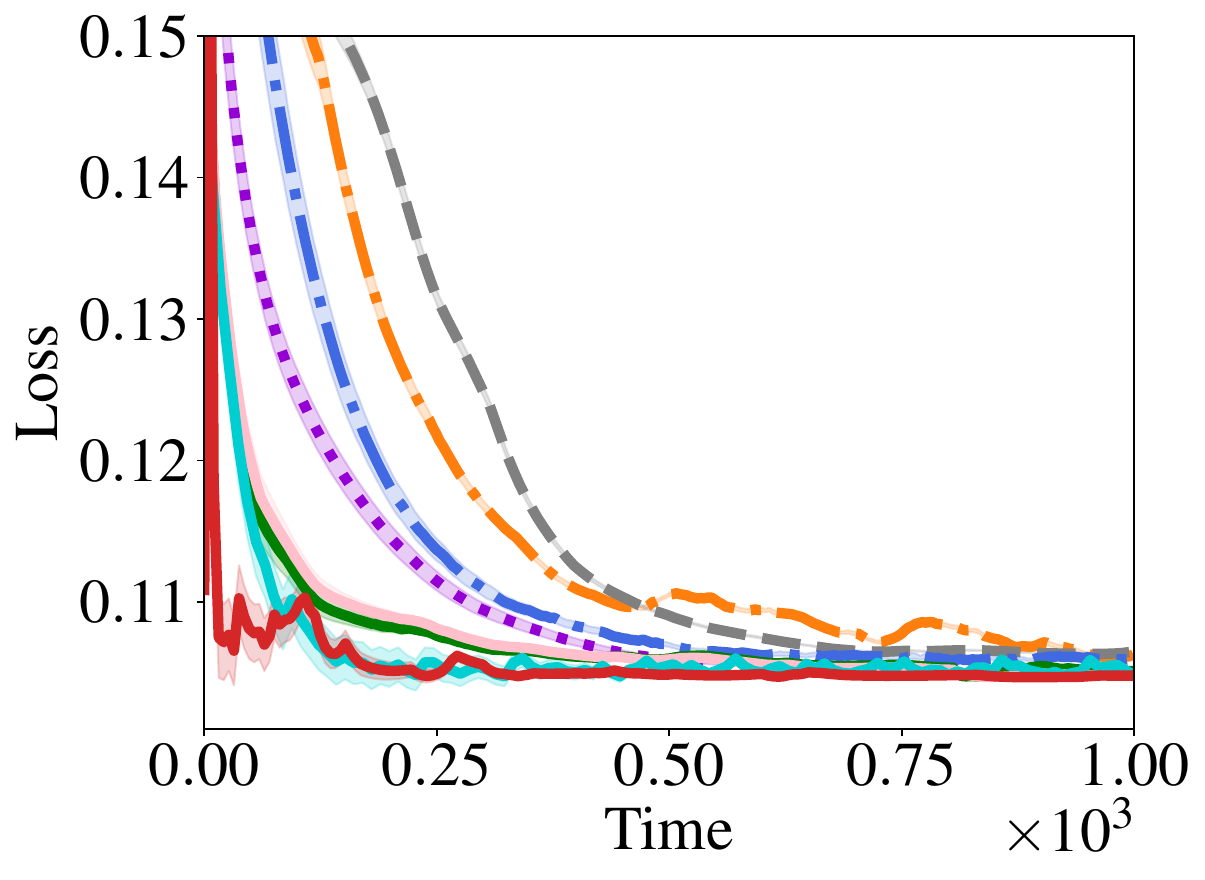}}
   \setcounter{subfigure}{0}
   \subfigure[ Industry-10]{
			\includegraphics[width=0.3\textwidth]{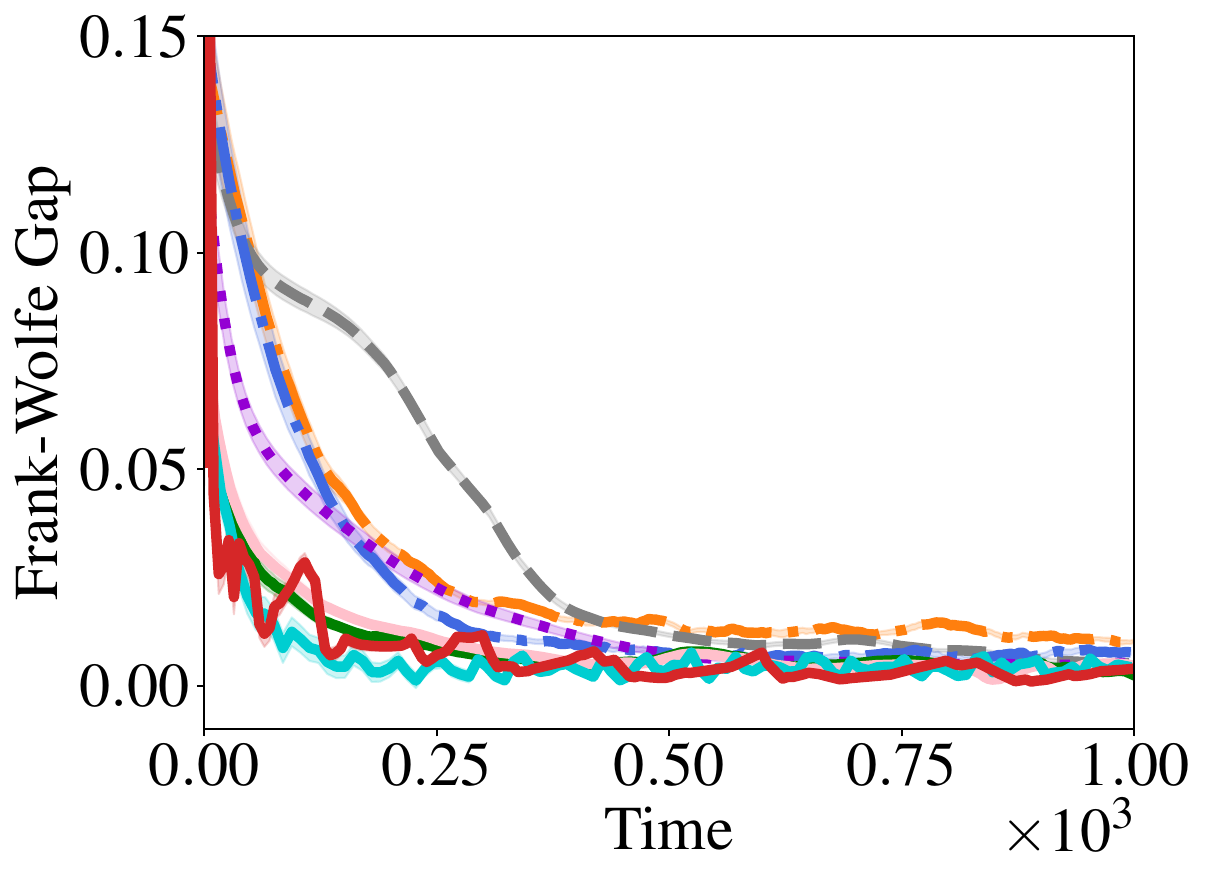}}
   \subfigure{
			\includegraphics[width=0.3\textwidth]{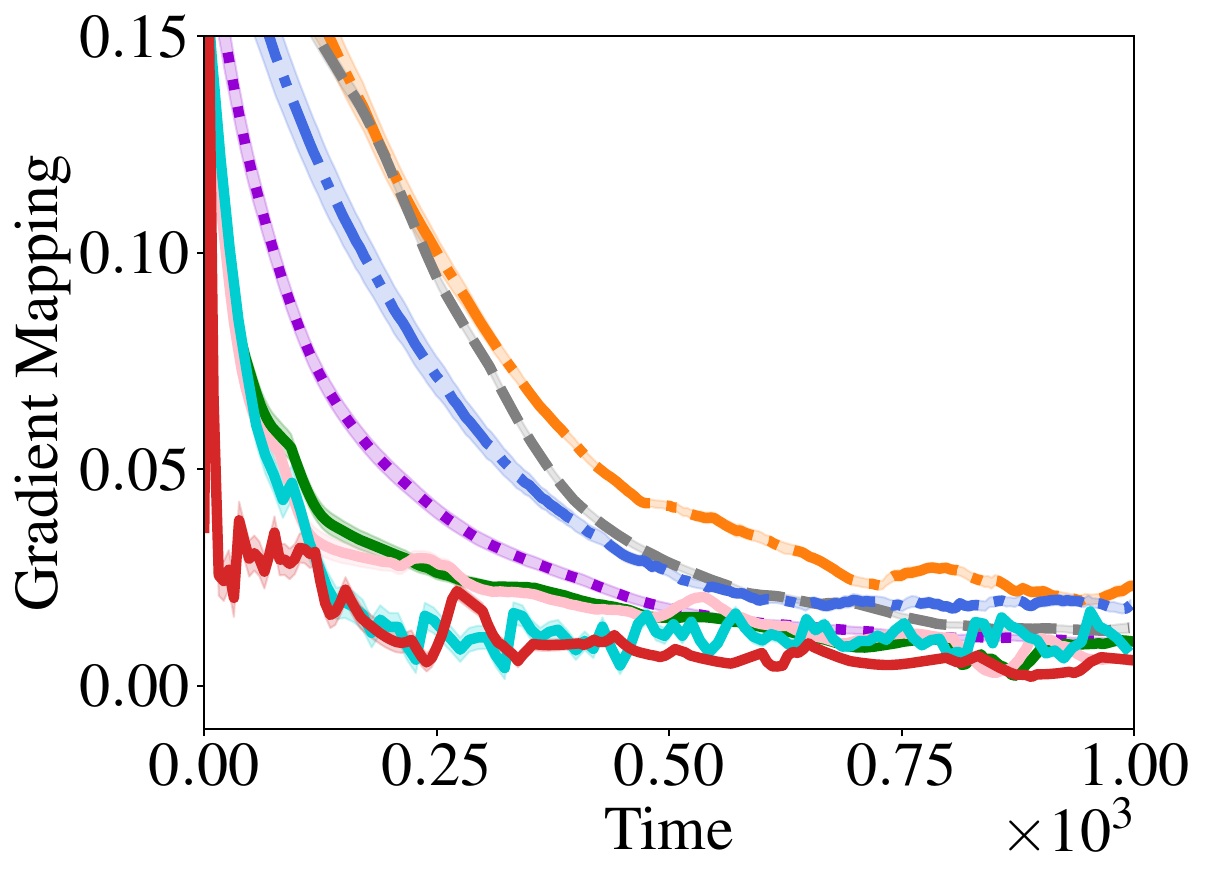}
		}
         \setcounter{subfigure}{0}
		\subfigure{
			\includegraphics[width=0.3\textwidth]{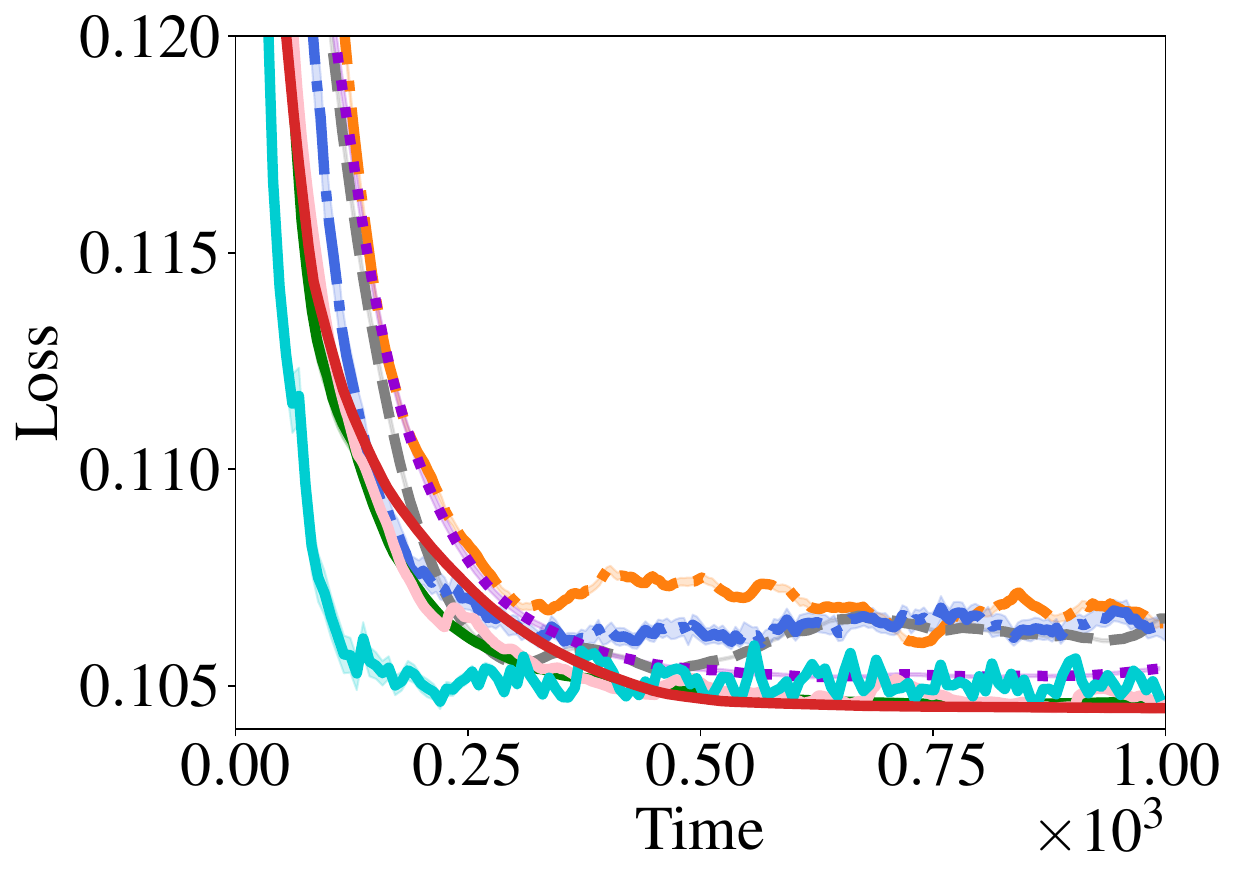}}
   \subfigure[ Industry-12]{
			\includegraphics[width=0.3\textwidth]{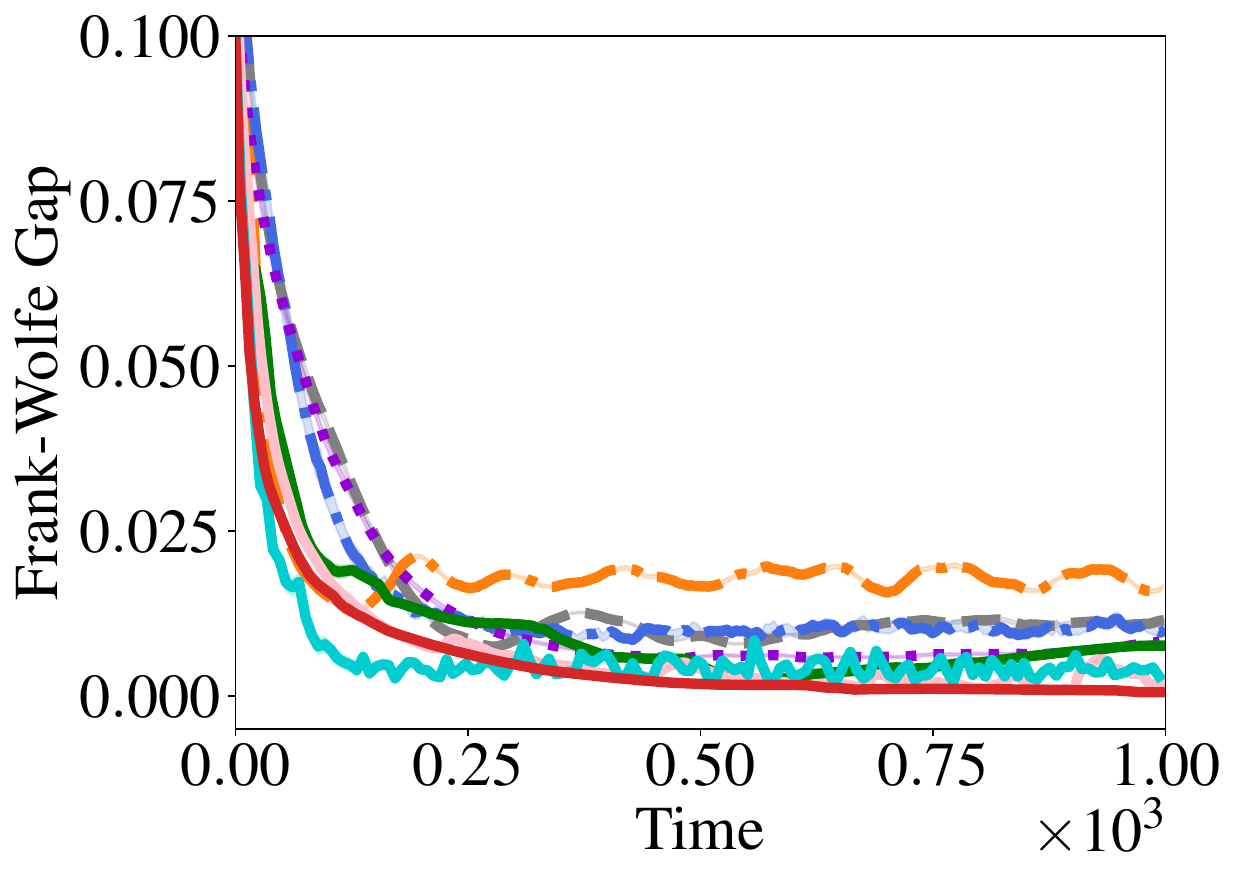}}
   \subfigure{
			\includegraphics[width=0.3\textwidth]{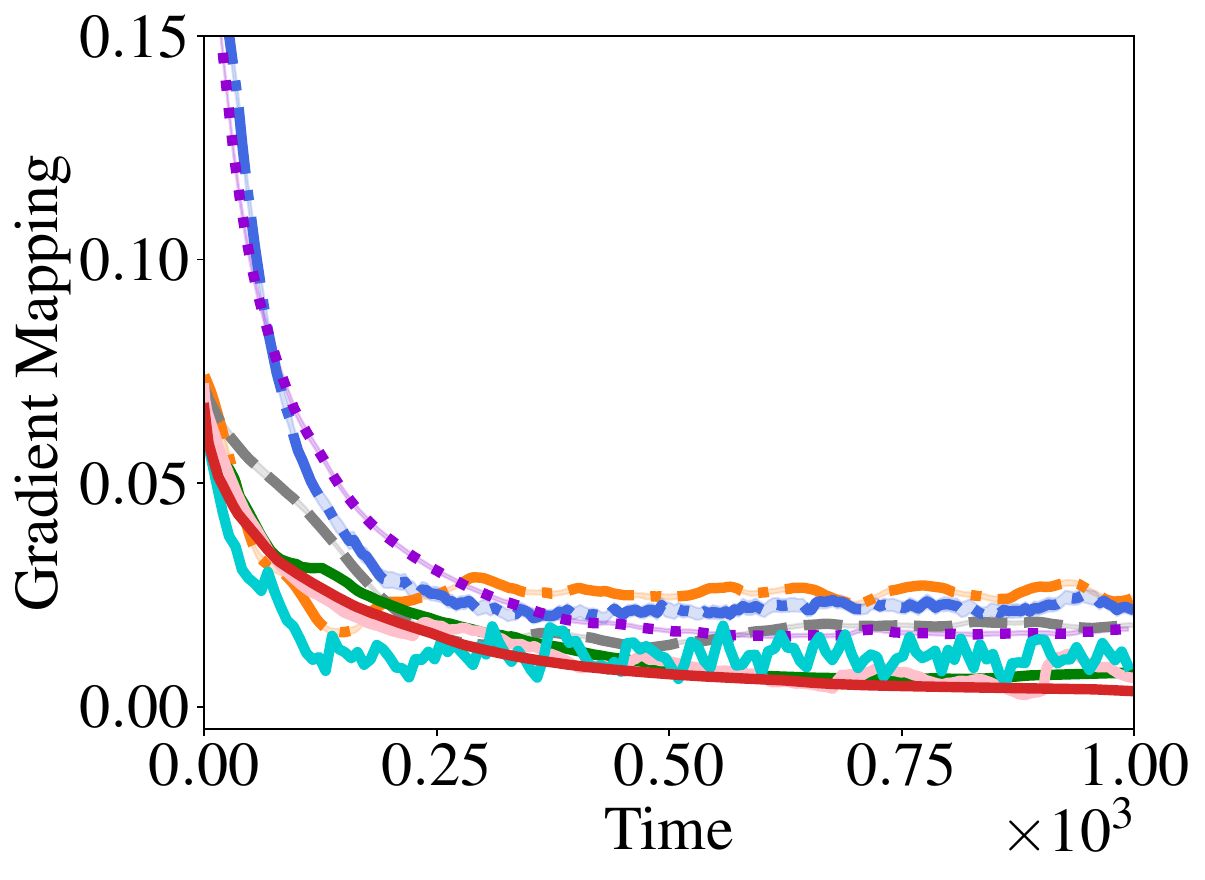}
		}
                 \setcounter{subfigure}{1}
		\subfigure{
			\includegraphics[width=0.3\textwidth]{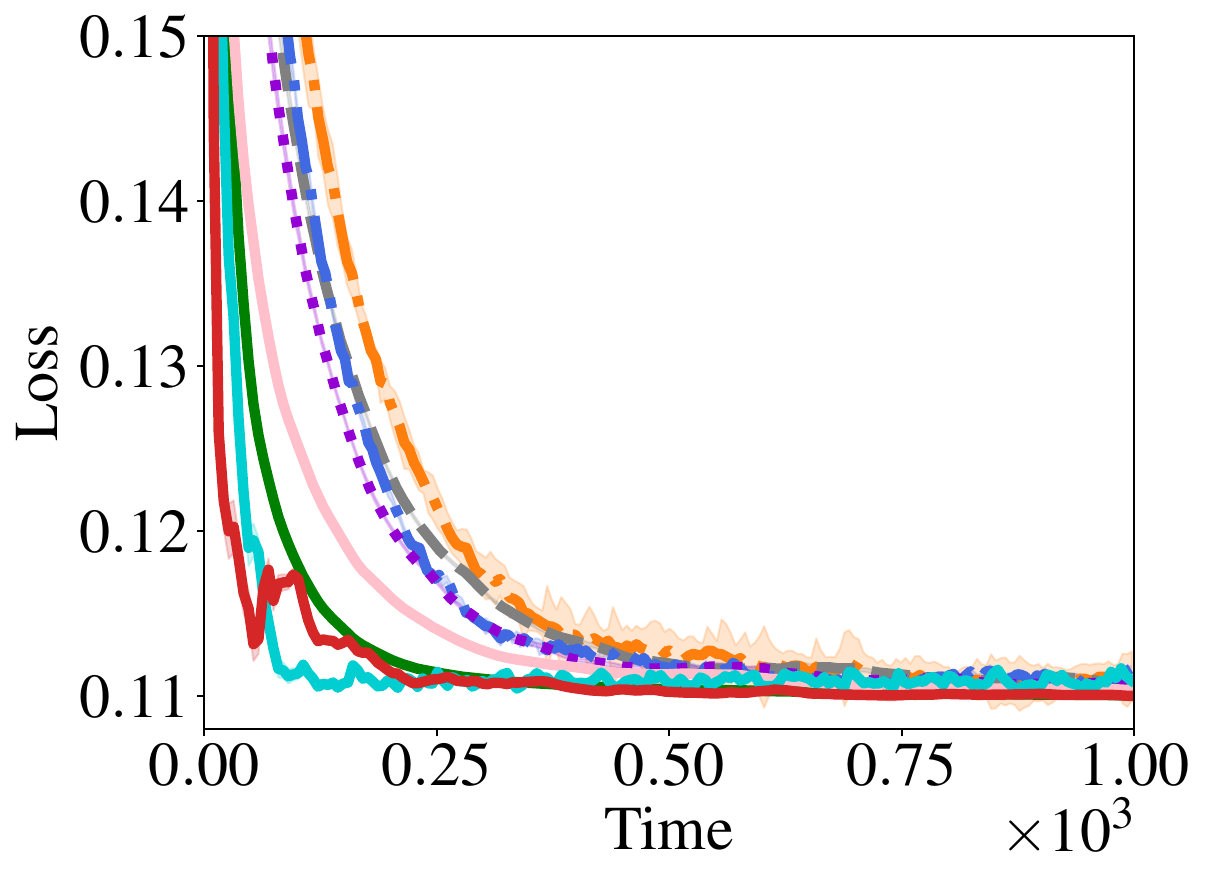}}
   \subfigure[ Industry-17]{
			\includegraphics[width=0.3\textwidth]{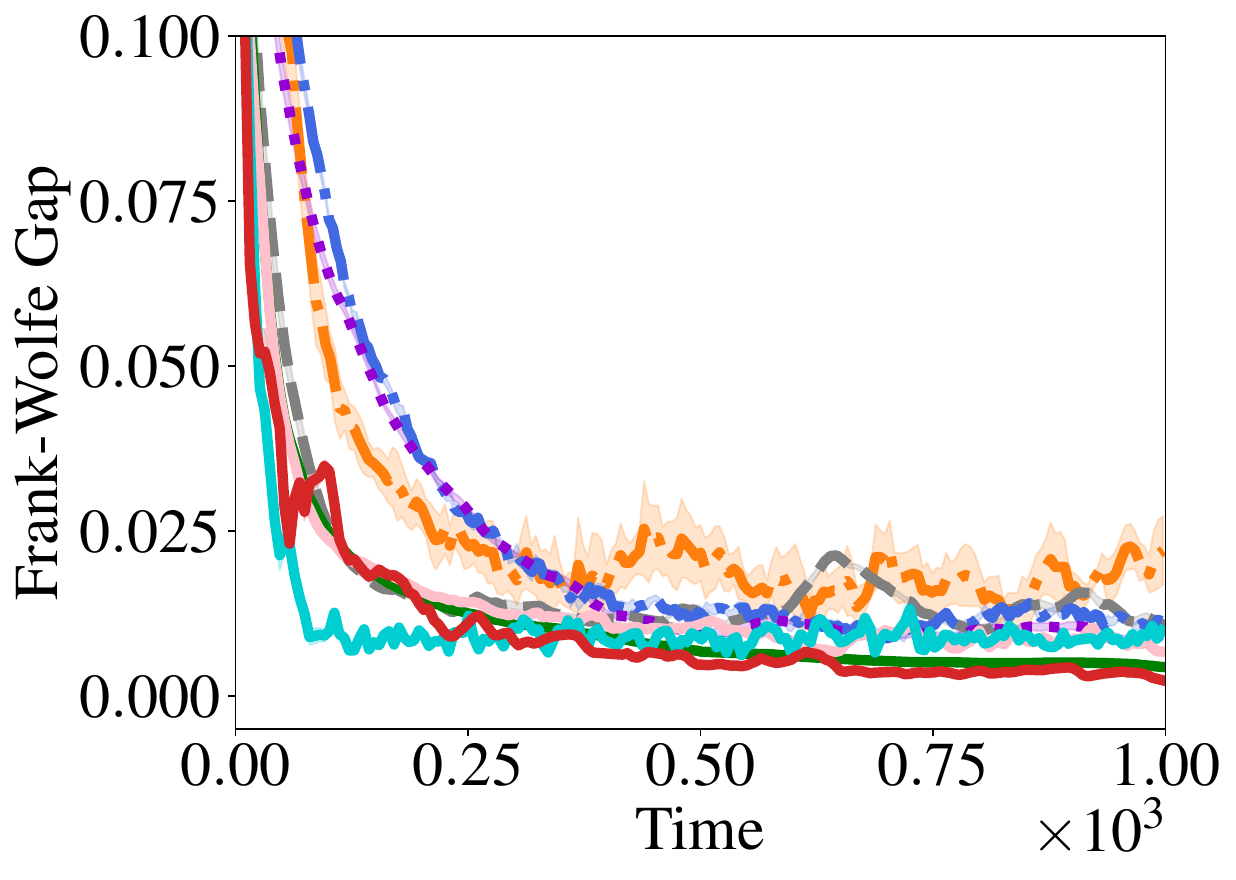}}
   \subfigure{
			\includegraphics[width=0.3\textwidth]{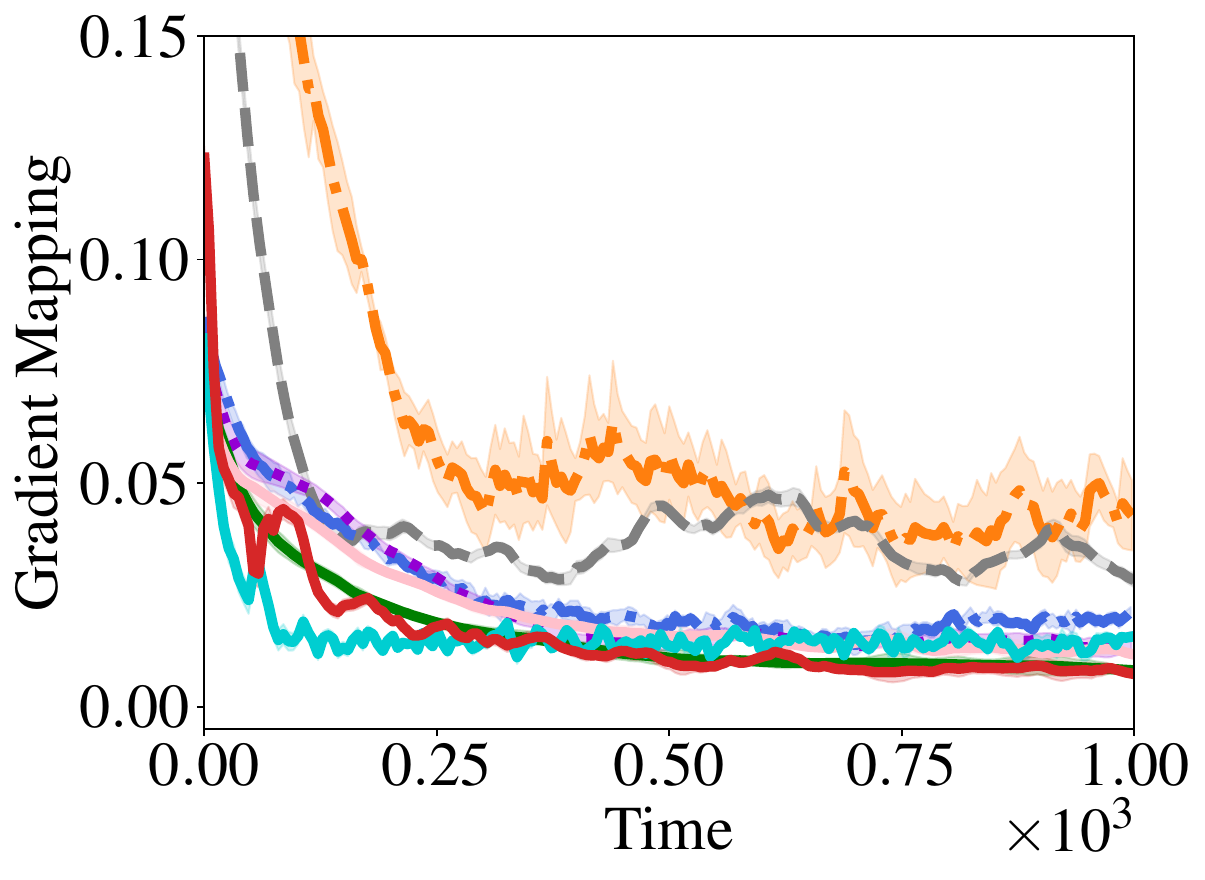}
		}
		\subfigure{
			\includegraphics[width=0.99\textwidth]{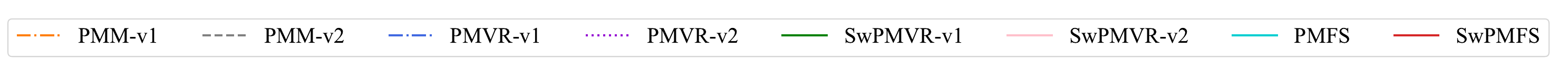}
		}
        
         \vspace{-0.15in}
		\caption{Results for mean-variance risk-averse portfolio optimization.}
        \vspace{-0.15in}
		\label{fig:c1}
	\end{center}
\end{figure*}
\subsection{Mean-deviation Risk-averse Optimization}
Finally, we consider the problem of mean-deviation risk-averse portfolio optimization, in which the risk is measured by the standard deviation. The mathematical formulation of this problem can be presented as:
\begin{align*} \max_{\mathbf{x} \in \mathcal{X}} \frac{1}{L} \sum_{l=1}^{L}\left\langle \mathbf r_l, \mathbf x \right\rangle-\lambda \sqrt{\frac{1}{L} \sum_{t=1}^{L}\left(\left\langle \mathbf r_{l}, 
\mathbf x \right\rangle-\left\langle \mathbf{\bar r}, \mathbf x \right\rangle\right)^{2}}, \end{align*}
where $\mathbf{\bar r}= \frac{1}{L}\sum_{l=1}^{L} \mathbf r_{l}$, $\mathcal{X}$ is the probability simplex, and the decision variable $ \mathbf x$ denotes the investment quantity vector in the $d$ assets. 
As shown by~\citet{jiang2022optimal}, this is a three-level compositional optimization problem.

In the experiments, we also evaluate different methods on the real-world datasets Industry-10, Industry-12, and Industry-17. 
Figure~\ref{fig:3} reports the objective value, the Frank-Wolfe gap, and the gradient mapping, averaged over 10 runs.
As can be seen, our methods~(especially PMVR-v2) demonstrate faster convergence compared to other algorithms in terms of loss value, Frank-Wolfe gap, and gradient mapping across all tasks.

\begin{figure*}[!t]
	\begin{center}
		\subfigure{
			\includegraphics[width=0.3\textwidth]{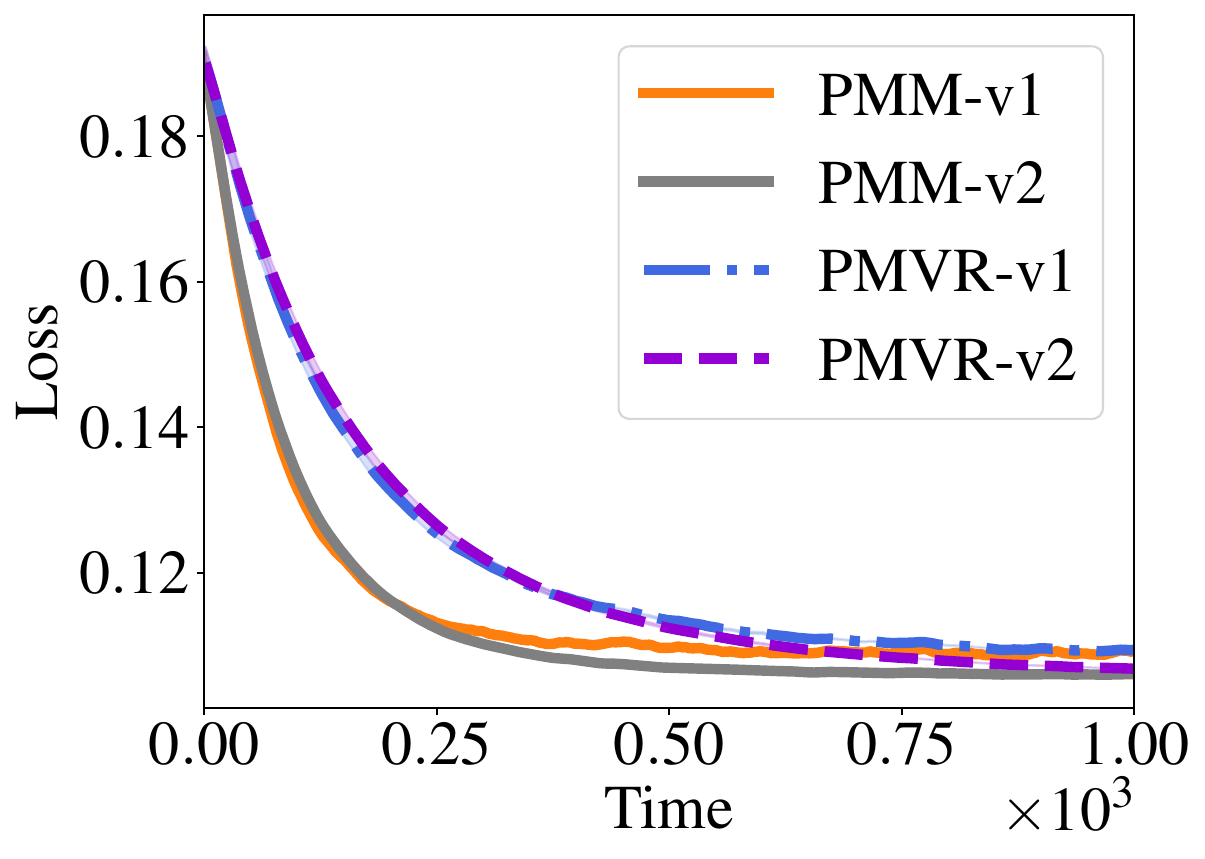}}
   \setcounter{subfigure}{0}
   \subfigure[ Industry-10]{
			\includegraphics[width=0.3\textwidth]{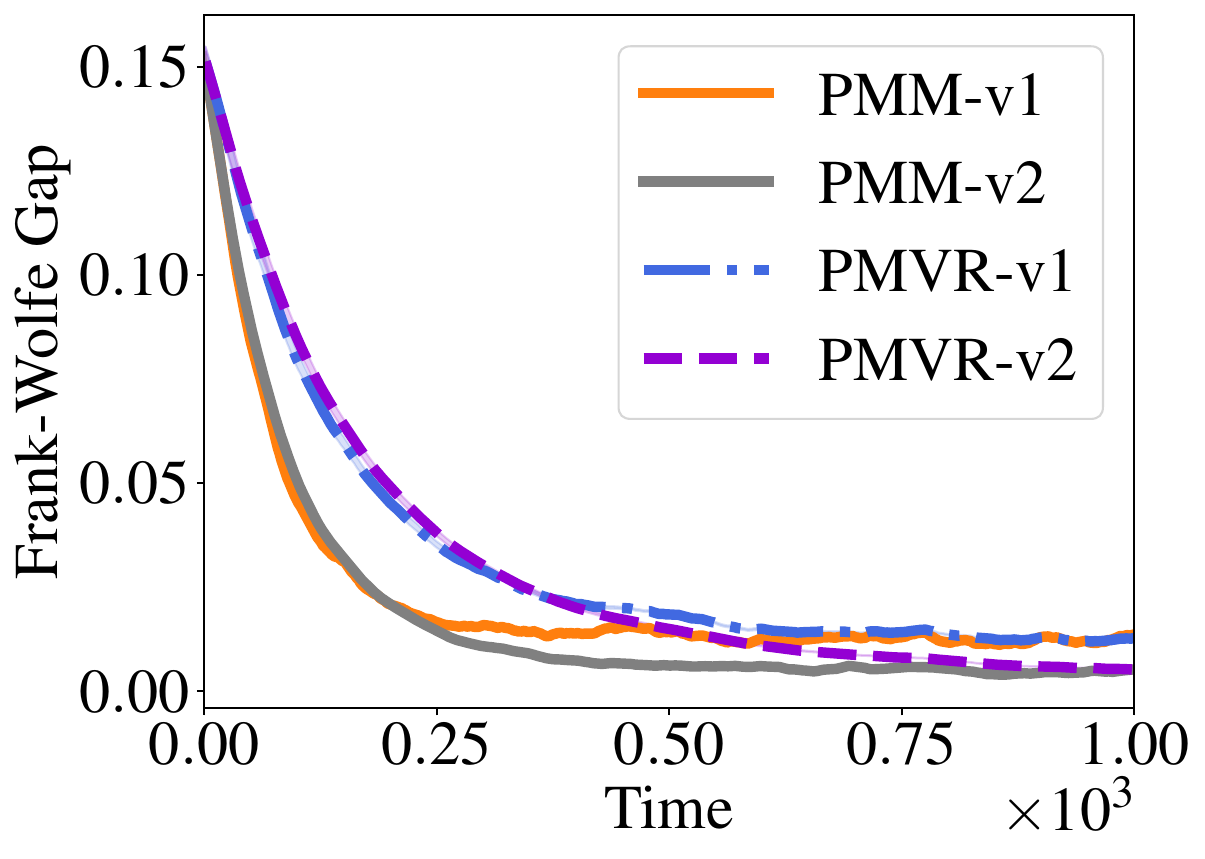}}
   \subfigure{
			\includegraphics[width=0.3\textwidth]{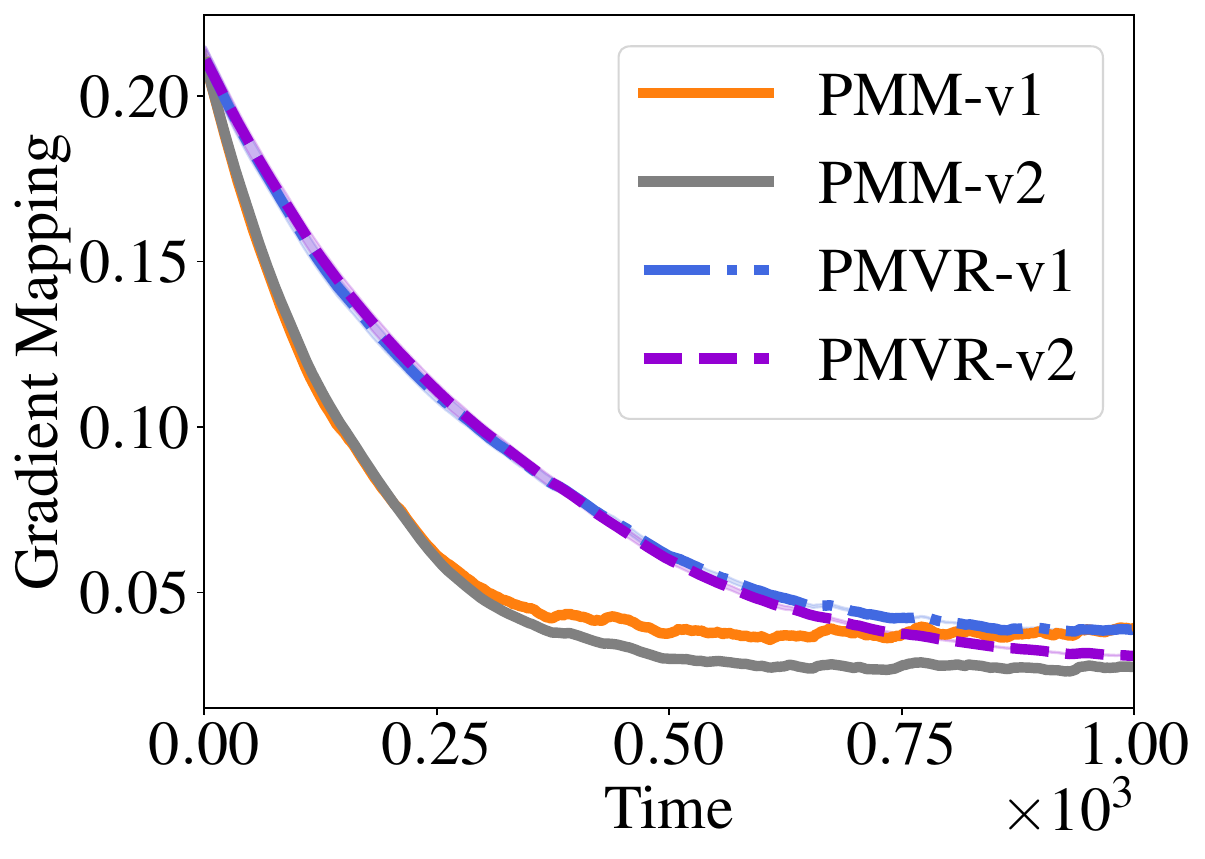}
		}
         \setcounter{subfigure}{0}
		\subfigure{
			\includegraphics[width=0.3\textwidth]{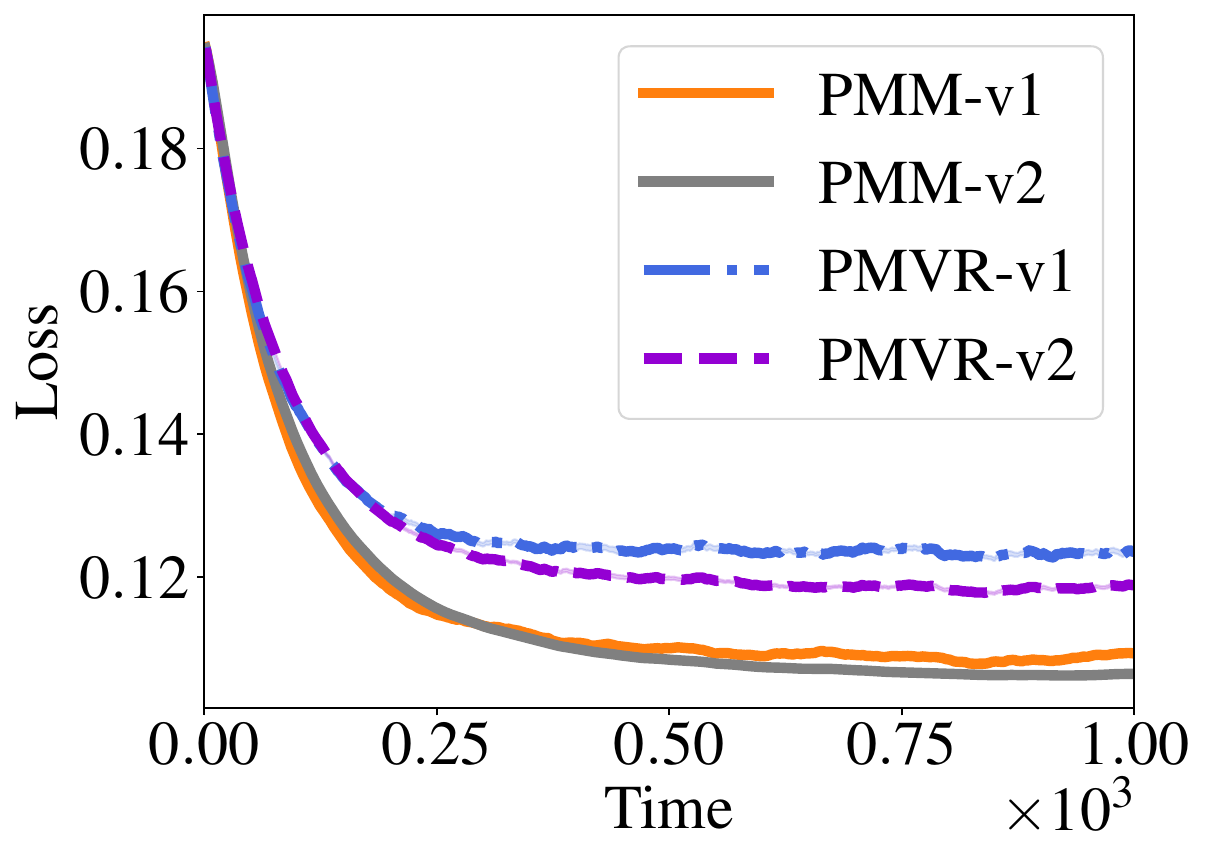}}
   \subfigure[ Industry-12]{
			\includegraphics[width=0.3\textwidth]{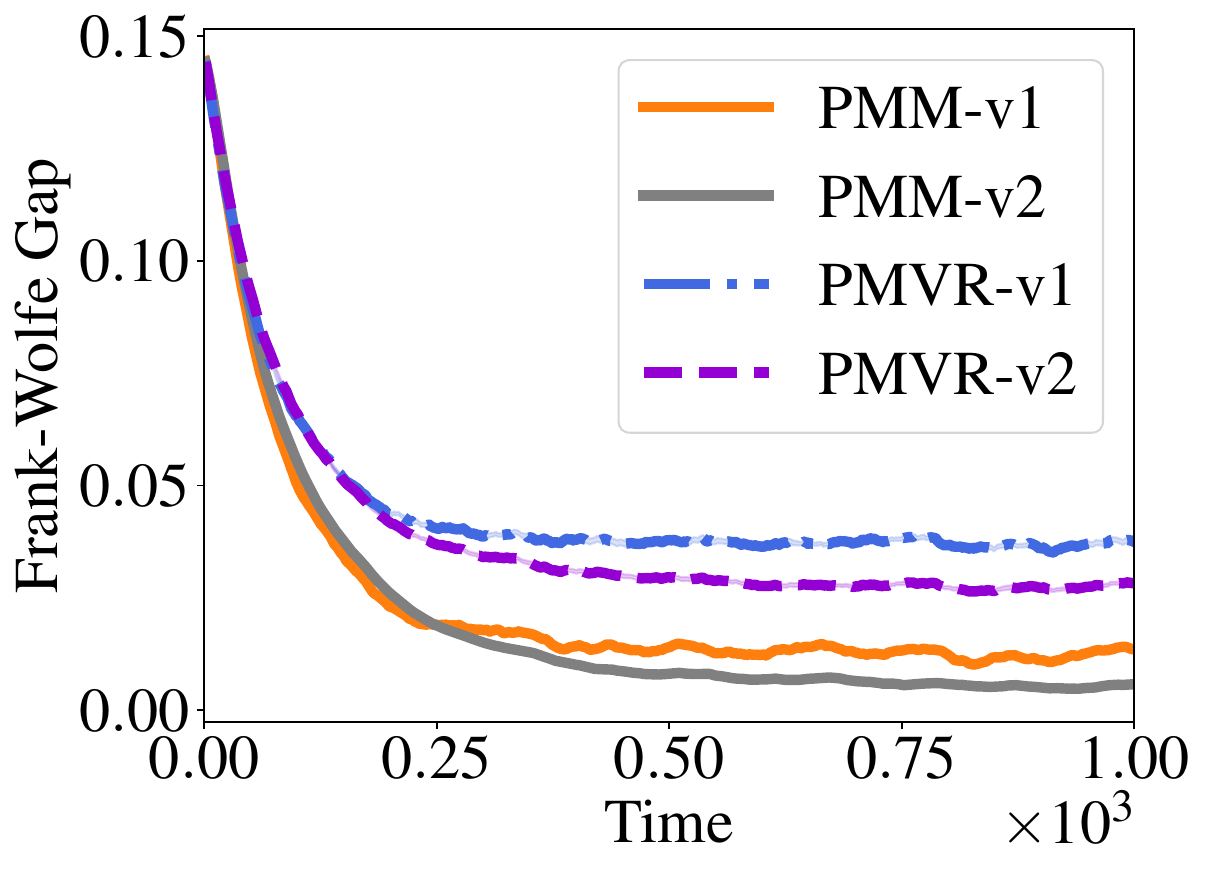}}
   \subfigure{
			\includegraphics[width=0.3\textwidth]{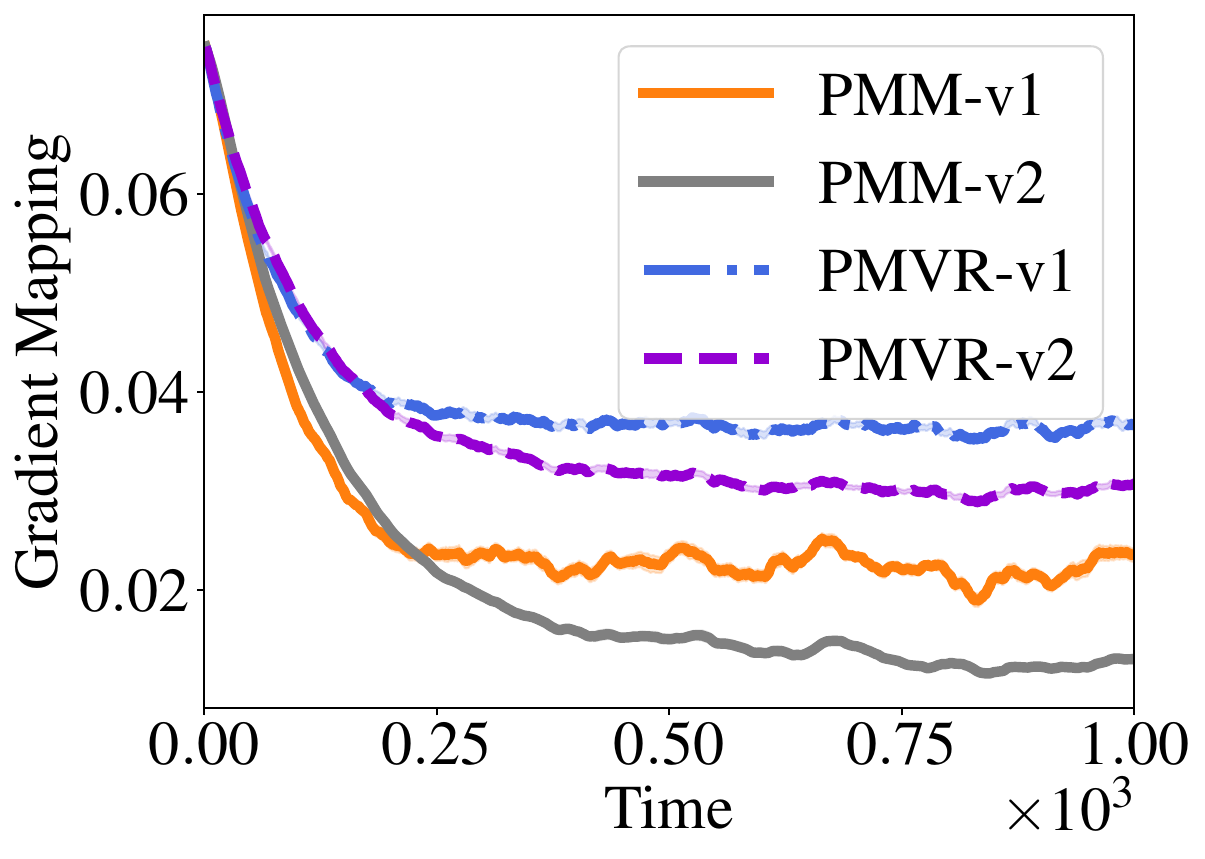}
		}
                 \setcounter{subfigure}{1}
		\subfigure{
			\includegraphics[width=0.3\textwidth]{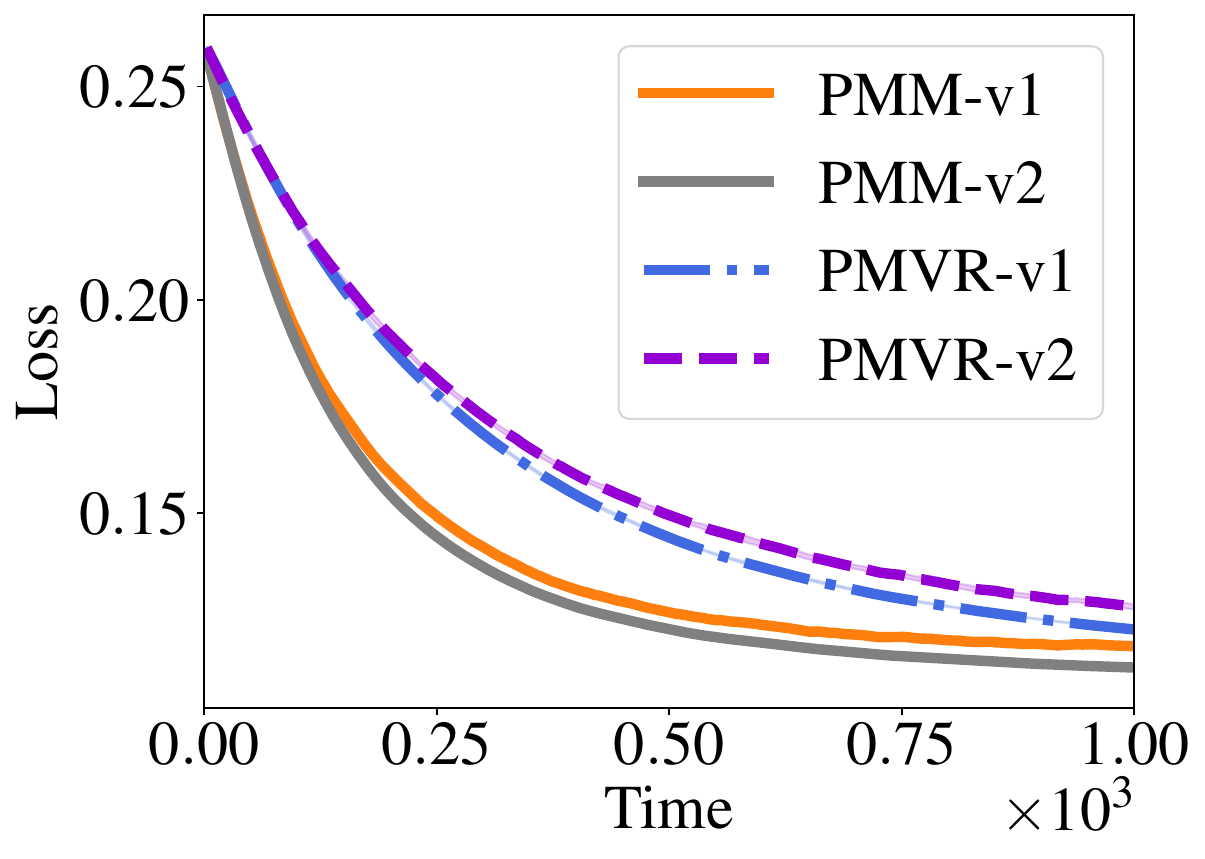}}
   \subfigure[ Industry-17]{
			\includegraphics[width=0.3\textwidth]{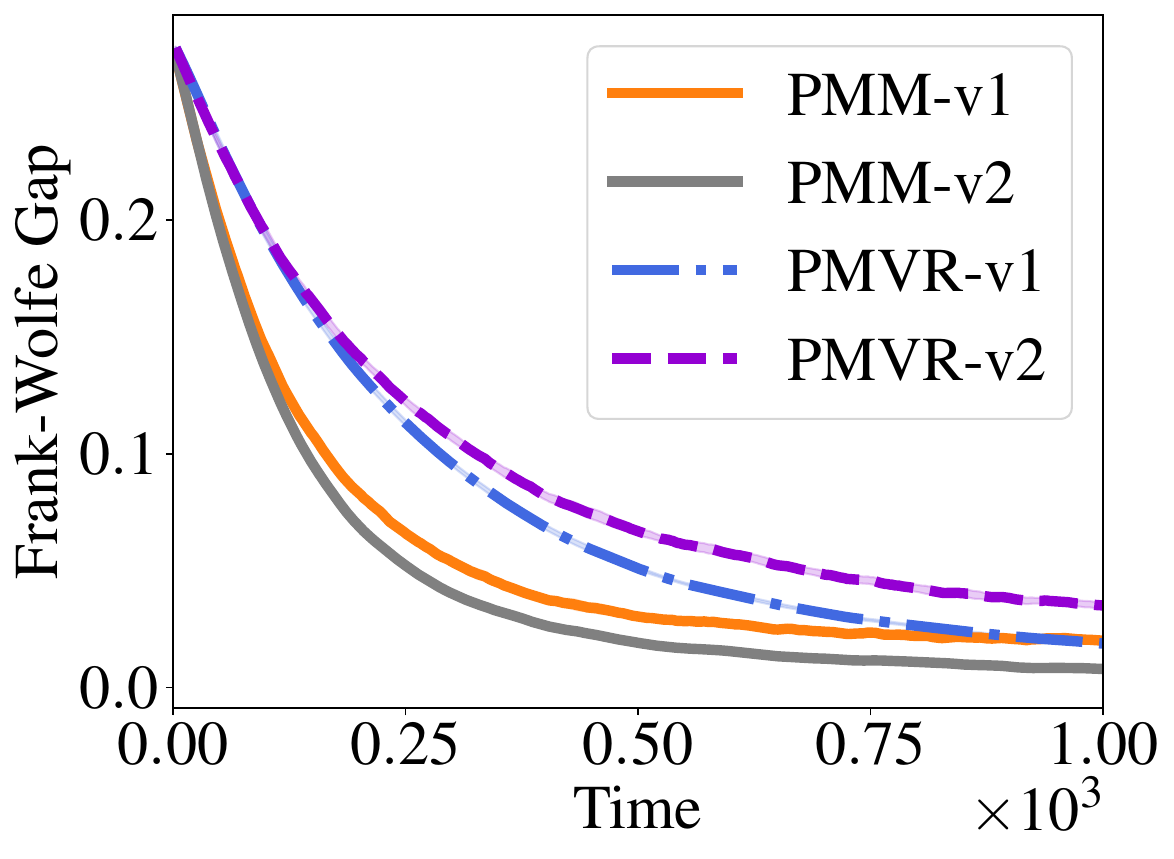}}
   \subfigure{
			\includegraphics[width=0.3\textwidth]{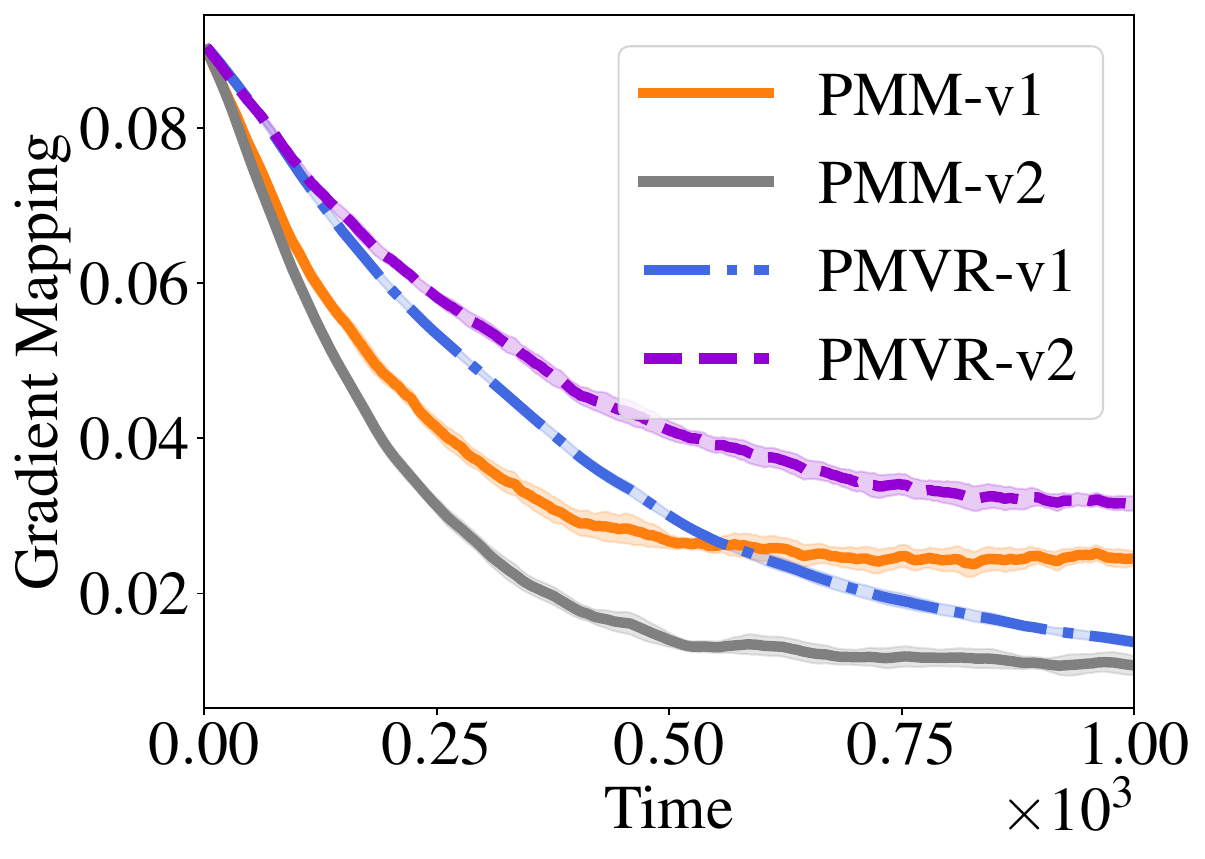}
		}
        \vspace{-0.1in}
		\caption{Results for mean-variance risk-averse portfolio optimization.}
        \vspace{-0.2in}
		\label{fig:c2}
	\end{center}
\end{figure*}
\subsection{Variants Comparison}
In this subsection, we compare the different algorithmic variants proposed in this paper: PMM-v1/v2, PMVR-v1/v2, the finite-sum variant PMFS, and their stage-wise counterparts for convex, strongly convex, and finite-sum objectives~(SwPMVR-v1, SwPMVR-v2, and SwPMFS). Specifically, we evaluate these methods on the mean-variance risk-averse portfolio optimization problem introduced in Section 7.2, whose objective function is convex. The experimental results are presented in Figure~\ref{fig:c1}. As illustrated, the finite-sum variants~(PMFS and SwPMFS) achieve the best overall performance in terms of the objective value, Frank-Wolfe gap, and gradient mapping. Moreover, the stage-wise variants~(SwPMVR-v1/v2) consistently outperform their non-stage-wise counterparts~(PMVR-v1/v2 and PMM-v1/v2), demonstrating the empirical benefit of gradually decreasing the hyperparameters.

Next, we compare PMM and PMVR. The PMM algorithm is designed under the general smoothness condition and utilizes a simpler estimator, whereas PMVR is explicitly tailored to the average smoothness assumption. In practice, we observe that PMVR performs better when the batch size is large. For example, in Figures~\ref{fig:1} and~\ref{fig:3}, where the batch size is set to 128, PMVR outperforms PMM in most instances. To further investigate the effect of batch size, we repeated the experiments from Section 7.2 with a smaller batch size of 32~(Figure~\ref{fig:c2}). Under this setting, PMM outperforms PMVR. This behavior can be attributed to PMVR's reliance on the average smoothness condition. Unlike standard smoothness, which applies only to the global objective function, average smoothness depends on the properties of individual samples. When the batch size is small, individual samples have a greater influence. Consequently, sample-level outliers can inflate the average smoothness constant or even violate the average smoothness assumption. Conversely, larger batches effectively smooth out this variability, rendering PMVR more stable and effective.

Based on these observations, we provide the following practical guidelines for selecting among the proposed variants. When the objective function is convex or strongly convex, the stage-wise algorithms are preferable. When the problem has a finite-sum structure and full gradients are available, PMFS and its stage-wise variant are the most effective among the proposed methods. For nonconvex objectives, when full gradients are unavailable or too expensive to compute, PMVR and PMM are recommended. Between these two methods, PMM is more suitable for smaller batch sizes, whereas PMVR is preferable when larger batch sizes can be used. Finally, we give a concise summary below.
\begin{compactenum}
\item \textbf{Stage-wise PMFS}: for convex objectives with a finite-sum structure.
\item \textbf{Stage-wise PMVR}: for convex objectives.
\item \textbf{PMFS}: for non-convex objectives with a finite-sum structure.
\item \textbf{PMVR}: for non-convex objectives when large batch sizes can be used.
\item \textbf{PMM}: for non-convex objectives when smaller batch sizes are required.
\end{compactenum}

\section{Conclusion}
In this paper, we investigate projection-free algorithms for stochastic constrained multi-level compositional optimization. 
Our methods improve prior projection-free multi-level results under the gradient mapping criterion and provide the first guarantees under the Frank-Wolfe gap in this setting.
We also introduce a stage-wise parameter-free variant which avoids knowing problem-dependent parameters, and a momentum-based approach that removes the stronger average-smoothness assumptions.
In addition, we derive guarantees for convex and strongly convex objectives and extend the approach to finite-sum settings.
Empirical results on several benchmarks corroborate the practical benefits of the proposed methods.

\newpage
\newpage
\appendix
\section{Proof of Theorem~\ref{thm:main}}
Note that when estimating the gradient, we use the term $\prod_{i=1}^K \nabla f_i(\u_t^{i-1})$ to approximate $\nabla F(\x_t)$. We first bound this estimation error as follows.
\begin{lemma}\label{lem:2} For $K\ge 2$, we can obtain the following guarantee:
\begin{align*}
\sum_{t=1}^T \E\left[\Norm{\prod_{i=1}^K\nabla f_i(\u_t^{i-1}) - \nabla F(\x_t)}^2\right] \leq  KL_F^2 \sum_{t=1}^T \sum_{i=1}^{K-1} \E\left[\Norm{\u_t^i - f_i(\u_t^{i-1})}^2\right].
\end{align*}
\end{lemma} 
\begin{proof}
According to the Lemma 4 of the literature~\citep{jiang2022optimal}, we know that
\begin{align*}
\E\left[\Norm{\nabla F(\x_t)- \prod_{i=1}^K\nabla f_i(\u_t^{i-1}) }^2\right] \leq K\sum_{i=1}^{K-1} C_i^2\E\left[\Norm{f_i(\u_t^{i-1})-\u_t^i}^2\right],
\end{align*}
where $C_i \coloneqq L_f^{K-1} L_J(1+L_f+\dotsc+L_f^{K-i-1}) \leq L_F$. Summing up, we complete the proof.
\end{proof}
Next, we can bound the gradient estimation error as follows.
%%%%%%%%%%%%%%%%%%%%%%%%%%   LEMMA  2    %%%%%%%%%%%%%%%%%%%%%%%%%%%%%%%%%%%%%%%%%%%%%%%%%%%%%
\begin{lemma}\label{lem:3} The gradient estimator $\v_t$ enjoys the following guarantee:
\begin{align*}
&\frac{1}{T}\sum_{t=1}^T \E\left[\Norm{\v_t -\prod_{i=1}^K \nabla f_i(\u_t^{i-1})}^2\right] \leq  \frac{1}{\alpha T}\E\left[\Norm{ \v_{1} - \prod_{i=1}^K \nabla f_i(\u_{1}^{i-1})}^2\right] \\
&\qquad+ 2 K^2 \sigma_J^2 L_f^{2K-2}\frac{\alpha}{B_1}+2K \LL_J^2L_f^{2K-2} \frac{1}{B_1\alpha T}\sum_{t=1}^T\sum_{i=1}^{K} \E\left[\Norm{\u_{t+1}^{i-1} - \u_{t}^{i-1}}^2\right].
\end{align*}
\end{lemma}	
\begin{proof}
According to the definition of $\v_t$, we know that:
\begin{align*}
&\E\left[\Norm{\v_t -\prod_{i=1}^K \nabla f_i(\u_t^{i-1})}^2\right]\\
=&\E\left[\left \| \frac{1}{B_1}\sum_{j=1}^{B_1}\prod_{i=1}^K \nabla f_i(\u_t^{i-1};\xi_t^{i,j})  + (1-\alpha)\left(\v_{t-1} - \frac{1}{B_1} \sum_{j=1}^{B_1} \prod_{i=1}^K \nabla f_i(\u_{t-1}^{i-1};\xi_t^{i,j})\right) - \prod_{i=1}^K \nabla f_i(\u_t^{i-1})\right \|^2 \right]\\
=& \E\left[\left \| (1-\alpha) \left( \v_{t-1} - \prod_{i=1}^K \nabla f_i(\u_{t-1}^{i-1}) \right) + \alpha \left( \frac{1}{B_1}\sum_{j=1}^{B_1}  \prod_{i=1}^K \nabla f_i(\u_{t-1}^{i-1};\xi_t^{i,j}) - \prod_{i=1}^K \nabla f_i(\u_{t-1}^{i-1})\right) \right.\right.\\
& \left.\left. + \left( \prod_{i=1}^K \nabla f_i(\u_{t-1}^{i-1}) -\prod_{i=1}^K \nabla f_i(\u_{t}^{i-1})   -  \frac{1}{B_1}\sum_{j=1}^{B_1}\prod_{i=1}^K \nabla f_i(\u_{t-1}^{i-1};\xi_t^{i,j}) + \frac{1}{B_1}\sum_{j=1}^{B_1}\prod_{i=1}^K \nabla f_i(\u_{t}^{i-1};\xi_t^{i,j})\right) \right\|^2\right] \\
\leq & (1-\alpha)^2 \E\left[\Norm{ \v_{t-1} - \prod_{i=1}^K \nabla f_i(\u_{t-1}^{i-1})}^2\right] + 2\alpha^2\E\left[\Norm{\frac{1}{B_1}\sum_{j=1}^{B_1}  \prod_{i=1}^K \nabla f_i(\u_{t-1}^{i-1};\xi_t^{i,j}) - \prod_{i=1}^K \nabla f_i(\u_{t-1}^{i-1})}^2\right]\\
& +2\E\left[\Norm{\frac{1}{B_1}\sum_{j=1}^{B_1}  \prod_{i=1}^K \nabla f_i(\u_{t-1}^{i-1};\xi_t^{i,j}) -  \frac{1}{B_1}\sum_{j=1}^{B_1}  \prod_{i=1}^K \nabla f_i(\u_{t}^{i-1};\xi_t^{i,j}) -  \prod_{i=1}^K \nabla f_i(\u_{t-1}^{i-1}) + \prod_{i=1}^K \nabla f_i(\u_{t}^{i-1})}^2\right],
\end{align*}
where the last equality is due to $(a+b)^2 \leq 2a^2 +2b^2$ and the fact that 
\begin{align*}
    \E \left[ \alpha\left( \frac{1}{B_1}\sum_{j=1}^{B_1} \prod_{i=1}^K \nabla f_i(\u_{t-1}^{i-1};\xi_t^{i,j}) - \prod_{i=1}^K \nabla f_i(\u_{t-1}^{i-1})\right) +\prod_{i=1}^K \nabla f_i(\u_{t-1}^{i-1}) -\prod_{i=1}^K \nabla f_i(\u_{t}^{i-1})\right. \\
\left. -  \frac{1}{B_1}\sum_{j=1}^{B_1} \prod_{i=1}^K \nabla f_i(\u_{t-1}^{i-1};\xi_t^{i,j}) +  \frac{1}{B_1}\sum_{j=1}^{B_1} \prod_{i=1}^K \nabla f_i(\u_{t}^{i-1};\xi_t^{i,j}) \right] =0.
\end{align*}
Then, we would bound the two terms, respectively. First, we have that:
\begin{align*}
    &\E\left[\Norm{ \frac{1}{B_1}\sum_{j=1}^{B_1}  \prod_{i=1}^K \nabla f_i(\u_{t-1}^{i-1};\xi_t^{i,j}) - \prod_{i=1}^K \nabla f_i(\u_{t-1}^{i-1})}^2\right]\\
    = & \frac{1}{B_1^2}\sum_{j=1}^{B_1}\E\left[\Norm{\prod_{i=1}^K \nabla f_i(\u_{t-1}^{i-1};\xi_t^{i,j}) - \prod_{i=1}^K \nabla f_i(\u_{t-1}^{i-1})}^2\right]\\
    =&\frac{1}{B_1^2}\sum_{j=1}^{B_1}\E\left[\left\|\prod_{i=1}^K \nabla f_i(\u_{t-1}^{i-1};\xi_t^{i,j}) - \nabla f_1(\u_{t-1}^0) \prod_{i=2}^K \nabla  f_i(\u_{t-1}^{i-1};\xi_t^{i,j}) \right. \right. \\
    & \quad \left.\left. + \nabla f_1(\u_{t-1}^0) \prod_{i=2}^K \nabla  f_i(\u_{t-1}^{i-1};\xi_t^{i,j}) - \nabla f_1(\u_{t-1}^0) \nabla f_2(\u_{t-1}^1) \prod_{i=3}^K \nabla  f_i(\u_{t-1}^{i-1};\xi_t^{i,j}) \right. \right.\\ 
    &  \qquad \cdots \\
    & \quad \left.\left. +  \left(\prod_{i=1}^{K-1} \nabla f_i(\u_{t-1}^{i-1})\right) \nabla f_K(\u_{t-1}^{K-1};\xi_t^{i,j}) - \prod_{i=1}^K \nabla f_i(\u_{t-1}^{i-1}) \right\|^2\right] \\
    \leq & \frac{K}{B_1^2}\sum_{j=1}^{B_1} \E\left[\Norm{\prod_{i=1}^K \nabla f_i(\u_{t-1}^{i-1};\xi_t^{i,j}) - \nabla f_1(\u_{t-1}^0) \prod_{i=2}^K \nabla  f_i(\u_{t-1}^{i-1};\xi_t^{i,j})}^2 \right] \\
    & + \frac{K}{B_1^2}\sum_{j=1}^{B_1}\E\left[\Norm{\nabla f_1(\u_{t-1}^0) \prod_{i=2}^K \nabla  f_i(\u_{t-1}^{i-1};\xi_t^{i,j}) -\nabla f_1(\u_{t-1}^0) \nabla f_2(\u_{t-1}^1) \prod_{i=3}^K \nabla  f_i(\u_{t-1}^{i-1};\xi_t^{i,j})}^2 \right] \\
    & + ... \\
    & + \frac{K}{B_1^2}\sum_{j=1}^{B_1}\E\left[\Norm{\left(\prod_{i=1}^{K-1} \nabla f_i(\u_{t-1}^{i-1})\right) \nabla f_K(\u_{t-1}^{K-1};\xi_t^{i,j})  - \prod_{i=1}^K \nabla f_i(\u_{t-1}^{i-1})}^2 \right]  \\
   \leq &  \frac{K}{B_1^2}\sum_{j=1}^{B_1} \sigma_J^2 L_f^{2K-2} \\
   = & \frac{K^2 }{B_1} \sigma_J^2 L_f^{2K-2}
\end{align*}
When dealing with the second term, we have
\begin{equation}\label{diff}
    \begin{split}
        &\E\left[\Norm{\frac{1}{B_1}\sum_{j=1}^{B_1}  \prod_{i=1}^K \nabla f_i(\u_{t-1}^{i-1};\xi_t^{i,j}) -  \frac{1}{B_1}\sum_{j=1}^{B_1}  \prod_{i=1}^K \nabla f_i(\u_{t}^{i-1};\xi_t^{i,j}) -  \prod_{i=1}^K \nabla f_i(\u_{t-1}^{i-1}) + \prod_{i=1}^K \nabla f_i(\u_{t}^{i-1})}^2\right]\\
    %= & \frac{1}{B_1^2} \E\left[\Norm{\sum_{j=1}^{B_1}\left(  \prod_{i=1}^K \nabla f_i(\u_{t-1}^{i-1};\xi_t^{i,j}) -  \prod_{i=1}^K \nabla f_i(\u_{t}^{i-1};\xi_t^{i,j}) -  \prod_{i=1}^K \nabla f_i(\u_{t-1}^{i-1}) + \prod_{i=1}^K \nabla f_i(\u_{t}^{i-1}) \right)}^2\right]\\
    = & \frac{1}{B_1^2}\sum_{j=1}^{B_1}\E\left[\Norm{  \prod_{i=1}^K \nabla f_i(\u_{t-1}^{i-1};\xi_t^{i,j}) -  \prod_{i=1}^K \nabla f_i(\u_{t}^{i-1};\xi_t^{i,j}) -  \prod_{i=1}^K \nabla f_i(\u_{t-1}^{i-1}) + \prod_{i=1}^K \nabla f_i(\u_{t}^{i-1}) }^2\right]\\
    \leq & \frac{1}{B_1^2}\sum_{j=1}^{B_1}\E\left[\Norm{  \prod_{i=1}^K \nabla f_i(\u_{t-1}^{i-1};\xi_t^{i,j}) -  \prod_{i=1}^K \nabla f_i(\u_{t}^{i-1};\xi_t^{i,j}) }^2\right]\\
    =&\frac{1}{B_1^2}\sum_{j=1}^{B_1} \E\left[ \left\|\prod_{i=1}^K \nabla f_i(\u_{t-1}^{i-1};\xi_t^{i,j}) - \nabla f_1(\u_{t}^{0};\xi_t^{1,j}) \prod_{i=2}^K \nabla f_i(\u_{t-1}^{i-1};\xi_t^{i,j})   \right. \right. \\
    & \left. \left.\quad\quad + \nabla f_1(\u_{t}^{0};\xi_t^{1,j}) \prod_{i=2}^K \nabla f_i(\u_{t-1}^{i-1};\xi_t^{i,j})  - \nabla f_1(\u_{t}^{0};\xi_t^{1,j}) \nabla f_2(\u_{t}^1;\xi_t^{2,j}) \prod_{i=3}^K \nabla f_i(\u_{t-1}^{i-1};\xi_t^{i,j})\right. \right.  \\
    &\left. \left.\quad\quad \cdots \right. \right. \\
    & \left. \left.\quad\quad+ \left(\prod_{i=1}^{K-1} \nabla f_i(\u_{t}^{i-1};\xi_t^{i,j})\right)\nabla f_K(\u_{t-1}^{K-1};\xi_t^{K,j}) - \prod_{i=1}^K \nabla f_i(\u_{t}^{i-1};\xi_t^{i,j})  \right\|^2 \right]\\
    \leq & \frac{K}{B_1^2}\sum_{j=1}^{B_1} \LL_J^2L_f^{2K-2} \sum_{i=1}^{K} \E\left[\Norm{\u_t^{i-1} - \u_{t-1}^{i-1}}^2\right]\\
    \leq & \frac{K}{B_1} \LL_J^2L_f^{2K-2} \sum_{i=1}^{K} \E\left[\Norm{\u_t^{i-1} - \u_{t-1}^{i-1}}^2\right],
    \end{split}
\end{equation}
where the first two steps are due to the fact that
\begin{align*}
    \E\left[\frac{1}{B_1}\sum_{j=1}^{B_1} \left(  \prod_{i=1}^K \nabla f_i(\u_{t-1}^{i-1};\xi_t^{i,j}) -  \prod_{i=1}^K \nabla f_i(\u_{t}^{i-1};\xi_t^{i,j}) \right)  -  \prod_{i=1}^K \nabla f_i(\u_{t-1}^{i-1}) + \prod_{i=1}^K \nabla f_i(\u_{t}^{i-1})\right]=0
\end{align*}
To this end, we can conclude that:
\begin{align*}
&\E\left[\Norm{\v_t -\prod_{i=1}^K \nabla f_i(\u_t^{i-1})}^2\right] \leq  (1-\alpha) \E\left[\Norm{ \v_{t-1} - \prod_{i=1}^K \nabla f_i(\u_{t-1}^{i-1})}^2\right] \\
&\qquad+ 2 K^2 \sigma_J^2 L_f^{2K-2}\frac{\alpha^2}{B_1}+2K \LL_J^2L_f^{2K-2} \frac{1}{B_1}\sum_{i=1}^{K} \E\left[\Norm{\u_t^{i-1} - \u_{t-1}^{i-1}}^2\right].
\end{align*}
Summing up over $t$ and rearranging, we complete the proof for this lemma.
\end{proof}
Next, we bound the estimation error of the function value estimator in the following lemma.
\begin{lemma}\label{lem:4} The inner function estimator $\u_t$ ensures that:
\begin{align*}
	\frac{1}{T}\sum_{t=1}^T \sum_{i=1}^K \E\left[\Norm{\u_t^i -  f_i(\u_t^{i-1})}^2\right] \leq &\frac{\sum_{i=1}^K \E\left[\Norm{\u_1^i -  f_i(\u_1^{i-1})}^2\right]}{\alpha T} \\
 &+ 2 K \sigma^2 \frac{\alpha }{B_1}  + \frac{2\LL_f^2}{\alpha B_1 T}  \sum_{t=1}^T \sum_{i=1}^K\E\left[\Norm{\u_{t+1}^{i-1} - \u_{t}^{i-1}}^2\right].
\end{align*}
\end{lemma}	
\begin{proof}
According to the definition of $\u_t^i$, we have:
\begin{align*}
& \E\left[\Norm{ \u_t^i - f_i(\u_t^{i-1})}^2\right]\\
%=&  \E\left[\Norm{\frac{1}{B_1}\sum_{j=1}^{B_1}f_i(\u_t^{i-1};\xi_t^{i,j}) + (1-\alpha)\left(\u_{t-1}^i - \frac{1}{B_1}\sum_{j=1}^{B_1}f_i(\u_{t-1}^{i-1};\xi_t^{i,j})\right)-  f_i(\u_t^{i-1})}^2\right]\\
=& \E\left[\left\|(1-\alpha)\left(\u_{t-1}^i - f_i(\u_{t-1}^{i-1})\right) +  \alpha\frac{1}{B_1}\sum_{j=1}^{B_1} \left(  f_i(\u_{t-1}^{i-1};\xi_t^{i,j}) -  f_i(\u_{t-1}^{i-1})\right) \right.\right.\\
&\quad\quad\quad \left.\left.   +  f_i(\u_{t-1}^{i-1}) - f_i(\u_t^{i-1})  -\frac{1}{B_1}\sum_{j=1}^{B_1} \left(f_i(\u_{t-1}^{i-1};\xi_t^{i,j}) -  f_i(\u_t^{i-1};\xi_t^{i,j}) \right) \right\|^2\right] \\
\leq & (1-\alpha)^2 \E \Norm{\u_{t-1}^i - f_i(\u_{t-1}^{i-1})}^2 + \frac{2\alpha^2}{B_1} \E\left[\Norm{f_i(\u_{t-1}^{i-1};\xi_t^{i,j}) - f_i(\u_{t-1}^{i-1})}^2\right] \\
& \qquad  +2\E\left[\left\|\frac{1}{B_1}\sum_{j=1}^{B_1}\left(f_i(\u_{t-1}^{i-1};\xi_t^{i,j}) - f_i(\u_t^{i-1};\xi_t^{i,j} )- f_i(\u_{t-1}^{i-1}) + f_i(\u_t^{i-1})\right)\right\|^2\right]\\
\leq&  (1-\alpha) \E \Norm{\u_{t-1}^i - f_i(\u_{t-1}^{i-1})}^2 + \frac{2\alpha^2 \sigma^2}{B_1} + \frac{2}{B_1}\LL_f^2 \left\|\u_{t-1}^{i-1} - \u_t^{i-1}\right\|^2, 
\end{align*}
%where the first inequality is due to $(a+b)^2 \leq 2a^2 +2b^2$ and the fact that
%\begin{align*}
%    \E\left[\frac{\alpha}{B_1}\sum_{j=1}^{B_1} \left(  f_i(\u_{t-1}^{i-1};\xi_t^{i,j}) -  f_i(\u_{t-1}^{i-1})\right)+f_i(\u_{t-1}^{i-1})   - f_i(\u_t^{i-1})\right. \\
% \left.-  \frac{1}{B_1}\sum_{j=1}^{B_1}\left( f_i(\u_{t-1}^{i-1};\xi_t^{i,j}) - f_i(\u_t^{i-1};\xi_t^{i,j})\right)\right] = 0, 
%\end{align*}
where the last inequality is because of
\begin{align*}
&\E\left[\left\|\frac{1}{B_1}\sum_{j=1}^{B_1}\left(f_i(\u_{t-1}^{i-1};\xi_t^{i,j})  - f_i(\u_t^{i-1};\xi_t^{i,j} )- f_i(\u_{t-1}^{i-1}) + f_i(\u_t^{i-1})\right)\right\|^2\right]\\
%\leq &  \frac{1}{B_1^2}\sum_{j=1}^{B_1}\E\left[\left\|f_i(\u_{t-1}^{i-1};\xi_t^{i,j}) - f_i(\u_t^{i-1};\xi_t^{i,j}) - f_i(\u_{t-1}^{i-1}) + f_i(\u_t^{i-1}) \right\|^2\right]\\
\leq & \frac{1}{B_1^2}\sum_{j=1}^{B_1}\E\left[\left\|f_i(\u_{t-1}^{i-1};\xi_t^{i,j})  -f_i(\u_t^{i-1};\xi_t^{i,j})\right\|^2\right]
\leq \frac{1}{B_1}\LL_f^2 \left\|\u_{t-1}^{i-1} - \u_t^{i-1}\right\|^2.
\end{align*}
This leads to the fact that:
\begin{align*}
\sum_{i=1}^K \E\left[\Norm{ \u_t^i - f_i(\u_t^{i-1}) }^2\right] \leq  (1-\alpha) \sum_{i=1}^K \E \Norm{\u_{t-1}^i - f_i(\u_{t-1}^{i-1})}^2 \\
+ \frac{2\alpha^2 \sigma^2 K}{B_1} + \frac{2}{B_1}\LL_f^2 \sum_{i=1}^K \left\|\u_{t-1}^{i-1} - \u_t^{i-1}\right\|^2.  
\end{align*}
By summing up and rearranging, we finish the proof of this lemma.
\end{proof}
Then, we bound the term $\sum_{i=1}^K\E\left[\Norm{\u_{t+1}^{i-1} - \u_{t}^{i-1}}^2\right]$.
\begin{lemma} We can obtain the following guarantee.
\begin{align*}
&\sum_{i=1}^K  \E\left[\Norm{\u_{t+1}^{i-1} - \u_{t}^{i-1}}^2 \right] \\ 
 \leq&  \left(\sum_{i=1}^{K}\left(2\LL_f^2\right)^{i-1}\right) \left(\E \left[\eta^2 \Norm{\z_t - \x_t}^2\right] + \frac{2\alpha^2\sigma^2 K}{B_1}   + 2\alpha^2K \sum_{i=1}^{K} \E\left[\Norm{\u_t^i - f_i(\u_t^{i-1})}^2\right]  \right).
\end{align*}
\end{lemma}
\begin{proof}
(1) For the first level, i.e., $i=1$, we have: 
    \begin{align*}
	    \E\left[\Norm{\u_{t+1}^{i-1} - \u_{t}^{i-1}}^2\right] = \E\left[\Norm{\x_{t+1} - \x_{t}}^2\right] = \E \left[\eta^2 \Norm{\z_t - \x_t}^2\right].
	\end{align*}
(2) For other levels, i.e., $2\leq i\leq K$, we have:
	\begin{align*}
		& \E\left[\Norm{\u_{t+1}^{i-1} - \u_{t}^{i-1}}^2\right]\\
		=&\E\left[\left\|\alpha \left(f_{i-1}(\u_{t}^{i-2}) - \u_{t}^{i-1}\right) + \frac{1}{B_1}\sum_{j=1}^{B_1}(f_{i-1}(\u_{t+1}^{i-2};\xi_{t+1}^{i-1,j}) - f_{i-1}(\u_{t}^{i-2};\xi_{t+1}^{i-1,j}))\right.\right.\\
		 &\quad \left.\left.+ \alpha \left( \frac{1}{B_1}\sum_{j=1}^{B_1}f_{i-1} (\u_{t}^{i-2};\xi_{t+1}^{i-1,j})-f_{i-1}(\u_{t}^{i-2})\right) \right\|^2 \right]\\
\leq& 2\E\left[\Norm{\alpha\left(f_{i-1}(\u_{t}^{i-2}) - \u_{t}^{i-1}\right) + \alpha \left(\frac{1}{B_1}\sum_{j=1}^{B_1}f_{i-1} (\u_{t}^{i-2};\xi_{t+1}^{i-1,j})-f_{i-1}(\u_{t}^{i-2})\right)}^2\right] \\
& \qquad + 2\LL_f^2\E\left[\Norm{\u_{t+1}^{i-2} - \u_{t}^{i-2}}^2 \right]\\
	\leq&  2\alpha^2 \E\left[\Norm{f_{i-1}(\u_{t}^{i-2}) - \u_{t}^{i-1}}^2\right] + \frac{2\alpha^2 \sigma^2}{B_1} + 2\LL_f^2\E\left[\Norm{\u_{t+1}^{i-2} - \u_{t}^{i-2}}^2 \right].
	\end{align*} 
Denote $\Upsilon_t^i =\E\left[\Norm{\u_t^i - f_i(\u_t^{i-1})}^2\right]$ and  $Q^{i}_t = \E\left[\Norm{\u_{t+1}^{i-1} - \u_{t}^{i-1}}^2\right]$, we have $Q^{i}_t \leq 2\LL_f^2 Q^{i-1}_t + 2\alpha^2 \Upsilon_t^{i-1} + \frac{2\alpha^2\sigma^2}{B_1}$ for $i\geq 2$. Then we can get:
\begin{align*}
Q^{1}_t &\leq \E \left[\eta^2 \Norm{\z_t - \x_t}^2\right] \\
Q^{2}_t &\leq \left(2\LL_f^2\right) \E \left[\eta^2 \Norm{\z_t - \x_t}^2\right] +  \frac{2\alpha^2\sigma^2}{B_1} +2\alpha^2 \Upsilon_t^{1} \\
Q^{3}_t &\leq \left(2\LL_f^2\right)^2 \E \left[\eta^2 \Norm{\z_t - \x_t}^2\right]  +  \frac{2\alpha^2\sigma^2 \left(1+2\LL_f^2\right)}{B_1} + 2\alpha^2\left( 2\LL_f^2  \Upsilon_t^{1}+ \Upsilon_t^{2}\right) \\
& \cdots\\
Q^{i}_t &\leq \left(2\LL_f^2\right)^{i-1} \E \left[\eta^2 \Norm{\z_t - \x_t}^2\right] + \frac{2\alpha ^2\sigma^2}{B_1} \sum_{j=1}^{i-1}\left(2\LL_f^2\right)^{j-1}  + 2\alpha^2 \sum_{j=1}^{i-1} \left(2\LL_f^2\right)^{i-1-j} \Upsilon_t^{j}\\
&\leq \left(2\LL_f^2\right)^{i-1} \E \left[\eta^2 \Norm{\z_t - \x_t}^2\right] + \frac{2\alpha^2\sigma^2}{B_1} \sum_{j=1}^{K}\left(2\LL_f^2\right)^{j-1}  + 2\alpha^2 \sum_{j=1}^{K} \sum_{l=1}^{K} \left(2L_f^2\right)^{K-l} \Upsilon_t^{j}.
\end{align*}
When summing up, we have:
\begin{align*}
    &\sum_{i=1}^K \E\left[\Norm{\u_{t+1}^{i-1} - \u_{t}^{i-1}}^2 \right] 
    = \sum_{i=1}^{K} Q^{i}_t \\
    \leq &\sum_{i=1}^{K}\left(2\LL_f^2\right)^{i-1} \E \left[\eta^2 \Norm{\z_t - \x_t}^2\right] + \frac{2\alpha^2\sigma^2 K}{B_1} \sum_{i=1}^{K}\left(2\LL_f^2\right)^{i-1} + 2\alpha^2 K \sum_{j=1}^{K} \sum_{l=1}^{K} \left(2L_f^2\right)^{K-l} \Upsilon_t^{j}  \\
    \leq &\left(\sum_{i=1}^{K}\left(2\LL_f^2\right)^{i-1}\right) \left(\E \left[\eta^2 \Norm{\z_t - \x_t}^2\right] + \frac{2\alpha^2\sigma^2 K}{B_1}   + 2\alpha^2K \sum_{i=1}^{K} \Upsilon_t^{j}  \right)\\
    \leq &\left(\sum_{i=1}^{K}\left(2\LL_f^2\right)^{i-1}\right) \left(\E \left[\eta^2 \Norm{\z_t - \x_t}^2\right] + \frac{2\alpha^2\sigma^2 K}{B_1}   + 2\alpha^2K \sum_{i=1}^{K} \E\left[\Norm{\u_t^i - f_i(\u_t^{i-1})}^2\right]  \right).
\end{align*}
%Thus, we complete the proof of this lemma.
\end{proof}
Next, we bound the error of the gradient estimator.
%%%%%%%%%%%%%%%%%%%%%%%%%%%%%%   LEMMA  7    %%%%%%%%%%%%%%%%%%%%%%%%%%%%%%%%%%%%%%%%%%%%%%%%%
\begin{lemma} \label{lem:66}
Denote that the constants  $L_1 = \left(2K \LL_J^2 L_f^{2K-2} + 4KL_F^2\LL_f^2\right)\left(\sum_{i=1}^{K}\left(2\LL_f^2\right)^{i-1}\right)$, $L_2 = 2(2K^2 \sigma_J^2L_f^{2K-2} + 4  K^2 L_F^2 \sigma^2+2 L_1 \sigma^2 K) + 2L_1 D^2 +2L_f^{2K}$, we can ensure that:
\begin{align*}
	\frac{1}{T}\sum_{t=1}^T \E\left[\Norm{\nabla F(\x_t) - \v_t }^2\right] \leq \frac{L_2}{\alpha B_0 T} + \frac{\alpha L_2}{B_1} + \frac{L_2\eta^2}{\alpha B_1}.
\end{align*}
\end{lemma}
\begin{proof}
Since $\E\left[\Norm{\v_1 -\prod_{i=1}^K \nabla f_i(\u_1^{i-1})}^2\right] \leq \frac{ L_f^{2K}}{B_0}$, $\sum_{i=1}^K \E\left[\Norm{\u_1^i -  f_i(\u_1^{i-1})}^2\right] \leq \frac{K \sigma^2}{B_0}$, by setting $\alpha \leq \frac{B_1 L_F^2}{2L_1}$, we can deduce that:
\begin{align*}
 & \frac{1}{T}\sum_{t=1}^T\E\left[\Norm{\v_t - \prod_{i=1}^K \nabla f_i(\u_t^{i-1}) }^2\right] + \frac{2KL_F^2}{T}\sum_{t=1}^T \sum_{i=1}^{K-1} \E\left[\Norm{\u_t^i - f_i(\u_t^{i-1})}^2\right]\\
 \leq & \frac{ L_f^{2K} + 2K^2L_F^2 \sigma^2}{\alpha B_0 T} + (2K^2 \sigma_J^2L_f^{2K-2} + 4K^2L_F^2 \sigma^2)\frac{\alpha}{B_1}\\
 &+\frac{2K\LL_J^2L_f^{2K-2}+4KL_F^2\LL_f^2}{\alpha B_1 T}\sum_{t=1}^T \sum_{i=1}^{K} \E\left[\Norm{\u_{t+1}^{i-1} - \u_{t}^{i-1}}^2\right]\\
\leq & \frac{ L_f^{2K} + 2K^2L_F^2 \sigma^2}{\alpha B_0 T} + (2K^2 \sigma_J^2L_f^{2K-2} + 4 K^2L_F^2 \sigma^2)\frac{\alpha}{B_1} \\
 & +\frac{L_1}{\alpha B_1} \left(\eta^2 D^2 + \frac{2\alpha^2\sigma^2 K}{B_1}   + 2\alpha^2K \frac{1}{T}\sum_{t=1}^T \sum_{i=1}^{K} \E\left[\Norm{\u_t^i - f_i(\u_t^{i-1})}^2\right]  \right) \\
 \leq & \frac{ L_f^{2K} + 2K^2 L_F^2 \sigma^2}{\alpha B_0 T} + (2K^2 \sigma_J^2L_f^{2K-2} + 4  K^2 L_F^2 \sigma^2+2 L_1 \sigma^2 K)\frac{\alpha}{B_1}  + \frac{L_1\eta^2D^2}{\alpha B_1} \\
 & + \frac{KL_F^2}{T}\sum_{t=1}^T \sum_{i=1}^{K}\E\left[\Norm{\u_t^i - f_i(\u_t^{i-1})}^2\right]. 
\end{align*}
Thus, we have 
\begin{align*}
 & \frac{1}{T}\sum_{t=1}^T\E\left[\Norm{\v_t - \prod_{i=1}^K \nabla f_i(\u_t^{i-1}) }^2\right] + \frac{KL_F^2}{T}\sum_{t=1}^T \sum_{i=1}^{K-1} \E\left[\Norm{\u_t^i - f_i(\u_t^{i-1})}^2\right]\\
 \leq & \frac{ L_f^{2K} + 2K^2 L_F^2 \sigma^2}{\alpha B_0 T} + (2K^2 \sigma_J^2L_f^{2K-2} + 4  K^2 L_F^2 \sigma^2+2 L_1 \sigma^2 K)\frac{\alpha}{B_1}  + \frac{L_1\eta^2D^2}{\alpha B_1}. 
\end{align*}
Then, we can obtain:
\begin{align*}
&\frac{1}{T}\sum_{t=1}^T \E\left[\Norm{\v_t - \nabla F(\x_t)}^2\right]\\
\leq &\frac{2}{T}\sum_{t=1}^T \E\left[\Norm{\v_t - \prod_{i=1}^K\nabla f_i(\u_t^{i-1})}^2\right] + \frac{2}{T}\sum_{t=1}^T \E\left[\Norm{\prod_{i=1}^K\nabla f_i(\u_t^{i-1}) - \nabla F(\x_t)}^2\right] \\ \leq & \frac{2}{T}\sum_{t=1}^T\E\left[\Norm{\v_t - \prod_{i=1}^K\nabla f_i(\u_t^{i-1})}^2\right]  + \frac{2KL_F^2}{T}\sum_{t=1}^T \sum_{i=1}^{K-1} \E\left[\Norm{\u_t^i - f_i(\u_t^{i-1})}^2\right]\\
\leq & \frac{2L_f^{2K} + 4K^2 L_F^2 \sigma^2}{\alpha B_0 T} + 2(2K^2 \sigma_J^2L_f^{2K-2} + 4  K^2 L_F^2 \sigma^2+2 L_1 \sigma^2 K)\frac{\alpha}{B_1}  + \frac{2L_1\eta^2D^2}{\alpha B_1}\\
\leq &  \frac{L_2}{\alpha B_0 T} + \frac{\alpha L_2}{B_1} + \frac{L_2\eta^2}{\alpha B_1}
\end{align*}
\end{proof}	
\paragraph{The rest proof of the Theorem:} Denote the Frank-Wolfe gap as $\F(\x):=\max _{\hat{\x} \in \X}\langle \hat{\x}-\x,-\nabla F(\x)\rangle$ and $\z_t^{\star} = \arg \max _{\hat{\x} \in \X}\langle \hat{\x}-\x,-\nabla F(\x)\rangle$.
\begin{align*} 
F\left(\x_{t+1}\right) 
& \leq F\left(\x_t\right)+\left\langle\nabla F\left(\x_t\right), \x_{t+1}-\x_t\right\rangle+\frac{L_F}{2}\left\|\x_{t+1}-\x_{t}\right\|^{2} \\ 
& \leq F\left(\x_{t}\right)+\eta\left\langle\nabla F\left(\x_{t}\right), \z_{t}-\x_{t}\right\rangle+\eta^{2} \frac{L}{2} D^{2} \\ 
& =F\left(\x_{t}\right)+\eta\left\langle \v_{t}, \z_{t}-\x_{t}\right\rangle+\eta\left\langle\nabla F\left(\x_{t}\right)-\v_t, \z_t-\x_t\right\rangle+\eta^{2} \frac{L_F}{2} D^{2} \\ 
& \leq F\left(\x_{t}\right)+\eta\left\langle \v_t, \z^{\star}_{t}-\x_{t}\right\rangle+\eta\left\langle\nabla F\left(\x_{t}\right)-\v_{t}, \z_{t}-\x_{t}\right\rangle+\eta^{2} \frac{L_F}{2} D^{2} \\ 
& =F\left(\x_t\right)+\eta\left\langle\nabla F\left(\x_t\right), \z^{\star}_{t}-\x_{t}\right\rangle+\eta\left\langle\nabla F\left(\x_{t}\right)-\v_{t}, \z_t-\z^{\star}_{t}\right\rangle+\eta^{2} \frac{L_F}{2} D^{2} \\ 
& \leq F\left(\x_{t}\right)-\eta \F\left(\x_{t}\right)+\eta D\left\|\nabla F\left(\x_t\right)-\v_t\right\|+\eta^{2} \frac{L_F}{2} D^{2}
\end{align*}
That is to say:
\begin{align*} 
&\E\left[\F(\x_\tau)\right]=\E\left[ \frac{1}{T}\sum_{t=1}^{T}  \F\left(\x_{t}\right)\right] \\ \leq &\frac{\E\left[F\left(\x_{1}\right)-F(\x_{T+1})\right]}{\eta T}+D\cdot\E\left[\frac{1}{T} \sum_{t=1}^T \left\|\nabla F\left(\x_t\right)-\v_t\right\|\right]+\eta \frac{L_F}{2} D^{2}\\
\leq &\frac{\Delta_F}{\eta T}+D\sqrt{\E\left[\frac{1}{T} \sum_{t=1}^T \left\|\nabla F\left(\x_t\right)-\v_t\right\|^2\right]}+\eta \frac{L_F}{2} D^{2}\\
\leq &\frac{\Delta_F}{\eta T}+D\sqrt{\frac{L_2}{\alpha B_0 T} + \frac{\alpha L_2}{B_1} + \frac{L_2\eta^2}{\alpha B_1}}+\eta \frac{L_F}{2} D^{2}
\end{align*} 
By setting $T = \mathcal{O} \left(\epsilon^{-2}\right)$, $\eta =\alpha =\Theta \left(\epsilon\right)$, $B_0 = B_1 = \Omega \left(\epsilon^{-1}\right)$, we can ensure $\E\left[\F(\x_\tau)\right] \leq \epsilon$. By setting $T = \mathcal{O}\left( \epsilon^{-3}\right), \eta = \alpha=\Theta\left( \epsilon^2\right),  B_0 = \Omega \left(\epsilon^{-1}\right), B_1 =  \Omega \left(1 \right)$, we can also ensure $\E\left[\F(\x_\tau)\right] \leq \epsilon$.

\section{Proof of Theorem~\ref{thm:main+}}
Since $2^0+2^1+\cdots+2^{S-1} < 2^S$, running the algorithm for $T$ iterations guarantees at least $S = \lfloor \log (T) \rfloor $ complete stages. In the theoretical analysis, we can simply use the output of the last complete stage $S=\lfloor \log (T) \rfloor$, which has been at least run for $2^{S-1} \geq T/4$ iterations. Note that in the previous analysis, we have already proved that
\begin{align*} 
\E\left[ \frac{1}{T}\sum_{t=1}^{T}  \F\left(\x_{t}^s\right)\right] 
\leq &\frac{\E\left[F\left(\x_{1}^s\right)-F(\x_{T+1}^s)\right]}{\eta_s T_s}+D\sqrt{\frac{L_2}{\alpha_s B_0^s T_s} + \frac{\alpha_s L_2}{B_1^s} + \frac{L_2\eta_s^2}{\alpha_s B_1^s}}+\eta_s \frac{L_F}{2} D^{2}\\
\leq &\frac{L_f^K D}{\eta_s T_s}+D\sqrt{\frac{L_2}{\alpha_s B_0^s T_s} + \frac{\alpha_s L_2}{B_1^s} + \frac{L_2\eta_s^2}{\alpha_s B_1^s}}+\eta_s \frac{L_F}{2} D^{2},
\end{align*}
where the last inequality is due to the fact that function $F$ is $L_f^K$ Lipschitz continuous.

First, when setting $B_0^s = B_1^s = \sqrt{T_s}$ and $\eta_s = \alpha_s = 1/\sqrt{T_s}$, the hyper-parameters for the last complete stage $S$ is $T_S = T/4$, $B_0^S = B_1^S = \sqrt{T}/2$ and $\eta_S = \alpha_S = 2/\sqrt{T}$. Thus, the output of the last complete stage is
\begin{align*} 
\E\left[ \frac{1}{T}\sum_{t=1}^{T}  \F\left(\x_{t}^s\right)\right] \leq &\frac{2L_f^K D}{\sqrt{ T}}+D\sqrt{\frac{4L_2}{T} + \frac{4 L_2}{T} + \frac{4L_2}{T}}+\frac{L_F}{\sqrt{T}} D^{2} = \mathcal{O}\left(\frac{1}{\sqrt{T}}\right),
\end{align*} 

Second, when setting $B_0^s = T_s^{1/3}$, $B_1^s = 1$ and $\eta_s = \alpha_s = T_s^{-2/3}$, the hyper-parameters for the last complete stage $S$ is $T_s = T/4$, $B_0^S = (T/4)^{1/3}$, $B_1^S = 1$ and $\eta_S = \alpha_S = (T/4)^{-2/3}$. Thus, the output of the last complete stage is
\begin{align*} 
\E\left[ \frac{1}{T}\sum_{t=1}^{T}  \F\left(\x_{t}^s\right)\right] \leq  \mathcal{O}\left(\frac{1}{T^{1/3}}\right),
\end{align*}which completes the proof of this Theorem. 

\newpage
\section{Proof of Theorem~\ref{thm:2}}
In the previous analysis of Lemma~\ref{lem:66}, we simply reduce ${\eta^2}\Norm{\z_t - \x_t}^{2} \leq \eta^2 D^2$. To obtain the rate for gradient mapping, we have to keep this term. That is to say, we rewrite the Lemma~\ref{lem:66} as follows
\begin{align*}
 & \frac{1}{T}\sum_{t=1}^T\E\left[\Norm{\v_t - \prod_{i=1}^K \nabla f_i(\u_t^{i-1}) }^2\right] + \frac{2KL_F^2}{T}\sum_{t=1}^T \sum_{i=1}^{K-1} \E\left[\Norm{\u_t^i - f_i(\u_t^{i-1})}^2\right]\\
\leq & \frac{ L_f^{2K} + 2K^2L_F^2 \sigma^2}{\alpha B_0 T} + (2K^2 \sigma_J^2L_f^{2K-2} + 4 K^2L_F^2 \sigma^2)\frac{\alpha}{B_1} \\
 & +\frac{L_1}{\alpha B_1} \left(\eta^2 \frac{1}{T}\sum_{t=1}^T\Norm{\z_t - \x_t}^{2} + \frac{2\alpha^2\sigma^2 K}{B_1}   + 2\alpha^2K \frac{1}{T}\sum_{t=1}^T \sum_{i=1}^{K} \E\left[\Norm{\u_t^i - f_i(\u_t^{i-1})}^2\right]  \right) \\
 \leq & \frac{L_f^{2K} + 2K^2 L_F^2 \sigma^2}{\alpha B_0 T} + (2K^2 \sigma_J^2L_f^{2K-2} + 4  K^2 L_F^2 \sigma^2+2 L_1 \sigma^2 K)\frac{\alpha}{B_1}   \\
 & + \frac{L_1\eta^2}{\alpha B_1} \frac{1}{T}\sum_{t=1}^T\Norm{\z_t - \x_t}^{2}+ \frac{KL_F^2}{T}\sum_{t=1}^T \sum_{i=1}^{K}\E\left[\Norm{\u_t^i - f_i(\u_t^{i-1})}^2\right]. 
\end{align*}
Thus, by setting $L_3 = 2(2K^2 \sigma_J^2L_f^{2K-2} + 4  K^2 L_F^2 \sigma^2+2 L_1 \sigma^2 K) + 2L_1+2L_f^{2K}  $, we have
\begin{align*}
&\frac{1}{T}\sum_{t=1}^T \E\left[\Norm{\v_t - \nabla F(\x_t)}^2\right] \\ \leq & \frac{2}{T}\sum_{t=1}^T\E\left[\Norm{\v_t - \prod_{i=1}^K\nabla f_i(\u_t^{i-1})}^2\right]  + \frac{2KL_F^2}{T}\sum_{t=1}^T \sum_{i=1}^{K-1} \E\left[\Norm{\u_t^i - f_i(\u_t^{i-1})}^2\right]\\
\leq &  \frac{L_3}{\alpha B_0 T} + \frac{\alpha L_3}{B_1} + \frac{L_3\eta^2 }{\alpha B_1}\frac{1}{T}\sum_{t=1}^T\Norm{\z_t - \x_t}^{2}
\end{align*}
According to Proposition 2 of~\citet{NEURIPS2022_7e16384b}, we know that the gradient mapping satisfies
\begin{align*}
\left\|\mathcal{G}(\x_t, \beta)\right\|^{2} \leq -4 \beta g(\x_t, \v_t)+2\E\left[\Norm{\nabla F(\x_t) - \v_t }^2\right],  
\end{align*}  
where $g(\x_t, \v_t)=\min _{\y \in \X}\left\{\langle \v_t, \y-\x \rangle+\frac{\beta}{2}\|\y-\x_t\|^{2}\right\}$.

Due to the convergence of Frank-Wolfe algorithm~\cite{pmlr-v28-jaggi13}, we know that 
\begin{align*}
    \langle \v_t, \z_t-\x \rangle+\frac{\beta}{2}\|\z_t-\x_t\|^{2} \leq g(\x_t, \v_t) + \frac{2 \beta D^2}{N+2},
\end{align*}
which is widely used in the analysis of Frank-Wolfe algorithm~\citep{NEURIPS2022_7e16384b, Zhang2019OneSS,AAAI:2021:Wan:B,AAAI:2021:Wan:C}.

Denote $\y^{\star}= \min _{\y \in \X}\left\{\langle \v_t, \y-\x \rangle+\frac{\beta}{2}\|\y-\x_t\|^{2}\right\}$. Set  $\eta \leq \frac{\beta}{2 L_F}$, and then we have, 
\begin{align*} 
F\left(\x_{t+1}\right) 
\leq & F\left(\x_t\right)+\left\langle\nabla F\left(\x_t\right), \x_{t+1}-\x_t\right\rangle+\frac{L_F}{2}\left\|\x_{t+1}-\x_{t}\right\|^{2} \\ 
\leq & F\left(\x_{t}\right)+\eta\left\langle\nabla F\left(\x_{t}\right), \z_{t}-\x_{t}\right\rangle+\eta^{2} \frac{L_F}{2} \Norm{\z_t - \x_t}^{2} \\ 
=&F\left(\x_{t}\right)+\eta\left\langle \v_{t}, \z_{t}-\x_{t}\right\rangle+\eta\left\langle\nabla F\left(\x_{t}\right)-\v_t, \z_t-\x_t\right\rangle+\eta^{2} \frac{L_F}{2} \Norm{\z_t - \x_t}^{2} \\ 
=&F\left(\x_{t}\right)+\eta\left\langle \v_{t}, \z_{t}-\x_{t}\right\rangle+\frac{\eta\beta}{2}\|\z_t-\x_t\|^{2}+\eta\left\langle\nabla F\left(\x_{t}\right)-\v_t, \z_t-\x_t\right\rangle \\
 &\quad - \frac{\eta\beta}{2}\|\z_t-\x_t\|^{2} +\eta^{2} \frac{L_F}{2} \Norm{\z_t - \x_t}^{2}\\ 
\leq & F\left(\x_{t}\right)+\eta\left\langle \v_{t}, \y^{\star}-\x_{t}\right\rangle+\frac{\eta\beta}{2}\|\y^{\star}-\x_t\|^{2}+\frac{2 \eta \beta D^2}{N+2}\\
&\quad +\eta\left\langle\nabla F\left(\x_{t}\right)-\v_t, \z_t-\x_t\right\rangle - \frac{\eta\beta}{2}\|\z_t-\x_t\|^{2} +\eta^{2} \frac{L_F}{2} \Norm{\z_t - \x_t}^{2} \\
\leq & F(\x_t) + \eta g(\x_t, \v_t) + \frac{2\beta D^2  \eta}{N+2} +\eta\left\langle\nabla F\left(\x_{t}\right)-\v_t, \z_t-\x_t\right\rangle - \frac{\eta\beta}{4}\|\z_t-\x_t\|^{2} \\
\leq & F(\x_t) + \eta g(\x_t, \v_t) + \frac{2\beta D^2  \eta}{N+2} +\frac{2\eta}{\beta} \E\left[\Norm{\nabla F(\x_t) - \v_t }^2\right] - \frac{\eta\beta}{8}\|\z_t-\x_t\|^{2}.
\end{align*}    
As a result,
\begin{align*} 
-g(\x_t, \v_t)
& \leq \frac{F(\x_t) - F\left(\x_{t+1}\right)}{\eta}  + \frac{2\beta D^2 }{N+2} +\frac{2}{\beta} \E\left[\Norm{\nabla F(\x_t) - \v_t }^2\right] - \frac{\beta}{8}\|\z_t-\x_t\|^{2}.
\end{align*}
So, we have:
\begin{align*}
  \left\|\mathcal{G}(\x_t, \beta)\right\|^{2} \leq &-4 \beta g(\x_t, \v_t)+2\E\left[\Norm{\nabla F(\x_t) - \v_t }^2\right]  \\
  \leq &\frac{4\beta (F(\x_t) - F\left(\x_{t+1}\right))}{\eta}  + \frac{8\beta^2 D^2 }{N+2} + 10 \E\left[\Norm{\nabla F(\x_t) - \v_t }^2\right] - \frac{\beta^2}{2}\|\z_t-\x_t\|^{2}.
\end{align*}  
Finally, by setting $\eta \leq \sqrt{\frac{\beta^2 \alpha B_1}{20 L_3}}$, we have
\begin{align*}
  &\quad \frac{1}{T}\sum_{t=1}^{T}\left\|\mathcal{G}(\x_t, \beta)\right\|^{2}   \\
  & \leq \frac{4\beta \Delta_F}{\eta T}  + \frac{8\beta^2 D^2 }{N} + \frac{10}{T}\sum_{t=1}^{T}\E\left[\Norm{\nabla F(\x_t) - \v_t }^2\right] - \frac{\beta^2}{2} \frac{1}{T}\sum_{t=1}^{T}\|\z_t-\x_t\|^{2} \\
  & \leq \frac{4\beta \Delta_F}{\eta T}  + \frac{8\beta^2 D^2 }{N} + \frac{10L_3}{ \alpha B_0 T}+ \frac{10 L_3\alpha}{B_1} + \left(  \frac{10 \eta^2 L_3} {\alpha B_1} -\frac{\beta^2}{2}\right)\frac{1}{T}  \sum_{t=1}^{T}\Norm{\z_t - \x_t}^{2} \\
  & \leq \frac{4\beta \Delta_F}{\eta T}  + \frac{8\beta^2 D^2 }{N} + \frac{10L_3}{ \alpha B_0 T}+ \frac{10 L_3\alpha}{B_1}
\end{align*}  
By setting $\alpha = \Theta(\sqrt{\epsilon})$, $\eta = \Theta(1)$, $N=T=\mathcal{O}(\epsilon^{-1})$, $B_0 = B_1=\Omega(\epsilon^{-0.5})$, We can ensure that $\E\left[\Norm{\G(\x_\tau,\beta)}^2\right]\leq \epsilon$. This guarantee can also be satisfied by setting $\alpha =\Theta\left( \epsilon\right), \eta =\Theta\left(\epsilon^{0.5}\right), T = \mathcal{O}\left( \epsilon^{-1.5} \right), B_0 = \Omega(\epsilon^{-0.5}), B_1 =  \Omega\left(1 \right), N = \mathcal{O}(\epsilon^{-1})$.

\section{Proof of Theorem~\ref{thm:main-}}
First, we bound the estimation error of the function estimator $\u_t^i$, as well as the difference between iteration steps.
\begin{lemma}\label{L1} By ensuring that $B_1 \leq \alpha^2 B_0 T$, we can obtain the following inequalities:
\begin{align*}
	\frac{1}{T}\sum_{t=1}^T  \E\left[\Norm{\u_t^i -  f_i(\u_t^{i-1})}^2\right] \leq &\frac{2\LL_f^2}{\alpha^2 T}  \sum_{t=1}^T \E\left[\Norm{\u_{t+1}^{i-1} - \u_{t}^{i-1}}^2\right] + \frac{2\alpha   \sigma^2}{B_1};\\
 \frac{1}{T}\sum_{t=1}^T \E\left[\Norm{\u_{t+1}^{i} - \u_{t}^{i}}^2\right] \leq & \frac{\alpha^2\sigma^2}{B_1}  + \frac{2\alpha^2\LL_f^2}{T}\sum_{t=1}^T \E\left[\left\| \u_{t+1}^{i-1} -\u_{t}^{i-1} \right\|^2 \right]\\
 &+  \frac{2\alpha^2}{T}\sum_{t=1}^T \E\left[\left\| f_{i}(\u_{t}^{i-1}) - \u_{t}^{i} \right\|^2 \right].
\end{align*}
\end{lemma}	
\textbf{Proof}
Considering that $\u_t^i = (1-\alpha)\u_{t-1}^i  + \alpha \frac{1}{B_1}\sum_{j=1}^{B_1}f_i(\u_t^{i-1};\xi_t^{i,j}) $, we have:
\begin{align*}
& \E\left[\Norm{\u_t^i - f_i(\u_t^{i-1})}^2\right]\\
=& \E\left[\left\|(1-\alpha)\left(\u_{t-1}^i - f_i(\u_{t}^{i-1})\right)   +\alpha\left(\frac{1}{B_1}\sum_{j=1}^{B_1}   f_i(\u_{t}^{i-1};\xi_t^{i,j}) - f_i(\u_t^{i-1}) \right)\right\|^2\right] \\
=& (1-\alpha)^2 \E\left[\left\|\u_{t-1}^i - f_i(\u_{t}^{i-1})  \right\|^2\right] + \alpha^2 \E\left[\left\|\frac{1}{B_1}\sum_{j=1}^{B_1}   f_i(\u_{t}^{i-1};\xi_t^{i,j}) - f_i(\u_t^{i-1}) \right\|^2\right]\\
\leq & (1-\alpha)^2(1+\alpha) \E\left[\left\|\u_{t-1}^i - f_i(\u_{t-1}^{i-1})  \right\|^2\right]+(1-\alpha)^2(1+\frac{1}{\alpha}) \E\left[\left\|f_i(\u_{t-1}^{i-1}) - f_i(\u_{t}^{i-1})  \right\|^2\right] + \frac{\alpha^2\sigma^2}{B_1} \\
\leq & (1-\alpha) \E\left[\left\|\u_{t-1}^i - f_i(\u_{t-1}^{i-1})  \right\|^2\right]+\frac{2}{\alpha} \E\left[\left\|f_i(\u_{t-1}^{i-1}) - f_i(\u_{t}^{i-1})  \right\|^2\right] + \frac{\alpha^2\sigma^2}{B_1} \\
\leq & (1-\alpha) \E\left[\left\|\u_{t-1}^i - f_i(\u_{t-1}^{i-1})  \right\|^2\right]+\frac{2\LL_f^2}{\alpha} \E\left[\left\|\u_{t-1}^{i-1} - \u_{t}^{i-1}  \right\|^2\right] + \frac{\alpha^2\sigma^2}{B_1}
\end{align*}
Summing up and rearranging, we can obtain:
\begin{align*}
	\frac{1}{T}\sum_{t=1}^T  \E\left[\Norm{\u_t^i -  f_i(\u_t^{i-1})}^2\right] \leq &\frac{\E\left[\Norm{\u_1^i -  f_i(\u_1^{i-1})}^2\right]}{\alpha T} + \frac{2\LL_f^2}{\alpha^2 T}  \sum_{t=1}^T \E\left[\Norm{\u_{t+1}^{i-1} - \u_{t}^{i-1}}^2\right] + \frac{\alpha   \sigma^2}{B_1} \\
 \leq &\frac{\sigma^2}{\alpha T B_0} + \frac{2\LL_f^2}{\alpha^2 T}  \sum_{t=1}^T \E\left[\Norm{\u_{t+1}^{i-1} - \u_{t}^{i-1}}^2\right] + \frac{\alpha   \sigma^2}{B_1}\\
 \leq &\frac{2\LL_f^2}{\alpha^2 T}  \sum_{t=1}^T \E\left[\Norm{\u_{t+1}^{i-1} - \u_{t}^{i-1}}^2\right] + \frac{2\alpha   \sigma^2}{B_1}
\end{align*}
where the last inequality is by setting $B_1 \leq \alpha^2 B_0 T$.

Next, we bound the difference of function estimators between adjacent steps:
\begin{align*}
		& \E\left[\Norm{\u_{t+1}^{i} - \u_{t}^{i}}^2\right] =\alpha^2\E\left[\left\|\frac{1}{B_1}\sum_{j=1}^{B_1}f_{i}(\u_{t+1}^{i-1};\xi_{t+1}^{i,j}) - \u_{t}^{i} \right\|^2 \right]\\
  =&\alpha^2\E\left[\left\|\frac{1}{B_1}\sum_{j=1}^{B_1}f_{i}(\u_{t+1}^{i-1};\xi_{t+1}^{i,j})-f_{i}(\u_{t+1}^{i-1}) + f_{i}(\u_{t+1}^{i-1}) - \u_{t}^{i} \right\|^2 \right]\\
  =&\alpha^2\E\left[\left\|\frac{1}{B_1}\sum_{j=1}^{B_1}f_{i}(\u_{t+1}^{i-1};\xi_{t+1}^{i,j})-f_{i}(\u_{t+1}^{i-1})  \right\|^2 \right] + \alpha^2\E\left[\left\| f_{i}(\u_{t+1}^{i-1}) - \u_{t}^{i} \right\|^2 \right]\\
  =&\frac{\alpha^2\sigma^2}{B_1}  +\alpha^2\E\left[\left\| f_{i}(\u_{t+1}^{i-1}) -f_{i}(\u_{t}^{i-1}) + f_{i}(\u_{t}^{i-1}) - \u_{t}^{i} \right\|^2 \right]\\
  \leq &\frac{\alpha^2\sigma^2}{B_1}  +2\alpha^2\E\left[\left\| f_{i}(\u_{t+1}^{i-1}) -f_{i}(\u_{t}^{i-1}) \right\|^2 \right] + 2\alpha^2\E\left[\left\| f_{i}(\u_{t}^{i-1}) - \u_{t}^{i} \right\|^2 \right]\\
  \leq &\frac{\alpha^2\sigma^2}{B_1}  +2\alpha^2\LL_f^2 \E\left[\left\| \u_{t+1}^{i-1} -\u_{t}^{i-1} \right\|^2 \right] + 2\alpha^2\E\left[\left\| f_{i}(\u_{t}^{i-1}) - \u_{t}^{i} \right\|^2 \right].
	\end{align*} 
Summing up, we obtain that
\begin{align*}
		\frac{1}{T}\sum_{t=1}^T \E\left[\Norm{\u_{t+1}^{i} - \u_{t}^{i}}^2\right] \leq & \frac{\alpha^2\sigma^2}{B_1}  + \frac{2\alpha^2\LL_f^2}{T}\sum_{t=1}^T \E\left[\left\| \u_{t+1}^{i-1} -\u_{t}^{i-1} \right\|^2 \right] \\
  &+  \frac{2\alpha^2}{T}\sum_{t=1}^T \E\left[\left\| f_{i}(\u_{t}^{i-1}) - \u_{t}^{i} \right\|^2 \right].
	\end{align*} 

\begin{lemma}\label{L2} By setting that $C_{i+1} = 5+6\LL_f^2 C_i$, we can obtain the following inequalities:
\begin{align*}
	\frac{1}{T}\sum_{t=1}^T  \E\left[\Norm{\u_t^i -  f_i(\u_t^{i-1})}^2\right] \leq C_i \left(\frac{\eta^2 D^2}{\alpha^2}  + \frac{ \sigma^2}{B_1} \right)\\
 \frac{1}{T}\sum_{t=1}^T \E\left[\Norm{\u_{t+1}^{i} - \u_{t}^{i}}^2\right]\leq C_i \alpha^2 \left(\frac{\eta^2 D^2}{\alpha^2}  + \frac{ \sigma^2}{B_1} \right)
\end{align*}
\end{lemma}	
\textbf{Proof}
We would prove by Induction.

(1) For $i=1$, by noting that $\u_t^0 = \x_t$ and setting $C_1=5+6\LL_f^2$, we have

\begin{align*}
	&\frac{1}{T}\sum_{t=1}^T  \E\left[\Norm{\u_t^1 -  f_1(\u_t^{0})}^2\right]\\ 
 \leq &\frac{2\LL_f^2}{\alpha^2 T}  \sum_{t=1}^T \E\left[\Norm{\u_{t+1}^{0} - \u_{t}^{0}}^2\right] + \frac{2\alpha   \sigma^2}{B_1}
 \leq  \frac{2\LL_f^2}{\alpha^2 T}  \sum_{t=1}^T \E\left[\Norm{\x_{t+1} - \x_{t}}^2\right] + \frac{2\alpha   \sigma^2}{B_1}\\
 \leq & \frac{2\LL_f^2}{\alpha^2} \eta^2 D^2 + \frac{2\alpha   \sigma^2}{B_1} \leq (2\LL_f^2+2) \left(\frac{\eta^2 D^2}{\alpha^2}  + \frac{ \sigma^2}{B_1} \right) \leq C_1 \left(\frac{\eta^2 D^2}{\alpha^2}  + \frac{ \sigma^2}{B_1} \right),
\end{align*}
as well as the following guarantee
\begin{align*}
 &\frac{1}{T}\sum_{t=1}^T \E\left[\Norm{\u_{t+1}^{1} - \u_{t}^{1}}^2\right]\\ \leq & \frac{\alpha^2\sigma^2}{B_1}  + \frac{2\alpha^2\LL_f^2}{T}\sum_{t=1}^T \E\left[\left\| \u_{t+1}^{0} -\u_{t}^{0} \right\|^2 \right] +  \frac{2\alpha^2}{T}\sum_{t=1}^T \E\left[\left\| f_{1}(\u_{t}^{0}) - \u_{t}^{1} \right\|^2 \right]\\
 \leq & \frac{\alpha^2\sigma^2}{B_1}  + 2\alpha^2\LL_f^2 \eta^2 D^2 +  2\alpha^2\left(\frac{2\LL_f^2}{\alpha^2} \eta^2 D^2 + \frac{2\alpha   \sigma^2}{B_1}\right)\\
 \leq & (6\LL_f^2+5)\alpha^2\left(\frac{\eta^2 D^2}{\alpha^2}  + \frac{\sigma^2}{B_1}\right) \leq C_1 \alpha^2 \left(\frac{\eta^2 D^2}{\alpha^2}  + \frac{ \sigma^2}{B_1} \right)
\end{align*}
(2) Suppose that for $i=k$, we have that
$\frac{1}{T}\sum_{t=1}^T  \E\left[\Norm{\u_t^k -  f_k(\u_t^{k-1})}^2\right] \leq C_k \left(\frac{\eta^2 D^2}{\alpha^2}  + \frac{ \sigma^2}{B_1} \right)$,
as well as $ \frac{1}{T}\sum_{t=1}^T \E\left[\Norm{\u_{t+1}^{k} - \u_{t}^{k}}^2\right]\leq C_k \alpha^2 \left(\frac{\eta^2 D^2}{\alpha^2}  + \frac{ \sigma^2}{B_1} \right)$.
Then, for $i=k+1$, we have:
\begin{align*}
  \frac{1}{T}\sum_{t=1}^T  \E\left[\Norm{\u_t^{k+1} -  f_{k+1}(\u_t^{k})}^2\right]\leq & \frac{2\LL_f^2}{\alpha^2 T}  \sum_{t=1}^T \E\left[\Norm{\u_{t+1}^{k} - \u_{t}^{k}}^2\right] + \frac{2\alpha   \sigma^2}{B_1} \\
  \leq &\frac{2\LL_f^2}{\alpha^2}C_k \alpha^2 \left(\frac{\eta^2 D^2}{\alpha^2}  + \frac{ \sigma^2}{B_1} \right) + \frac{2\alpha   \sigma^2}{B_1}\\
  \leq &(2+2\LL_f^2 C_k) \left(\frac{\eta^2 D^2}{\alpha^2}  + \frac{ \sigma^2}{B_1} \right) \leq C_{k+1} \left(\frac{\eta^2 D^2}{\alpha^2}  + \frac{ \sigma^2}{B_1} \right), 
\end{align*}
as well as the fact that
\begin{align*}
 &\frac{1}{T}\sum_{t=1}^T \E\left[\Norm{\u_{t+1}^{k+1} - \u_{t}^{k+1}}^2\right] \\
 \leq & \frac{\alpha^2\sigma^2}{B_1}  + \frac{2\alpha^2\LL_f^2}{T}\sum_{t=1}^T \E\left[\left\| \u_{t+1}^{k} -\u_{t}^{k} \right\|^2 \right] +  \frac{2\alpha^2}{T}\sum_{t=1}^T \E\left[\left\| f_{k+1}(\u_{t}^{k}) - \u_{t}^{k+1} \right\|^2 \right]\\ 
 \leq & \frac{\alpha^2\sigma^2}{B_1}  + 2\alpha^2\LL_f^2 C_k \alpha^2 \left(\frac{\eta^2 D^2}{\alpha^2}  + \frac{ \sigma^2}{B_1} \right) +  2\alpha^2(2+2\LL_f^2 C_k) \left(\frac{\eta^2 D^2}{\alpha^2}  + \frac{ \sigma^2}{B_1} \right)\\ 
 \leq & (5+6\LL_f^2 C_k) \alpha^2\left(\frac{\eta^2 D^2}{\alpha^2}  + \frac{ \sigma^2}{B_1} \right) \leq C_{k+1} \alpha^2\left(\frac{\eta^2 D^2}{\alpha^2}  + \frac{ \sigma^2}{B_1} \right).
\end{align*}
Thus, we have proved this lemma via Induction. Also, by setting that $C_0=1$ and noting that $C_{i+1} = 5+6L_f^2 C_i$ for $i\ge 1$, we know that $C_i = (5+6L_f^2)^i$.

\begin{lemma}
We have the following guarantee:
    \begin{align*}
    &\frac{1}{T}\sum_{t=1}^T \E\left[\Norm{\v_{t}-\nabla F(\x_{t})}^2\right] \\
    \leq &\frac{(\sigma^2+L_f^2)^K}{B_0 \alpha T}+\frac{4L_F^2}{\alpha^2}\eta^2 D^2 +4K L_F^2\sum_{i=1}^{K-1} C_i \left(\frac{\eta^2 D^2}{\alpha^2}  + \frac{ \sigma^2}{B_1} \right)+\alpha\frac{K^2 \sigma^2}{B_1} (\sigma^2 +L_f^2 )^{K-1} 
\end{align*}
\end{lemma}
\textbf{Proof} Considering the definition of $\v_{t+1}$, we have
    \begin{align*}
        &\E\left[\left\|\v_{t+1} - \nabla F\left(\x_{t+1}\right) \right\|^2\right] \\
        %\leq & \E\left[\Norm{(1-\alpha)(\v_{t}-\nabla F(\x_{t+1}))  +\alpha\left( \frac{1}{B_1} \sum_{j=1}^{B_1} \left[ \prod_{i=1}^K \nabla f_i(\u_{t+1}^{i-1};\xi_{t+1}^{i,j}) \right] -\nabla F(\x_{t+1})\right)}^2\right]\\
        = & \E\left[\Norm{(1-\alpha)(\v_{t}-\nabla F(\x_{t+1}))  +\alpha\left( \frac{1}{B_1} \sum_{j=1}^{B_1} \left[ \prod_{i=1}^K \nabla f_i(\u_{t+1}^{i-1};\xi_{t+1}^{i,j}) \right] -\prod_{i=1}^K \nabla f_i(\u_{t+1}^{i-1})\right)\right.\right.\\
        &\quad \left.\left.+\alpha\left( \prod_{i=1}^K \nabla f_i(\u_{t+1}^{i-1}) -\nabla F(\x_{t+1})\right)}^2\right]\\
        \leq & \E\left[\Norm{(1-\alpha)(\v_{t}-\nabla F(\x_{t+1}))  +\alpha\left( \prod_{i=1}^K \nabla f_i(\u_{t+1}^{i-1}) -\nabla F(\x_{t+1})\right)}^2 \right]\\
        &+\alpha^2\E\left[\Norm{\frac{1}{B_1} \sum_{j=1}^{B_1} \left[ \prod_{i=1}^K \nabla f_i(\u_{t+1}^{i-1};\xi_{t+1}^{i,j}) \right] -\prod_{i=1}^K \nabla f_i(\u_{t+1}^{i-1})}^2\right]\\
        \leq & \E\left[\Norm{(1-\alpha)(\v_{t}-\nabla F(\x_{t}))+(1-\alpha)(\nabla F(\x_{t}) -\nabla F(\x_{t+1}))  \right.\right.\\
        &\left.\left.+\alpha\left( \prod_{i=1}^K \nabla f_i(\u_{t+1}^{i-1}) -\nabla F(\x_{t+1})\right)}^2\right] +\alpha^2\frac{K^2}{B_1} (\sigma^2 +L_f^2 )^{K-1}\sigma^2\\
        \leq & (1-\alpha)\E\left[\Norm{\v_{t}-\nabla F(\x_{t})}^2\right] +\frac{4}{\alpha}\Norm{\nabla F(\x_{t}) -\nabla F(\x_{t+1})}^2\\
        &+4\alpha\E\left[\Norm{ \prod_{i=1}^K \nabla f_i(\u_{t+1}^{i-1}) -\nabla F(\x_{t+1})}^2\right]+\frac{\alpha^2}{B_1^2} \sum_{j=1}^{B_1} \E\left[\Norm{ \prod_{i=1}^K \nabla f_i(\u_{t+1}^{i-1};\xi_{t+1}^{i,j})  -\prod_{i=1}^K \nabla f_i(\u_{t+1}^{i-1})}^2\right]\\
        \leq & (1-\alpha)\E\left[\Norm{\v_{t}-\nabla F(\x_{t})}^2\right] +\frac{4L_F^2}{\alpha}\eta^2 D^2 +4\alpha K L_F^2 \sum_{i=1}^{K-1} \E\left[\Norm{f_i(\u_{t+1}^{i-1})-\u_{t+1}^i}^2\right]\\
        &+\alpha^2\frac{K^2 \sigma^2}{B_1} (\sigma^2 +L_f^2 )^{K-1}.
    \end{align*}
Summing up and rearranging, we proved the conclusion by noting that $ \E\left[\Norm{\v_{1}-\nabla F(\x_{1})}^2\right] \leq {(\sigma^2+L_f^2)^K}/{B_0}$ and $\frac{1}{T}\sum_{t=1}^T  \E\left[\Norm{\u_{t+1}^i -  f_i(\u_{t+1}^{i-1})}^2\right] \leq C_i \left(\frac{\eta^2 D^2}{\alpha^2}  + \frac{ \sigma^2}{B_1} \right)$.

\textbf{Proof of the Theorem:} Next, we can finish the proof as follows.
Set $B_0 = 1$, $B_1 = T$, $\eta = 1/\sqrt{T}$ and $\alpha$ is a positive constant within $(0,1)$, and we have:
\begin{align*} 
&\E\left[\frac{1}{T}\sum_{t=1}^{T}  \F\left(\x_{t}\right)\right] 
\leq \frac{1}{\eta T}+D\sqrt{\E\left[\frac{1}{T} \sum_{t=1}^T \left\|\nabla F\left(\x_t\right)-\v_t\right\|^2\right]}+\eta \frac{L_F}{2} D^{2}\\
\leq &\mathcal{O}\left(\frac{1}{\eta T}+\sqrt{\frac{1}{B_0 \alpha T}+\frac{1}{\alpha^2}\eta^2  +\frac{\eta^2}{\alpha^2}  + \frac{ 1}{B_1} +\alpha\frac{1}{B_1} }+\eta \right)\\
\leq & \mathcal{O}\left(\frac{1}{\sqrt{T}}\right).
\end{align*} 
%To ensure that $\E\left[\frac{1}{T}\sum_{t=1}^{T}  \F\left(\x_{t}\right)\right] \leq \epsilon$, the iteration number $T=\mathcal{O}(\epsilon^{-2})$ and thus the overall complexity is $B_1 T = T^2= \mathcal{O}(\epsilon^{-4})$

\section{Proof of Theorem~\ref{thm:2-}}
In the previous analysis of Lemma~\ref{L2}, we simply reduce $\frac{1}{T}\sum_{t=1}^T \Norm{\z_t -\x_t}^2 \leq  D^2$. In this case, we keep this term and rewrite Lemma~\ref{L2} as follows.
\begin{lemma}\label{L4} By setting that $C_{i+1} = 5+6\LL_f^2 C_i$, we can obtain the following inequalities:
\begin{align*}
	\frac{1}{T}\sum_{t=1}^T  \E\left[\Norm{\u_t^i -  f_i(\u_t^{i-1})}^2\right] \leq C_i \left(\frac{\eta^2}{\alpha^2} \left(\frac{1}{T}\sum_{t=1}^T \Norm{\z_t -\x_t}^2\right) + \frac{ \sigma^2}{B_1} \right),\\
 \frac{1}{T}\sum_{t=1}^T \E\left[\Norm{\u_{t+1}^{i} - \u_{t}^{i}}^2\right]\leq C_i \alpha^2 \left(\frac{\eta^2}{\alpha^2} \left(\frac{1}{T}\sum_{t=1}^T \Norm{\z_t -\x_t}^2\right) + \frac{ \sigma^2}{B_1} \right).
\end{align*}
\end{lemma}	
Then, following a very similar analysis, we can have that
\begin{align*}
    &\frac{1}{T}\sum_{t=1}^T \E\left[\Norm{\v_{t}-\nabla F(\x_{t})}^2\right] \leq \frac{(\sigma^2+L_f^2)^K}{B_0 \alpha T}+\frac{4L_F^2}{\alpha^2}\eta^2 \left(\frac{1}{T}\sum_{t=1}^T \Norm{\z_t -\x_t}^2\right)\\
    &\quad +4K L_F^2\sum_{i=1}^{K-1} C_i \left(\frac{\eta^2}{\alpha^2}\left(\frac{1}{T}\sum_{t=1}^T \Norm{\z_t -\x_t}^2\right)  + \frac{ \sigma^2}{B_1} \right)+\alpha\frac{K^2 \sigma^2}{B_1} (\sigma^2 +L_f^2 )^{K-1}. 
\end{align*}
Finally, by setting  $N=B_1=T$, $\eta \leq \min \left\{\frac{\beta}{4L_F}, \frac{\alpha \beta}{\sqrt{2\left(4L_F^2+4K L_F^2\sum_{i=1}^{K-1} C_i\right)}} \right\}$, $\alpha=\Theta(1)$ and $B_0=\Omega(1)$, we have
\begin{align*}
  &\quad \E\left[\frac{1}{T}\sum_{t=1}^{T}\left\|\mathcal{G}(\x_t, \beta)\right\|^{2} \right]  \\
  & \leq \frac{4\beta \Delta_F}{\eta T}  + \frac{8\beta^2 D^2 }{N} + \frac{10}{T}\sum_{t=1}^{T}\E\left[\Norm{\nabla F(\x_t) - \v_t }^2\right] -\frac{\beta^2}{2} \frac{1}{T}\sum_{t=1}^{T}\|\z_t-\x_t\|^{2} \\
  & \leq \frac{4\beta \Delta_F}{\eta T}  + \frac{8\beta^2 D^2 }{N} + \frac{(\sigma^2+L_f^2)^K}{B_0 \alpha T}+\frac{ \sigma^2}{B_1}\left(4K L_F^2\sum_{i=1}^{K-1} C_i +K^2  (\sigma^2 +L_f^2 )^{K-1}\right)\\
    &\quad + \left(\frac{1}{T}\sum_{t=1}^T \Norm{\z_t -\x_t}^2\right)\left(\frac{\eta^2}{\alpha^2}\left(4L_F^2+4K L_F^2\sum_{i=1}^{K-1} C_i\right) - \frac{\beta^2}{2} \right)\\
    & \leq \mathcal{O}\left( \frac{1}{\eta T}  + \frac{1}{N} + \frac{1}{B_0 \alpha T}+\frac{1}{B_1} \right)\\
    &\leq \mathcal{O}\left( \frac{1}{T}\right).
\end{align*}  
Thus, we finish the proof for this Theorem.

\section{Proof of Theorem~\ref{thm:3}}
According to the equation~(C.21) of \citet{pmlr-v97-yurtsever19b}, for Algorithm~\ref{alg:1}~(PMVR) with convex objectives, we have that
\begin{align*}
\mathbb{E}\left[F\left(\x_{t+1}\right)\right]-F_{\star} \leq\left(1-\eta\right)\left(\mathbb{E}\left[F\left(\x_{t}\right)\right]-F_{\star}\right)+\eta D \mathbb{E}\left\|\nabla F\left(\x_{t}\right)-\v_{t}\right\|+\eta^{2}\frac{L_F D^2}{2}  
\end{align*}
Then, we have:
\begin{align*}
 &\frac{1}{T}\sum_{t=1}^T \mathbb{E}\left[F\left(\x_{t}\right)\right]-F_{\star} \\
 \leq &\frac{\left(\mathbb{E}\left[F\left(\x_{1}\right)\right]-F_{\star}\right)}{\eta T}+\frac{ D}{T} \sum_{t=1}^T \mathbb{E}\left\|\nabla F\left(\x_{t}\right)-\v_{t}\right\|+\eta \frac{L_F D^2}{2} \\
  \leq & \frac{\left(\mathbb{E}\left[F\left(\x_{1}\right)\right]-F_{\star}\right)}{\eta T}+D\sqrt{\frac{ 1}{T} \sum_{t=1}^T \mathbb{E}\left\|\nabla F\left(\x_{t}\right)-\v_{t}\right\|^2}+\eta \frac{L_F D^2}{2} \\
  \leq & \frac{\left(\mathbb{E}\left[F\left(\x_{1}\right)\right]-F_{\star}\right)}{\eta T}+D\sqrt{\frac{2\E\left[\Norm{\v_1 -\prod_{i=1}^K \nabla f_i(\u_1^{i-1})}^2\right] + 2K L_F^2\sum_{i=1}^K \E\left[\Norm{\u_1^i -  f_i(\u_1^{i-1})}^2\right] }{\alpha T}}\\
  & +D\sqrt{\frac{\alpha L_2}{B_1} + \frac{\eta^2  L_2}{\alpha B_1}} +\eta \frac{L_F D^2}{2}, 
\end{align*}
Next, we denote $\Gamma_s = \frac{1}{T}\sum_{t=1}^{T} \Norm{\v_t -\prod_{i=1}^K \nabla f_i(\u_t^{i-1})}^2 + \frac{KL_F^2}{T}\sum_{t=1}^{T}\sum_{i=1}^K \Norm{\u_t^i -  f_i(\u_t^{i-1})}^2 $ for stage $s$ and we denote $\x^s$ as the output of Algorithm\ref{alg:3}. Then, we have:
\begin{align*}
\mathbb{E}\left[F\left(\x^{s}\right)-F_{\star}\right] \leq \frac{\mathbb{E}\left[F\left(\x^{s-1}\right)-F_{\star}\right]}{\eta_s T_s}+D\sqrt{\frac{2 \Gamma_{s-1} }{\alpha_s  T_s}} +D\sqrt{\frac{\alpha_s L_2}{B_1^s} + \frac{\eta_s^2  L_2}{\alpha_s B_1^s}} +\eta_s \frac{L_F D^2}{2}. 
\end{align*}
Also, according to the previous analysis, we have that:
\begin{align*}
    \Gamma_{s} \leq&  \frac{\Gamma_{s-1}}{\alpha_{s} T_{s}} + \frac{\alpha_{s} L_2}{B_1^{s}} + \frac{L_2 \eta_{s}^2}{\alpha_{s} B_1^{s}}
\end{align*}
Set $\epsilon_s = \left(\frac{1}{2}\right)^{s-1} $, $\eta_{s} = \alpha_{s} \leq \frac{\epsilon_s}{2L_F D^2},B_1^s \geq (64 L_2  D^2+3L_2)\frac{\eta_s}{\epsilon_s^2} $, $T_s \geq \frac{128D^2+12+4\Delta_F+3\Gamma_0+32D^2\Gamma_0}{\eta_s}$. We can guarantee that $\mathbb{E}\left[F\left(\x^{s}\right)-F_{\star}\right] \leq \epsilon_s$ and $\E\left[\Gamma_{s} \right] \leq \epsilon_s^2$.
We will use  Induction to prove.
\begin{proof}
When $s=1$, we have:
\begin{align*}
\mathbb{E}\left[F\left(\x^{1}\right)-F_{\star}\right] & \leq \frac{\Delta_F}{\eta_1 T_1}+D\sqrt{\frac{2 \Gamma_{0} }{\alpha_1  T_1}} +D\sqrt{\frac{\alpha_1 L_2}{B_1^1} + \frac{\eta_1^2  L_2}{\alpha_1 B_1^1}} +\eta_1 \frac{L_F D^2}{2} \leq \epsilon_1 =1 
\end{align*}
where $\Gamma_0 = L_f^{2K} + 2K^2 L_F^2 \sigma^2$. Also, we have that:
\begin{align*}
     \Gamma_1 &\leq \frac{\Gamma_0}{\alpha_1 T_1} + \frac{\alpha_1 L_2}{B_1^1} + \frac{L_2 \eta_1^2}{\alpha_1 B_1^1}\leq  \epsilon_1^2 = 1 
 \end{align*}
Then, assume  $\mathbb{E}\left[F\left(\x^{s}\right)-F_{\star}\right] \leq \epsilon_s$ and $\mathbb{E}\left[\Gamma_s\right] \leq \epsilon_s^2$, we prove that it holds for stage $s+1$.
\begin{align*}
& \mathbb{E}\left[F\left(\x^{s+1}\right)-F_{\star}\right]  \\
&\leq \frac{\epsilon_s}{\eta_{s+1} T_{s+1}}+D\sqrt{\frac{2 \epsilon_s^2 }{\alpha_{s+1}  T_{s+1}}} +D\sqrt{\frac{\alpha_{s+1} L_2}{B_1^{s+1}} + \frac{\eta_{s+1}^2  L_2}{\alpha_{s+1} B_1^{s+1}}} +\eta_{s+1} \frac{L_F D^2}{2} \\
&\leq \frac{\epsilon_s}{8}+\frac{\epsilon_s}{8}+\frac{\epsilon_s}{8}+\frac{\epsilon_s}{8} =\epsilon_{s}/2 = \epsilon_{s+1}.
\end{align*}
We also know that
\begin{align*}
    \Gamma_{s+1} \leq&  \frac{\epsilon_s^2}{\alpha_{s+1} T_{s+1}} + \frac{\alpha_{s+1} L_2}{B_1^{s+1}} + \frac{L_2 \eta_{s+1}^2}{\alpha_{s+1} B_1^{s+1}}
    \leq  \frac{\epsilon_s^2}{12}+\frac{\epsilon_s^2}{12}+\frac{\epsilon_s^2}{12} =  \epsilon_{s}^2/4
    =  \epsilon_{s+1}^2.
\end{align*}
So we prove that $\mathbb{E}\left[F\left(\x^{s}\right)-F_{\star}\right] \leq \left(\frac{1}{2}\right)^{s-1}$ with $\eta_{s} = \alpha_{s} \leq \frac{\epsilon_s}{2L_F D^2},B_1^s \geq (64 L_2  D^2+3L_2)\frac{\eta_s}{\epsilon_s^2} $, $T_s \geq \frac{128D^2+12+4\Delta_F+3\Gamma_0+32D^2\Gamma_0}{\eta_s}$.
 
This condition can be satisfied by setting that $\eta_s =\alpha_s = \Theta (\epsilon_s)$, $B_1^s = \Omega (\epsilon_s^{-1})$ and $T_s = \Omega(\epsilon_s^{-1})$ [Large batch version]. This can also be achieved by setting that $\eta_s = \alpha_s = \Theta (\epsilon_s^2)$, $B_1^s = \Omega (1)$ and $T_s = \Omega(\epsilon_s^{-2})$  [Constant Batch]. To ensure $\mathbb{E}\left[F\left(\x^{s}\right)-F_{\star}\right] \leq \epsilon$, set $S=\log_2 (\frac{2}{\epsilon})$, and the SFO rate is $\sum_{s=1}^S T^s B_1^s = \mathcal{O}\left(\sum_{s=1}^S 2^{(2s)}\right) = \mathcal{O}\left(\frac{1}{\epsilon^2}\right)$. 
\end{proof}

\section{Proof of Theorem~\ref{thm:4}}
We assume $F(\x)$ is $\lambda$-strongly convex function. Note that we have proved:
\begin{align*} 
F\left(\x_{t+1}\right)
\leq &F\left(\x_{t}\right)+\eta\left\langle \v_{t}, \z_{t}-\x_{t}\right\rangle+\eta\left\langle\nabla F\left(\x_{t}\right)-\v_t, \z_t-\x_t\right\rangle +\eta^{2} \frac{L_F}{2} \Norm{\z_t - \x_t}^{2}. 
\end{align*}    
Set $\eta \leq \frac{\lambda}{4L_F}$. Since we know that
\begin{align*}
    \eta\left\langle \v_{t}, \z_{t}-\x_{t}\right\rangle + \frac{\eta \lambda}{4} \Norm{\z_t - \x_t}^{2}  \leq \eta\left\langle \v_{t}, \x^{\star}-\x_{t}\right\rangle + \frac{\eta \lambda}{4} \Norm{\x^{\star} - \x_t}^{2}  + \frac{\eta \lambda D^2}{N},
\end{align*}
and  $\left\langle \nabla F(\x_t), \x^{\star} - \x_t \right\rangle \leq F_{\star} - F(\x_t) - \frac{\lambda}{2} \Norm{\x_t - \x^{\star}}^2$ for $\lambda$-strongly convex function, we have
\begin{align*} 
&F\left(\x_{t+1}\right)-F(\x_t) \\
\leq &  \eta \left\langle \nabla F(\x_t) - \v_t, \z_t - \x^{\star} \right\rangle + \eta \left\langle \nabla F(\x_t) , \x^{\star} - \x_t \right\rangle +  \frac{\eta\lambda}{4} \Norm{\x^{\star}-\x_t}^2 + \frac{\eta \lambda D^2}{N} - \frac{\eta \lambda}{8} \Norm{\z_t - \x_t}^2\\
\leq &  \eta (F_{\star} - F(\x_t)) - \frac{\eta \lambda}{4}\Norm{\x^{\star} - \x_t}^2 + \frac{\eta \lambda}{16} \Norm{\z_t - \x^{\star}}^2 + \frac{4\eta}{\lambda} \Norm{\nabla F(\x_t) - \v_t}^2+ \frac{\eta \lambda D^2}{N} - \frac{\eta \lambda}{8} \Norm{\z_t - \x_t}^2\\
\leq &  \eta (F_{\star} - F(\x_t))  + \frac{4\eta}{\lambda} \Norm{\nabla F(\x_t) - \v_t}^2+ \frac{\eta \lambda D^2}{N} -\frac{\eta \lambda}{24} \Norm{\z_t - \x_t}^2.
\end{align*}    
So we have:
$$F\left(\x_{t+1}\right)  - F_{\star} \leq (1-\eta) (F\left(\x_{t}\right) -F_{\star} ) +\frac{4\eta}{\lambda} \Norm{\nabla F(\x_t) - \v_t}^2 + \frac{\eta \lambda D^2}{N}.$$
Summing up and rearranging, we obtain: 
\begin{align*}
    \frac{1}{T}\sum_{i=1}^T (F\left(\x_{t}\right) -F_{\star} ) & \leq \frac{F\left(\x_1\right) - F_{\star}}{\eta T} + \frac{4}{\lambda T}\sum_{t=1}^T \left\|\nabla F\left(\x_t\right)-\v_t\right\|^2+\frac{ \lambda D^2}{N}
    % &\leq \frac{F\left(\x_1\right) - F_{\star}}{\eta T} + \frac{4}{\lambda}(\frac{L_3}{\alpha B_0 T}+ \frac{\alpha L_3}{B_1}+  \frac{\eta^2 L_3D^2} {\alpha B_1}) +\frac{ \lambda D^2}{N}
\end{align*}
Next, we denote $\x^s$ as the output for stage $s$. Then, we have:
\begin{align*}
\mathbb{E}\left[F\left(\x^{s}\right)\right]-F_{\star} \leq &\frac{F\left(\x_{s-1}\right) - F_{\star}}{\eta_s T_s} + \frac{8\Gamma_{s-1}}{\lambda \alpha_s T_s } + \frac{4 \alpha_s L_3}{\lambda B_1^s} + \frac{4 \eta^2_s  L_3 D^2}{\lambda \alpha_s B_1^s} +  \frac{\lambda D^2}{N},\\
    \Gamma_{s} \leq&  \frac{\Gamma_{s-1}}{\alpha_{s} T_{s}} + \frac{\alpha_{s} L_3}{B_1^{s}} + \frac{L_3 \eta_{s}^2 D^2}{\alpha_{s} B_1^{s}}.
\end{align*}
Set that $\epsilon_s  = (\frac{1}{2})^{s-1},    B_1^s \geq  \frac{ 80 \alpha_s L_3 (1+D^2)}{ \lambda \epsilon_s }, N \geq \frac{10\lambda D^2}{\epsilon_s}, T_s \geq \frac{(160+5\Delta_F)}{\eta_s}$, $B_0 =\max \{ \frac{\Gamma_0}{4\lambda},1\}$. We can guarantee that $\mathbb{E}\left[F\left(\x^{s}\right)-F_{\star}\right] \leq \epsilon_s$ and $\E\left[\Gamma_{s} \right] \leq \lambda \epsilon_s$.
We will use Induction to prove.

\begin{proof}
When $s=1$, we have: 
\begin{align*}
\Gamma_{1}  \leq & \frac{\Gamma_{0}}{\alpha_{1} T_{1} B_0} + \frac{\alpha_{1} L_3}{B_1^{1}} + \frac{L_3 \eta_{1}^2 D^2}{\alpha_{1} B_1^{1}}
\leq   \lambda \epsilon_1 = \lambda ,
\end{align*}
where $B_0$ is the batch size used in the first iteration of the first stage. Also, we have that:
\begin{align*}
 \mathbb{E}\left[F\left(\x^{1}\right)-F_{\star}\right] & \leq \frac{\Delta_F}{\eta_1 T_1}+\frac{8\Gamma_0}{\lambda \alpha_1 T_1 B_0} + \frac{4 \alpha_1 L_3}{\lambda B_1^1} + \frac{4 \eta^2_1  L_3 D^2}{\lambda \alpha_1 B_1^1 } +  \frac{\lambda D^2}{N}\leq \epsilon_1 =1 
\end{align*}
Assume $\mathbb{E}\left[\Gamma_s\right] \leq \lambda \epsilon_{s}$ and $\mathbb{E}\left[F\left(\x^{s}\right)-F_{\star}\right] \leq \epsilon_s$, we would prove that it holds for stage $s+1$.
\begin{align*}
    \Gamma_{s+1} \leq&  \frac{\Gamma_{s}}{\alpha_{s+1} T_{s+1}} + \frac{\alpha_{s+1} L_3}{B_1^{s+1}} + \frac{L_3 \eta_{s+1}^2 D^2}{\alpha_{s+1} B_1^{s+1}}\\
    \leq & \frac{\lambda \epsilon_s}{\alpha_{s+1} T_{s+1}} + \frac{\alpha_{s+1} L_3}{B_1^{s+1}} + \frac{L_3 \eta_{s+1}^2 D^2}{\alpha_{s+1} B_1^{s+1}}
    \leq  \frac{\lambda \epsilon_s}{2} = \lambda \epsilon_{s+1}
\end{align*}
Besides, we know that
\begin{align*}
\mathbb{E}\left[F\left(\x^{s+1}\right)-F_{\star} \right]
& \leq \frac{\epsilon_s}{\eta_{s+1} T_{s+1}} + \frac{8\Gamma_{s}}{\lambda \alpha_{s+1} T_{s+1}} +  \frac{4 \alpha_{s+1} L_3}{\lambda  B_1^{s+1}} + \frac{4 \eta^2_{s+1}  L_3 D^2}{\lambda \alpha_{s+1} B_1^{s+1}} +  \frac{\lambda D^2}{N}\\
& \leq  \frac{\epsilon_{s}}{2} = \epsilon_{s+1}
\end{align*}
So we have proved that $\mathbb{E}\left[F\left(\x^{s}\right)-F_{\star}\right] \leq \left(\frac{1}{2}\right)^{s-1}$ with $\epsilon_s  = (\frac{1}{2})^{s-1},    B_1^s \geq  \frac{ 80 \alpha_s L_3 (1+D^2)}{ \lambda \epsilon_s }, N \geq \frac{10\lambda D^2}{\epsilon_s}, T_s \geq \frac{(160+5\Delta_F)}{\eta_s}$, $B_0 =\max \{ \frac{\Gamma_0}{4\lambda},1\}$.
\end{proof}
This condition can be satisfied by setting that $\eta_s= \alpha_s = \Theta (\lambda)$, $T_s = \Omega(\lambda^{-1})$, $B_1^s = \Omega (\epsilon_s^{-1}), N=\Omega(\frac{\lambda}{\epsilon_s})$  (Large Batch) or by setting that $\eta_s = \alpha_s= \Theta (\lambda \epsilon_s)$, $T_s = \Omega(\lambda^{-1}\epsilon_s^{-1})$, $B_1^s = \Omega (1), N=\Omega(\frac{\lambda}{\epsilon_s})$  (Constant Batch). 

To ensure $\mathbb{E}\left[F\left(\x^{s}\right)-F_{\star}\right] \leq \epsilon$, set $S=\log_2 (\frac{2}{\epsilon})$, and the SFO rate is $\sum_{s=1}^S T_s B_1^s =  \mathcal{O}(1)\cdot \sum_{s=1}^S\frac{2^{s-1}}{\lambda}  =\mathcal{O}(\frac{1}{\lambda \epsilon})$. 

\newpage

\section{Proof of Theorem~\ref{thm:fsnc}}
First, we can bound the term $\E\left[\Norm{\v_t - \prod_{i=1}^K \nabla f_i(\u_t^{i-1})}^2\right]$ as follows.
%%%%%%%%%%%%%%%%%%%%%%%%%%   LEMMA  2    %%%%%%%%%%%%%%%%%%%%%%%%%%%%%%%%%%%%%%%%%%%%%%%%%%%%%
\begin{lemma}%\label{lem:3} The gradient estimator $\v_t$ enjoys the following guarantee:
\begin{align*}
&\frac{1}{T}\sum_{t=1}^T \E\left[\Norm{\v_t -\prod_{i=1}^K \nabla f_i(\u_t^{i-1})}^2\right] \leq  \frac{4K}{\alpha B_1 T} L_J^2 L_f^{2K-2}\sum_{t=1}^T\sum_{i=1}^K \Norm{\u_t^{i-1} - \u_{t+1}^{i-1}}^2.
\end{align*}
\end{lemma}	
\begin{proof}
According to the definition of $\v_t$, we know that:
\begin{align*}
&\E\left[\Norm{\v_t -\prod_{i=1}^K \nabla f_i(\u_t^{i-1})}^2\right]\\
=& \E\left[\left \| (1-\alpha) \left( \v_{t-1} - \prod_{i=1}^K \nabla f_i(\u_{t-1}^{i-1}) \right) + \alpha \left( \z_t - \prod_{i=1}^K \nabla f_i(\u_{t}^{i-1})\right) \right.\right.\\
& \left.\left. +(1-\alpha) \left( \prod_{i=1}^K \nabla f_i(\u_{t-1}^{i-1}) -\prod_{i=1}^K \nabla f_i(\u_{t}^{i-1})   -  \frac{1}{B_1}\sum_{j=1}^{B_1}\prod_{i=1}^K \nabla f_i(\u_{t-1}^{i-1};\xi_t^{i,j}) + \frac{1}{B_1}\sum_{j=1}^{B_1}\prod_{i=1}^K \nabla f_i(\u_{t}^{i-1};\xi_t^{i,j})\right) \right\|^2\right] \\
= & (1-\alpha)^2 \E\left[\Norm{ \v_{t-1} - \prod_{i=1}^K \nabla f_i(\u_{t-1}^{i-1})}^2\right] +  \E \left[ \left\| \alpha \left(\z_t - \prod_{i=1}^K \nabla f_i(\u_{t}^{i-1}) \right) \right.\right. \\
& \left.\left. +(1-\alpha)\left(\frac{1}{B_1}\sum_{j=1}^{B_1}  \prod_{i=1}^K \nabla f_i(\u_{t-1}^{i-1};\xi_t^{i,j}) -  \frac{1}{B_1}\sum_{j=1}^{B_1}  \prod_{i=1}^K \nabla f_i(\u_{t}^{i-1};\xi_t^{i,j}) -  \prod_{i=1}^K \nabla f_i(\u_{t-1}^{i-1}) + \prod_{i=1}^K \nabla f_i(\u_{t}^{i-1})\right) \right\|^2 \right] \\
\leq & (1-\alpha) \E\left[\Norm{ \v_{t-1} - \prod_{i=1}^K \nabla f_i(\u_{t-1}^{i-1})}^2\right] + 2\alpha^2\E\left[\Norm{\z_t-\prod_{i=1}^K \nabla f_i(\u_{t}^{i-1})}^2\right]\\
& +2(1-\alpha)^2\E\left[\Norm{\frac{1}{B_1}\sum_{j=1}^{B_1}  \prod_{i=1}^K \nabla f_i(\u_{t-1}^{i-1};\xi_t^{i,j}) -  \frac{1}{B_1}\sum_{j=1}^{B_1}  \prod_{i=1}^K \nabla f_i(\u_{t}^{i-1};\xi_t^{i,j}) -  \prod_{i=1}^K \nabla f_i(\u_{t-1}^{i-1}) + \prod_{i=1}^K \nabla f_i(\u_{t}^{i-1})}^2\right]\\
\leq & (1-\alpha) \E\left[\Norm{ \v_{t-1} - \prod_{i=1}^K \nabla f_i(\u_{t-1}^{i-1})}^2\right] + 2\alpha^2\E\left[\Norm{\z_t-\prod_{i=1}^K \nabla f_i(\u_{t}^{i-1})}^2\right]\\
& +\frac{2K}{B_1} \LL_J^2L_f^{2K-2} \sum_{i=1}^{K} \E\left[\Norm{\u_t^{i-1} - \u_{t-1}^{i-1}}^2\right]\\
\leq & (1-\alpha) \E\left[\Norm{ \v_{t-1} - \prod_{i=1}^K \nabla f_i(\u_{t-1}^{i-1})}^2\right] + 2\alpha^2\frac{K}{B_1} L_J^2 L_f^{2K-2}\sum_{i=1}^K \Norm{\u_\tau^{i-1} - \u_t^{i-1}}^2\\
& +\frac{2K}{B_1} \LL_J^2L_f^{2K-2} \sum_{i=1}^{K} \E\left[\Norm{\u_t^{i-1} - \u_{t-1}^{i-1}}^2\right],
\end{align*}
where the last inequality is due to equation~(\ref{diff}) and the fact that 
\begin{align*}
    &\E\left[\Norm{\z_t - \prod_{i=1}^K \nabla f_i(\u_{t}^{i-1})}^2\right]\\
    \leq &\E\left[\Norm{\frac{1}{B_1} \sum_{j=1}^{B_1} \left[\prod_{i=1}^K \nabla f_i(\u_t^{i-1};\xi_t^{i,j})-\prod_{i=1}^K \nabla f_i(\u_{\tau}^{i-1};\xi_t^{i,j})\right] +\prod_{i=1}^K\nabla f_i(\u_\tau^{i-1}) - \prod_{i=1}^K \nabla f_i(\u_{t}^{i-1})}^2\right]\\
    = & \frac{1}{B_1^2}\sum_{j=1}^{B_1}\E\left[\Norm{\prod_{i=1}^K \nabla f_i(\u_{t}^{i-1};\xi_t^{i,j}) - \prod_{i=1}^K \nabla f_i(\u_{\tau}^{i-1};\xi_t^{i,j})}^2\right]\\
    =&\frac{1}{B_1^2}\sum_{j=1}^{B_1}\E\left[\left\|\prod_{i=1}^K \nabla f_i(\u_{t}^{i-1};\xi_t^{i,j}) - \nabla f_1(\u_{\tau}^0;\xi_t^{1,j}) \prod_{i=2}^K \nabla  f_i(\u_{t}^{i-1};\xi_t^{i,j}) \right. \right. \\
    & \quad \left.\left. + \nabla f_1(\u_{\tau}^0;\xi_t^{1,j}) \prod_{i=2}^K \nabla  f_i(\u_{t}^{i-1};\xi_t^{i,j}) - \nabla f_1(\u_{\tau}^0;\xi_t^{1,j}) \nabla f_2(\u_{\tau}^1;\xi_t^{2,j}) \prod_{i=3}^K \nabla  f_i(\u_{t}^{i-1};\xi_t^{i,j}) \right. \right.\\ 
    &  \qquad \cdots \\
    & \quad \left.\left. +  \left(\prod_{i=1}^{K-1} \nabla f_i(\u_{\tau}^{i-1};\xi_t^{i,j})\right) \nabla f_K(\u_{t}^{K-1};\xi_t^{i,j}) - \prod_{i=1}^K \nabla f_i(\u_{\tau}^{i-1};\xi_t^{i,j}) \right\|^2\right] \\
   \leq &  \frac{K}{B_1^2}\sum_{j=1}^{B_1} L_J^2 L_f^{2K-2}\sum_{i=1}^K \Norm{\u_\tau^{i-1} - \u_t^{i-1}}^2 \\
   = & \frac{K}{B_1} L_J^2 L_f^{2K-2}\sum_{i=1}^K \Norm{\u_\tau^{i-1} - \u_t^{i-1}}^2.
\end{align*}
Summing up over $t$ and rearranging, set $\alpha =1/I$, we have:
\begin{align*}
&\frac{1}{T}\sum_{t=1}^T \E\left[\Norm{\v_t -\prod_{i=1}^K \nabla f_i(\u_t^{i-1})}^2\right] \\
\leq & 2\alpha\frac{K}{B_1 T} L_J^2 L_f^{2K-2}\sum_{i=1}^K\sum_{t=1}^T \Norm{\u_\tau^{i-1} - \u_{t+1}^{i-1}}^2 +\frac{2K}{\alpha B_1 T} \LL_J^2L_f^{2K-2} \sum_{i=1}^{K}\sum_{t=1}^T \E\left[\Norm{\u_{t+1}^{i-1} - \u_{t}^{i-1}}^2\right]\\
\leq & 2\alpha\frac{K}{B_1 T} L_J^2 L_f^{2K-2}\sum_{i=1}^K\sum_{t=1}^T \Norm{\sum_{j=\tau}^{t}\u_{j}^{i-1} - \u_{j+1}^{i-1}}^2 +\frac{2K}{\alpha B_1 T} \LL_J^2L_f^{2K-2} \sum_{i=1}^{K}\sum_{t=1}^T \E\left[\Norm{\u_{t+1}^{i-1} - \u_{t}^{i-1}}^2\right]\\
\leq & 2\alpha\frac{KI^2}{B_1 T} L_J^2 L_f^{2K-2}\sum_{i=1}^K\sum_{t=1}^T \Norm{\u_t^{i-1} - \u_{t+1}^{i-1}}^2 +\frac{2K}{\alpha B_1 T} \LL_J^2L_f^{2K-2} \sum_{i=1}^{K}\sum_{t=1}^T \E\left[\Norm{\u_{t+1}^{i-1} - \u_{t}^{i-1}}^2\right]\\
\leq & \frac{4K}{\alpha B_1 T} L_J^2 L_f^{2K-2}\sum_{i=1}^K\sum_{t=1}^T \Norm{\u_t^{i-1} - \u_{t+1}^{i-1}}^2.
\end{align*}
\end{proof}
Next, we bound the estimation error of the function value estimator in the following lemma.
\begin{lemma}%\label{lem:4} The inner function estimator $\u_t$ ensures that:
\begin{align*}
	\frac{1}{T}\sum_{t=1}^T \sum_{i=1}^K \E\left[\Norm{\u_t^i -  f_i(\u_t^{i-1})}^2\right] \leq  \frac{4\LL_f^2}{\alpha B_1 T}  \sum_{t=1}^T \sum_{i=1}^K\E\left[\Norm{\u_{t+1}^{i-1} - \u_{t}^{i-1}}^2\right].
\end{align*}
\end{lemma}	
\begin{proof}
Following the very similar analysis, we have:
\begin{align*}
& \E\left[\Norm{ \u_t^i - f_i(\u_t^{i-1})}^2\right]\\
\leq & (1-\alpha)^2 \E \Norm{\u_{t-1}^i - f_i(\u_{t-1}^{i-1})}^2 + {2\alpha^2} \E\left[\Norm{\h_t^i - f_i(\u_{t}^{i-1})}^2\right] \\
& \qquad  +2(1-\alpha)^2\E\left[\left\|\frac{1}{B_1}\sum_{j=1}^{B_1}\left(f_i(\u_{t-1}^{i-1};\xi_t^{i,j}) - f_i(\u_t^{i-1};\xi_t^{i,j} )- f_i(\u_{t-1}^{i-1}) + f_i(\u_t^{i-1})\right)\right\|^2\right]\\
\leq&  (1-\alpha) \E \Norm{\u_{t-1}^i - f_i(\u_{t-1}^{i-1})}^2 + \frac{2\alpha^2 L_f^2}{B_1}\left\|\u_{\tau}^{i-1} - \u_t^{i-1}\right\|^2 + \frac{2\LL_f^2}{B_1} \left\|\u_{t-1}^{i-1} - \u_t^{i-1}\right\|^2. 
\end{align*}
This leads to the fact that:
\begin{align*}
	&\frac{1}{T}\sum_{t=1}^T \sum_{i=1}^K \E\left[\Norm{\u_t^i -  f_i(\u_t^{i-1})}^2\right]  \\
 \leq& \frac{2\alpha L_f^2}{B_1 T}\sum_{t=1}^T \sum_{i=1}^K\left\|\u_{\tau}^{i-1} - \u_{t+1}^{i-1}\right\|^2 + \frac{2\LL_f^2}{\alpha B_1 T}\sum_{t=1}^T \sum_{i=1}^K\left\|\u_{t+1}^{i-1} - \u_t^{i-1}\right\|^2\\
 \leq& \frac{2\alpha L_f^2 I^2}{B_1 T}\sum_{t=1}^T \sum_{i=1}^K\left\|\u_{t+1}^{i-1} - \u_t^{i-1}\right\|^2 + \frac{2\LL_f^2}{\alpha B_1 T}\sum_{t=1}^T \sum_{i=1}^K\left\|\u_{t+1}^{i-1} - \u_t^{i-1}\right\|^2\\
 \leq &\frac{4\LL_f^2}{\alpha B_1 T}\sum_{t=1}^T \sum_{i=1}^K\left\|\u_{t+1}^{i-1} - \u_t^{i-1}\right\|^2.
\end{align*}
By summing up and rearranging, we finish the proof of this lemma.
\end{proof}
Then, we bound the term $\sum_{i=1}^K\E\left[\Norm{\u_{t+1}^{i-1} - \u_{t}^{i-1}}^2\right]$.
\begin{lemma} We can obtain the following guarantee.
\begin{align*}
&\sum_{t=1}^T\sum_{i=1}^K  \E\left[\Norm{\u_{t+1}^{i-1} - \u_{t}^{i-1}}^2 \right] \\ 
 \leq&  \left(\sum_{i=1}^{K}\left(4\LL_f^2\right)^{i-1}\right) \left(\sum_{t=1}^T\E \left[\eta^2 \Norm{\z_t - \x_t}^2\right] +  2\alpha^2K \sum_{t=1}^T\sum_{i=1}^{K} \E\left[\Norm{\u_t^i - f_i(\u_t^{i-1})}^2\right]  \right).
\end{align*}
\end{lemma}
\begin{proof}
(1) For the first level, i.e., $i=1$, we have: 
    \begin{align*}
	    \sum_{t=1}^T \E\left[\Norm{\u_{t+1}^{i-1} - \u_{t}^{i-1}}^2\right] = \sum_{t=1}^T\E\left[\Norm{\x_{t+1} - \x_{t}}^2\right] = \sum_{t=1}^T\E \left[\eta^2 \Norm{\z_t - \x_t}^2\right].
	\end{align*}
(2) For other levels, i.e., $2\leq i\leq K$, we have:
	\begin{align*}
		& \sum_{t=1}^T\E\left[\Norm{\u_{t+1}^{i-1} - \u_{t}^{i-1}}^2\right]\\
		=&\sum_{t=1}^T\E\left[\left\|\alpha \left(\h_{t+1}^{i-1} - \u_{t}^{i-1}\right) + \frac{(1-\alpha)}{B_1}\sum_{j=1}^{B_1}(f_{i-1}(\u_{t+1}^{i-2};\xi_{t+1}^{i-1,j}) - f_{i-1}(\u_{t}^{i-2};\xi_{t+1}^{i-1,j})) \right\|^2 \right]\\
\leq& 2\alpha^2\sum_{t=1}^T\E\left[\Norm{f_{i-1}(\u_{t}^{i-2}) - \u_{t}^{i-1} + \h_{t+1}^{i-1}-f_{i-1}(\u_{t}^{i-2})}^2\right]+ 2\LL_f^2\sum_{t=1}^T\E\left[\Norm{\u_{t+1}^{i-2} - \u_{t}^{i-2}}^2 \right]\\
	\leq&  2\alpha^2\sum_{t=1}^T \E\left[\Norm{f_{i-1}(\u_{t}^{i-2}) - \u_{t}^{i-1}}^2\right] + \frac{2\alpha^2 L_f^2}{B_1}\sum_{t=1}^T\E\left[\Norm{\u_{t+1}^{i-2} - \u_{\tau}^{i-2}}^2 \right] \\
 &\quad + 2\LL_f^2\sum_{t=1}^T\E\left[\Norm{\u_{t+1}^{i-2} - \u_{t}^{i-2}}^2 \right]\\
	\leq&  2\alpha^2\sum_{t=1}^T \E\left[\Norm{f_{i-1}(\u_{t}^{i-2}) - \u_{t}^{i-1}}^2\right]  + 4\LL_f^2\sum_{t=1}^T\E\left[\Norm{\u_{t+1}^{i-2} - \u_{t}^{i-2}}^2 \right].
	\end{align*} 
Denote that $\Upsilon^i =\sum_{t=1}^T\E\left[\Norm{\u_t^i - f_i(\u_t^{i-1})}^2\right]$ and  $Q^{i} = \sum_{t=1}^T\E\left[\Norm{\u_{t+1}^{i-1} - \u_{t}^{i-1}}^2\right]$, we have $Q^{i} \leq 4\LL_f^2 Q^{i-1} + 2\alpha^2 \Upsilon^{i-1} $ for $i\geq 2$. Then we can get:
\begin{align*}
Q^{1} &\leq \E \left[\eta^2 \Norm{\z_t - \x_t}^2\right] \\
Q^{2} &\leq \left(4\LL_f^2\right) \E \left[\eta^2 \Norm{\z_t - \x_t}^2\right] +2\alpha^2 \Upsilon_t^{1} \\
Q^{3} &\leq \left(4\LL_f^2\right)^2 \E \left[\eta^2 \Norm{\z_t - \x_t}^2\right]  + 2\alpha^2\left( 4\LL_f^2  \Upsilon_t^{1}+ \Upsilon_t^{2}\right) \\
& \cdots\\
Q^{i} &\leq \left(4\LL_f^2\right)^{i-1} \E \left[\eta^2 \Norm{\z_t - \x_t}^2\right] + 2\alpha^2 \sum_{j=1}^{i-1} \left(4\LL_f^2\right)^{i-1-j} \Upsilon_t^{j}\\
&\leq \left(4\LL_f^2\right)^{i-1} \E \left[\eta^2 \Norm{\z_t - \x_t}^2\right] + 2\alpha^2 \sum_{j=1}^{K} \sum_{l=1}^{K} \left(4L_f^2\right)^{K-l} \Upsilon_t^{j}.
\end{align*}
When summing up, we have:
\begin{align*}
    &\sum_{t=1}^T\sum_{i=1}^K \E\left[\Norm{\u_{t+1}^{i-1} - \u_{t}^{i-1}}^2 \right] 
    = \sum_{t=1}^T\sum_{i=1}^{K} Q^{i}_t \\
    \leq &\sum_{t=1}^T\sum_{i=1}^{K}\left(4\LL_f^2\right)^{i-1} \E \left[\eta^2 \Norm{\z_t - \x_t}^2\right] + 2\alpha^2 K \sum_{t=1}^T\sum_{j=1}^{K} \sum_{l=1}^{K} \left(4L_f^2\right)^{K-l} \Upsilon_t^{j}  \\
    %\leq &\left(\sum_{i=1}^{K}\left(2\LL_f^2\right)^{i-1}\right) \left(\E \left[\eta^2 \Norm{\z_t - \x_t}^2\right] + \frac{2\alpha^2\sigma^2 K}{B_1}   + 2\alpha^2K \sum_{i=1}^{K} \Upsilon_t^{j}  \right)\\
    \leq &\left(\sum_{i=1}^{K}\left(4\LL_f^2\right)^{i-1}\right) \left(\sum_{t=1}^T\E \left[\eta^2 \Norm{\z_t - \x_t}^2\right] + 2\alpha^2K \sum_{t=1}^T\sum_{i=1}^{K} \E\left[\Norm{\u_t^i - f_i(\u_t^{i-1})}^2\right]  \right)
\end{align*}
Thus, we complete the proof of this lemma.
\end{proof}
Next, we bound the error of gradient estimator as follows.
%%%%%%%%%%%%%%%%%%%%%%%%%%%%%%   LEMMA  7    %%%%%%%%%%%%%%%%%%%%%%%%%%%%%%%%%%%%%%%%%%%%%%%%%
Denote that the constants as $L_1 = \left(4K \LL_J^2 L_f^{2K-2} + 8KL_f^2\LL_f^2\right)\left(\sum_{i=1}^{K}\left(2\LL_f^2\right)^{i-1}\right)$, and setting $\alpha \leq \frac{B_1 L_F^2}{2L_1}$, we can deduce:
\begin{align*}
 & \frac{1}{T}\sum_{t=1}^T\E\left[\Norm{\v_t - \prod_{i=1}^K \nabla f_i(\u_t^{i-1}) }^2\right] + \frac{2KL_F^2}{T}\sum_{t=1}^T \sum_{i=1}^{K-1} \E\left[\Norm{\u_t^i - f_i(\u_t^{i-1})}^2\right]\\
 \leq &\frac{4K\LL_J^2L_f^{2K-2}+8KL_F^2\LL_f^2}{\alpha B_1 T}\sum_{t=1}^T \sum_{i=1}^{K} \E\left[\Norm{\u_{t+1}^{i-1} - \u_{t}^{i-1}}^2\right]\\
\leq & \frac{L_1}{\alpha B_1} \left(\eta^2 D^2 + 2\alpha^2K \frac{1}{T}\sum_{t=1}^T \sum_{i=1}^{K} \E\left[\Norm{\u_t^i - f_i(\u_t^{i-1})}^2\right]  \right) \\
 \leq &  \frac{L_1\eta^2D^2}{\alpha B_1} + \frac{KL_F^2}{T}\sum_{t=1}^T \sum_{i=1}^{K}\E\left[\Norm{\u_t^i - f_i(\u_t^{i-1})}^2\right]. 
\end{align*}
Then, we can obtain:
\begin{align*}
&\frac{1}{T}\sum_{t=1}^T \E\left[\Norm{\v_t - \nabla F(\x_t)}^2\right] \\ \leq & \frac{2}{T}\sum_{t=1}^T\E\left[\Norm{\v_t - \prod_{i=1}^K\nabla f_i(\u_t^{i-1})}^2\right]  + \frac{2KL_F^2}{T}\sum_{t=1}^T \sum_{i=1}^{K-1} \E\left[\Norm{\u_t^i - f_i(\u_t^{i-1})}^2\right]\\
\leq &\frac{2L_1\eta^2D^2}{\alpha B_1}
\end{align*}
By setting $I=\frac{m}{B_1}$ and $\alpha = 1/I$, We have already shown that
\begin{align*} 
&\E\left[\F(\x_\tau)\right]=\E\left[ \frac{1}{T}\sum_{t=1}^{T}  \F\left(\x_{t}\right)\right] \\ \leq &\frac{\E\left[F\left(\x_{1}\right)-F(\x_{T+1})\right]}{\eta T}+D\cdot\E\left[\frac{1}{T} \sum_{t=1}^T \left\|\nabla F\left(\x_t\right)-\v_t\right\|\right]+\eta \frac{L_F}{2} D^{2}\\
\leq &\frac{\Delta_F}{\eta T}+D\sqrt{\E\left[\frac{1}{T} \sum_{t=1}^T \left\|\nabla F\left(\x_t\right)-\v_t\right\|^2\right]}+\eta \frac{L_F}{2} D^{2}\\
\leq &\frac{\Delta_F}{\eta T}+D\sqrt{\frac{2L_1\eta^2D^2}{\alpha B_1}}+\eta \frac{L_F}{2} D^{2}\\
\leq &\frac{\Delta_F}{\eta T}+D\sqrt{\frac{2L_1\eta^2 m D^2}{B_1^2}}+\eta \frac{L_F}{2} D^{2}.
\end{align*} 
By setting $\eta =1/\sqrt{T}$, $B_1 = \sqrt{m}$, we can ensure $\E\left[\F(\x_\tau)\right] \leq \mathcal{O}\left(1/\sqrt{T}\right)$. 

\noindent By setting $\eta = \frac{1}{m^{1/4}T^{1/2}}, B_1 =  1$, we can ensure $\E\left[\F(\x_\tau)\right] \leq \mathcal{O}\left(\frac{m^{1/4}}{{T}^{1/2}}\right)$.
\newpage

\section{Proof of Theorem~\ref{thm:fsnc2}}
In the previous analysis, we simply reduced ${\eta^2}\Norm{\z_t - \x_t}^{2} \leq \eta^2 D^2$. To obtain the rate for gradient mapping, we keep this term. 
By setting $\alpha \leq \frac{B_1 L_F^2}{2L_1}$, we can deduce that:
\begin{align*}
 & \frac{1}{T}\sum_{t=1}^T\E\left[\Norm{\v_t - \prod_{i=1}^K \nabla f_i(\u_t^{i-1}) }^2\right] + \frac{2KL_F^2}{T}\sum_{t=1}^T \sum_{i=1}^{K-1} \E\left[\Norm{\u_t^i - f_i(\u_t^{i-1})}^2\right]\\
 %\leq &\frac{4K\LL_J^2L_f^{2K-2}+8KL_F^2\LL_f^2}{\alpha B_1 T}\sum_{t=1}^T \sum_{i=1}^{K} \E\left[\Norm{\u_{t+1}^{i-1} - \u_{t}^{i-1}}^2\right]\\
\leq & \frac{L_1}{\alpha B_1} \left(\eta^2 \frac{1}{T}\sum_{t=1}^T\Norm{\z_t - \x_t}^{2} + 2\alpha^2K \frac{1}{T}\sum_{t=1}^T \sum_{i=1}^{K} \E\left[\Norm{\u_t^i - f_i(\u_t^{i-1})}^2\right]  \right) \\
 \leq &  \frac{L_1\eta^2}{\alpha B_1 T}\sum_{t=1}^T\Norm{\z_t - \x_t}^{2} + \frac{KL_F^2}{T}\sum_{t=1}^T \sum_{i=1}^{K}\E\left[\Norm{\u_t^i - f_i(\u_t^{i-1})}^2\right]. 
\end{align*}
Thus, we have
\begin{align*}
&\frac{1}{T}\sum_{t=1}^T \E\left[\Norm{\v_t - \nabla F(\x_t)}^2\right] \\ \leq & \frac{2}{T}\sum_{t=1}^T\E\left[\Norm{\v_t - \prod_{i=1}^K\nabla f_i(\u_t^{i-1})}^2\right]  + \frac{2KL_F^2}{T}\sum_{t=1}^T \sum_{i=1}^{K-1} \E\left[\Norm{\u_t^i - f_i(\u_t^{i-1})}^2\right]\\
\leq &  \frac{2L_1\eta^2 }{\alpha B_1}\frac{1}{T}\sum_{t=1}^T\Norm{\z_t - \x_t}^{2}.
\end{align*}
By setting $\alpha = B_1/m$, $\eta \leq \sqrt{\frac{\beta^2 \alpha B_1}{40 L_1}}=\frac{\beta B_1}{\sqrt{40 m L_1}}$, we can prove that  
\begin{align*}
  &\quad \frac{1}{T}\sum_{t=1}^{T}\left\|\mathcal{G}(\x_t, \beta)\right\|^{2}   \\
  & \leq \frac{4\beta \Delta_F}{\eta T}  + \frac{8\beta^2 D^2 }{N} + \frac{10}{T}\sum_{t=1}^{T}\E\left[\Norm{\nabla F(\x_t) - \v_t }^2\right] - \frac{\beta^2}{2} \frac{1}{T}\sum_{t=1}^{T}\|\z_t-\x_t\|^{2} \\
  & \leq \frac{4\beta \Delta_F}{\eta T}  + \frac{8\beta^2 D^2 }{N} + \left(  \frac{20 \eta^2 L_1} {\alpha B_1} -\frac{\beta^2}{2}\right)\frac{1}{T}  \sum_{t=1}^{T}\Norm{\z_t - \x_t}^{2} 
   \leq \frac{4\beta \Delta_F}{\eta T}  + \frac{8\beta^2 D^2 }{N}
\end{align*}  
Set $N=\Omega(\epsilon^{-1})$. We ensure $\E\left[\Norm{\G(\x_t,\beta)}^2\right]\leq \epsilon$ by choosing $\alpha = \Theta(1/\sqrt{m})$, $\eta = \Theta(1)$, $T=\Omega(\epsilon^{-1})$, $ B_1=\Omega(\sqrt{m})$, or setting $\alpha =\Theta\left( 1/m\right), \eta =\Theta\left(1/\sqrt{m}\right), T = \Omega\left( \sqrt{m}\epsilon^{-1} \right),  B_1 =  \Omega\left(1 \right)$.

\section{Proof of Theorem~\ref{thm:3+}}
Note that in the previous analysis, we have proved that
\begin{align*}
 \frac{1}{T}\sum_{t=1}^T \mathbb{E}\left[F\left(\x_{t}\right)\right]-F_{\star} 
 %\leq &\frac{\left(\mathbb{E}\left[F\left(\x_{1}\right)\right]-F_{\star}\right)}{\eta T}+\frac{ D}{T} \sum_{t=1}^T \mathbb{E}\left\|\nabla F\left(\x_{t}\right)-\v_{t}\right\|+\eta \frac{L_F D^2}{2} \\
  \leq & \frac{\left(\mathbb{E}\left[F\left(\x_{1}\right)\right]-F_{\star}\right)}{\eta T}+D\sqrt{\frac{ 1}{T} \sum_{t=1}^T \mathbb{E}\left\|\nabla F\left(\x_{t}\right)-\v_{t}\right\|^2}+\eta \frac{L_F D^2}{2} \\
  \leq & \frac{\left(\mathbb{E}\left[F\left(\x_{1}\right)\right]-F_{\star}\right)}{\eta T}+D\sqrt{\frac{2L_1 D^2 \eta^2  }{\alpha B_1}} +\eta \frac{L_F D^2}{2}. 
\end{align*}
Next, we denote $\x^s$ as the output of Algorithm 3 for the stage $s$. Then, we have:
\begin{align*}
\mathbb{E}\left[F\left(\x^{s}\right)-F_{\star}\right] \leq \frac{\mathbb{E}\left[F\left(\x^{s-1}\right)-F_{\star}\right]}{\eta_s T_s}+D^2\sqrt{ \frac{2L_1 \eta_s^2 }{\alpha_s B_1^s}} +\eta_s \frac{L_F D^2}{2}. 
\end{align*}
Set $\epsilon_s = \left(\frac{1}{2}\right)^{s-1} $, $\eta_{s} \leq \frac{\epsilon_s}{3L_F D^2},B_1^s \geq D^2\sqrt{18L_1 n}\frac{\eta_s}{\epsilon_s} $, $T_s \geq \frac{6+3\Delta_F}{\eta_s}$. We can guarantee that $\mathbb{E}\left[F\left(\x^{s}\right)-F_{\star}\right] \leq \epsilon_s$.
We will use Induction to give the proof:
\begin{proof}
When $s=1$, we have:
\begin{align*}
\mathbb{E}\left[F\left(\x^{1}\right)-F_{\star}\right] & \leq \frac{\Delta_F}{\eta_1 T_1}+D^2\sqrt{ \frac{2L_1 \eta_1^2  }{\alpha_1 B_1^1}} +\eta_1 \frac{L_F D^2}{2} \leq \epsilon_1 =1 
\end{align*}
Assume that $\mathbb{E}\left[F\left(\x^{s}\right)-F_{\star}\right] \leq \epsilon_s$, we would prove that it holds for stage $s+1$ as well.
\begin{align*}
\mathbb{E}\left[F\left(\x^{s+1}\right)-F_{\star}\right] 
&\leq \frac{\epsilon_s}{\eta_{s+1} T_{s+1}}+D^2\sqrt{ \frac{ 2L_1\eta_{s+1}^2 }{\alpha_{s+1} B_1^{s+1}}} +\eta_{s+1} \frac{L_F D^2}{2} \leq \epsilon_{s}/2 = \epsilon_{s+1}
\end{align*}
So we prove  $\mathbb{E}\left[F\left(\x^{s}\right)-F_{\star}\right] \leq \left(\frac{1}{2}\right)^{s-1}$ with $\eta_{s} \leq \frac{\epsilon_s}{3L_F D^2},B_1^s \geq D^2\sqrt{18L_1 m}\frac{\eta_s}{\epsilon_s} $, $T_s \geq \frac{6+3\Delta_F}{\eta_s}$.
 
This condition can be satisfied by setting that $\eta_s =\Theta(\epsilon_s)$, $B_1^s = \Omega (\sqrt{m})$ and $T_s = \Omega(\epsilon_s^{-1})$ or $\eta_s = \Theta (\epsilon_s /\sqrt{m})$, $B_1^s = \Omega (1)$ and $T_s = \Omega(\sqrt{m}\epsilon_s^{-1})$.  
To ensure $\mathbb{E}\left[F\left(\x^{s}\right)-F_{\star}\right] \leq \epsilon$, set $S=\log_2 (\frac{2}{\epsilon})$, and the SFO rate is $\sum_{s=1}^S T^s B_1^s = \mathcal{O}\left(\sqrt{m}\sum_{s=1}^S 2^{s}\right) = \mathcal{O}\left(\frac{\sqrt{m}}{\epsilon}\right)$. 
\end{proof}

\section{Proof of Theorem~\ref{thm:4+}}
When $F(\x)$ is $\lambda$-strongly convex and set $\eta \leq \frac{\lambda}{4L_F}$, we have:
$$F\left(\x_{t+1}\right)  - F_{\star} \leq (1-\eta) (F\left(\x_{t}\right) -F_{\star} ) +\frac{4\eta}{\lambda} \Norm{\nabla F(\x_t) - \v_t}^2 + \frac{\eta \lambda D^2}{N}- \frac{\eta \lambda}{24} \Norm{\z_t - \x_t}^2.$$
Summing up and setting $\eta_s =\eta \leq \frac{\lambda B_1^s}{\sqrt{192L_1 m}}$, $T=\frac{2}{\eta}=\frac{2\sqrt{192L_1 m}}{\lambda B_1^s}$, we obtain
\begin{align*}
    &\frac{1}{T}\sum_{i=1}^T (F\left(\x_{t}\right) -F_{\star} ) \\ 
    \leq& \frac{F\left(\x_1\right) - F_{\star}}{\eta T} + \frac{4}{\lambda T}\sum_{t=1}^T \left\|\nabla F\left(\x_t\right)-\v_t\right\|^2+\frac{ \lambda D^2}{N} - \frac{\lambda}{24T}\sum_{i=1}^T \Norm{\z_t - \x_t}^2\\
    \leq& \frac{F\left(\x_1\right) - F_{\star}}{\eta T} +\frac{ \lambda D^2}{N} +\left(\frac{8L_1\eta^2}{\lambda \alpha B_1}- \frac{\lambda}{24}\right)\frac{1}{T}\sum_{i=1}^T \Norm{\z_t - \x_t}^2\\
    \leq& \frac{F\left(\x_1\right) - F_{\star}}{2} +\frac{ \lambda D^2}{N} 
\end{align*}
Next, we denote $\x^s$ as the output for stage $s$. Then, we have:
\begin{align*}
\mathbb{E}\left[F\left(\x^{s}\right)\right]-F_{\star} \leq \frac{F\left(\x_{s-1}\right) - F_{\star}}{2} +  \frac{\lambda D^2}{N} \leq \frac{F\left(\x_{1}\right) - F_{\star}}{2^{s-1}} +  \frac{2\lambda D^2}{N}
\end{align*}
To ensure $\mathbb{E}\left[F\left(\x^{S}\right)-F_{\star}\right] \leq \epsilon$, set that $S=\log_2 (\frac{2\Delta_F}{\epsilon})$ and $N=\Omega(\lambda/\epsilon)$. 
For constant batch size $B_1^s = \Omega(1)$, we choose $\eta_s=\Theta(\lambda/\sqrt{m})$, $\alpha_s=\Theta(1/m)$, and $T=\Omega(\sqrt{m}/\lambda)$. For large batch size $B_1^s = \Omega(\sqrt{m}/\lambda)$, we choose $\eta_s = \Theta(1)$, $\alpha_s=\Theta(1/(\lambda\sqrt{m}))$ and $T=\Omega(1)$.

Thus, the SFO complexity is $\sum_{s=1}^S T_s B_1^s =  \mathcal{O}(1)\cdot \sum_{s=1}^S\frac{\sqrt{m}}{\lambda}  =\mathcal{O}(\frac{\sqrt{m}}{\lambda }S) =\mathcal{O}(\frac{\sqrt{m}}{\lambda }\log(\frac{1}{\epsilon}))$ and the LMO rate is $\sum_{s=1}^S T_s  =\mathcal{O}(\frac{\sqrt{m}}{\lambda B_1^s }\log(\frac{1}{\epsilon}))$.

\vskip 1in
\bibliography{ref}

\end{document}